%% file: 0_main.tex
\documentclass[reqno, 12pt]{amsart}

\usepackage[colorlinks, linkcolor=blue, citecolor=red, urlcolor=cyan,pagebackref]{hyperref}

\allowdisplaybreaks

\usepackage{etex}
\usepackage{amsthm,amssymb,amscd,amsmath}
\usepackage{mathrsfs}
\usepackage{bbm}
\usepackage{graphicx}
\usepackage{epic,eepic}
\usepackage{tikz-cd}
\usepackage{stmaryrd}
\usepackage{mathtools}
\usepackage{stmaryrd}
\usepackage{diagbox}

\usepackage{ascmac}
\usepackage[normalem]{ulem}

\usepackage{accents}

\usepackage{here}
\usepackage{bm}
\usepackage[all]{xy}

\usepackage{xcolor} 
\definecolor{mygreen}{rgb}{0,0.7,0.3}
\definecolor{myblue}{rgb}{0,0.50,1.20}
\definecolor{myorange}{rgb}{1,0.5,0.1}

\definecolor{fillred}{rgb}{1,0.9,0.9}
\definecolor{fillgreen}{rgb}{0.9,1,0.9}

 \definecolor{refkey}{rgb}{0,0.7,0.3}
 \definecolor{labelkey}{rgb}{1,0,0}

\usepackage{tikz}
\usetikzlibrary{calc,decorations.markings}
\usetikzlibrary{arrows.meta}
\usetikzlibrary{positioning}
\usetikzlibrary{cd, intersections, calc, decorations.pathmorphing, arrows, decorations.pathreplacing}

\numberwithin{equation}{section}

\usepackage{version}

\usepackage{thmtools}
\usepackage{cleveref}

\declaretheorem[
  style=plain,
  name=Theorem,
  numberwithin=section
]{thm}

\declaretheorem[
  style=plain,
  name=Proposition,
  sibling=thm
]{prop}

\declaretheorem[
  style=plain,
  name=Corollary,
  sibling=thm
]{cor}

\declaretheorem[
  style=plain,
  name=Lemma,
  sibling=thm
]{lem}

\declaretheorem[
  style=plain,
  name=Theorem
]{introthm}

\declaretheorem[
  style=plain,
  name=Conjecture,
  sibling=introthm
]{introconj}

\declaretheorem[
  style=definition,
  name=Definition,
  sibling=thm
]{dfn}

\declaretheorem[
  style=definition,
  name=Conjecture,
  sibling=thm
]{conj}

\declaretheorem[
  style=definition,
  name=Notation,
  sibling=thm
]{conv}

\declaretheorem[
  style=remark,
  name=Remark,
  sibling=thm
]{rem}

\crefname{thm}{Theorem}{Theorems}
\crefname{prop}{Proposition}{Propositions}
\crefname{cor}{Corollary}{Corollaries}
\crefname{lem}{Lemma}{Lemmas}
\crefname{sublem}{Sublemma}{Sublemmas}

\crefname{introthm}{Theorem}{Theorems}
\crefname{introcor}{Corollary}{Corollaries}
\crefname{introconj}{Conjecture}{Conjectures}

\crefname{dfn}{Definition}{Definitions}
\crefname{defi}{Definition}{Definitions}
\crefname{ex}{Example}{Examples}
\crefname{claim}{Claim}{Claims}
\crefname{conj}{Conjecture}{Conjectures}
\crefname{conv}{Notation}{Notations}

\crefname{rmk}{Remark}{Remarks}
\crefname{rem}{Remark}{Remarks}
\crefname{prob}{Problem}{Problems}

\crefname{figure}{Figure}{Figures}
\crefname{section}{Section}{Sections}
\crefname{subsection}{Section}{Sections}
\crefname{appendix}{Appendix}{Appendices}

\newcommand{\bZ}{\mathbb{Z}}

\newcommand{\bR}{\mathbb{R}}
\newcommand{\bC}{\mathbb{C}}

\newcommand{\bM}{\mathbb{M}}

\newcommand{\bi}{{\boldsymbol{i}}}
\newcommand{\bp}{{\mathbf{p}}}
\newcommand{\bbp}{{\boldsymbol{p}}}
\newcommand{\br}{{\mathbf{r}}}
\newcommand{\by}{{\mathbf{y}}}

\newcommand{\cD}{\mathcal{D}}

\newcommand{\cF}{\mathcal{F}}

\newcommand{\cO}{\mathcal{O}}
\def\P{{\mathcal{P}}}

\newcommand{\cW}{\mathcal{W}}
\newcommand{\X}{\mathcal{X}}

\newcommand{\ve}{\varepsilon}
\newcommand{\dbra}[1]{[ #1 ]}
\newcommand{\ddbra}[1]{\llbracket #1 \rrbracket}

\newcommand{\bD}{{\accentset{\mbox{\bfseries .}}{\tri} }}

\newcommand{\inv}{\mathrm{inv}}

\newcommand{\Hom}{\mathrm{Hom}}
\newcommand{\tri}{\triangle}
\newcommand{\sgn}{\mathrm{sgn}}

\newcommand{\Teich}{Teichm\"uller}
\newcommand{\interior}{{\mathrm{int}}}

\newcommand{\op}{\mathrm{op}}

\newcommand{\opU}{\mathsf{U}}

\newcommand{\opP}{\mathsf{P}}

\newcommand{\opX}{\mathsf{X}}
\newcommand{\opY}{\mathsf{Y}}
\newcommand{\opA}{\mathsf{A}}
\newcommand{\opT}{\mathsf{T}}
\newcommand{\opS}{\mathsf{S}}

\newcommand{\oph}{\mathsf{h}}
\newcommand{\opE}{\mathsf{E}}
\newcommand{\opF}{\mathsf{F}}
\newcommand{\opG}{\mathsf{G}}
\newcommand{\opV}{\mathsf{V}}

\newcommand{\opK}{\mathsf{K}}

\newcommand{\bra}[1]{\langle #1 |}
\newcommand{\ket}[1]{| #1 \rangle}
\newcommand{\braket}[2]{\langle #1 | #2 \rangle}
\newcommand{\mb}[1]{\underline{#1}}
\newcommand{\mbb}[1]{\uuline{#1}}

\newcommand{\wt}{w}
\newcommand{\tdot}{\scalebox{0.4}{$\bullet$}}

\newcommand{\bbe}{\mathbbm{e}}

\DeclareMathOperator{\Frac}{\mathrm{Frac}}
\DeclareMathOperator{\Ad}{\mathrm{Ad}}

\allowdisplaybreaks

\makeatletter
\newcommand{\oset}[3][0ex]{%
  \mathrel{\mathop{#3}\limits^{
    \vbox to#1{\kern-2\ex@
    \hbox{$\scriptstyle#2$}\vss}}}}
\makeatother

\makeatletter
\newcommand{\osetnear}[3][0ex]{%
  \mathrel{\mathop{#3}\limits^{
    \vbox to#1{\kern-.3\ex@
    \hbox{$\scriptstyle#2$}\vss}}}}
\makeatother

\tikzset{pics/.cd,
handle/.style={code={
\draw (-0.72,0) to[bend left] (0.72,0);
\draw (-0.9,0.1) to[bend right] (0.9,0.1);
}}}
 
\tikzset{
  mid arrow/.style={postaction={decorate,decoration={
        markings,
        mark=at position .5 with {\arrow[#1]{stealth}}
      }}},
}

\tikzset{->-/.style 2 args={
	postaction={decorate},
	decoration={markings, mark=at position #1 with {\arrow[thick, #2]{>}}} 
    },
    ->-/.default={0.5}{}
}
\tikzset{-<-/.style 2 args={
	postaction={decorate},
	decoration={markings, mark=at position #1 with {\arrow[thick, #2]{<}}} 
    },
    -<-/.default={0.5}{}
}

\tikzset{
    partial ellipse/.style args={#1:#2:#3}{
        insert path={+ (#1:#3) arc (#1:#2:#3)}
    }
}

\newcommand{\dast}[1]{\draw #1 node[scale=0.5]{$\bullet$}}

\pgfdeclarelayer{bg}    
\pgfsetlayers{bg,main}  

\title{Cyclic quantum Teichm\"uller theory}

\author[Tsukasa Ishibashi]{Tsukasa Ishibashi}
\address{Tsukasa Ishibashi, Mathematical Institute, Tohoku University, 
6-3 Aoba, Aramaki, Aoba-ku, Sendai, Miyagi 980-8578, Japan.}
\email{tsukasa.ishibashi.a6@tohoku.ac.jp}

\date{\today}

\begin{document}

\begin{abstract}
Based on the pioneering ideas of Kashaev \cite{Kas98,Kas00}, we present a fully explicit construction of a finite-dimensional projective representation of the dotted Ptolemy groupoid when the quantum parameter $q$ is a root of unity, which reproduces the central charge of the $SU(2)$ Wess--Zumino--Witten model. A basic ingredient is the cyclic quantum dilogarithm \cite{FK}. 

A notable contribution of this work is a reinterpretation of the relations among the parameters in the cyclic quantum dilogarithm to ensure its pentagon identity in terms of the \emph{mutations of coefficients}. 
In particular, we find the dual roles of these parameters: as coefficients in quantum cluster algebras and as the central characters of quantum cluster variables. We also provide a geometric method to decompose the space of quantum states into irreducible modules of the Chekhov--Fock algebra. 

We introduce two versions of quantum intertwiners associated with a mapping class: on the entire representation space and on each irreducible component, each being an explicit composite of cyclic quantum dilogarithm operators. We prove that the former gives an intertwiner of local representations of quantum \Teich\ space in the sense of Bai--Bonahon--Liu \cite{BBL}, and also coincides with the transpose of the reduced quantum hyperbolic operator of Baseilhac--Benedetti \cite{BB-GD}. The mutation relation of coefficients is equivalent to the quantum gluing equation.
The irreducible intertwiner conjecturally coincides with the Bonahon--Liu intertwiner \cite{BL}, and we give a partial evidence. 

\end{abstract}

\maketitle

\setcounter{tocdepth}{1} 
\tableofcontents

\input{1_introduction}
\input{2_dilogarithm}
\input{3_Kashaev}

\input{4_Cluster}
\input{5_MCG_local}

\input{6_MCG_irreducible}

\input{7_Proof}
\input{8_Root}

\bigskip

\subsection*{Author Declarations}
The author declares no conflict of interest associated with this manuscript. The author confirms that the data supporting the findings of this study are available within the article.

\end{document}

%% file: 1_introduction.tex
\section{Introduction}\label{sec:intro}
Quantization of the \Teich\ theory is a program motivated by physical observations originally presented in the Verlinde's seminal paper \cite{Verlinde}, which has since drawn significant interest from the mathematical community. In a modern mathematical formulation, a key output of a quantum \Teich\ theory associated with a (marked) surface $\Sigma$ is a (projective) representation of (certain variants of) the \emph{Ptolemy groupoid}, whose objects are ideal triangulations of $\Sigma$ and morphisms are flips among them. Specifically, the program seeks to associate a vector space $V_\tri$ (`space of quantum states') to each triangulation $\tri$, and a linear operator $f_\omega: V_{\tri'} \to V_\tri$ (`quantum coordinate change') for a morphism $\omega:\tri \to \tri'$. 
These vector spaces should also carry actions of certain quantum tori (of local `quantum observables').
A fundamental consistency condition for these linear operators is the \emph{pentagon relation}, which involves five flips within the pentagon.

The representation can be finite-dimensional or infinite-dimensional, depending on the properties of the quantum parameter $\hbar$. 

\subsubsection*{Infinite dimensional case}
Infinite-dimensional representations have been extensively studied when $\hbar$ is generic (say $\hbar>0$), with two notable constructions provided by Kashaev \cite{Kas98,Kas99,Kas00,Kas01} and Chekhov--Fock \cite{CF99}. In both constructions, the representation spaces are $L^2$-spaces on certain coordinate spaces, and the flip operators are constructed using the \emph{Faddeev's (non-compact) quantum dilogarithm} \cite{FK}. The relation between these two constructions is elucidated by Teschner \cite{Tes07}, who demonstrated that these theories yield an infinite-dimensional modular functor corresponding to the Liouville CFT. Both of the constructions carry certain action of the quantum torus associated with each triangulation.

The Chekhov--Fock construction is significantly generalized into a representation theory of \emph{quantum cluster varieties} by Fock--Goncharov \cite{FG08} and Goncharov--Shen \cite{GS19}, including the higher \Teich\ spaces. The higher-rank analogs of these constructions are anticipated to produce the Toda CFTs, as discussed in \cite{SS17,GS19}. On the other hand, Kashaev's construction has been extended to a $(2+1)$-dimensional TQFT known as the \emph{Andersen--Kashaev TQFT} (or the \emph{\Teich\ TQFT}) \cite{AK-TQFT}, although its higher-rank generalizations remain undeveloped (see \cite{Kim16} for preliminary attempts). 

Notably, a key distinction between the two constructions lies in their underlying groupoid settings: the Chekhov--Fock construction provides a representation of the Ptolemy groupoid, whereas Kashaev's construction yields a projective representation of the \emph{dotted Ptolemy groupoid}, which parametrizes triangulations with `dots' and offers a refined control over the polarizations. 

\subsubsection*{Finite dimensional case}
Finite-dimensional representations have been studied from slightly different perspectives when $\hbar$ is rational (or equivalently $q:=e^{\pi i \hbar}$ is a root of unity), focusing on producing \emph{quantum invariants} of mapping classes or an appropriate class of 3-manifolds (with additional structures). Such studies include the works of Bonahon--Liu \cite{BL}, Bai--Bonahon--Liu \cite{BBL}, and Baseilhac--Benedetti \cite{BB-GD}, where the latter one is extended to the quantum hyperbolic field theory (QHFT) \cite{BB-GT,BB-AGT} and generalizes the Kashaev's link invariants \cite{Kas94,Kas95}. However, their approaches do not provide an explicit projective representation of the (dotted) Ptolemy groupoid. 

\bigskip
\paragraph{\textbf{Aim of the paper.}}
The primary goal of this paper is to provide a fully explicit construction of the finite-dimensional version of the Kashaev representation (building upon the ideas outlined in \cite{Kas98,Kas00}) in a way parallel to the works in the infinite-dimensional setting. We give all the quantum coordinate change maps in operator form, where the core ingredient is the \emph{cyclic quantum dilogarithm} \cite{FK}. 
We use a re-parametrized version of the cyclic quantum dilogarithm in literature, in order to clarify the role of the parameters in the framework of cluster algebra. 
We then obtain a unified (but partially conjectural) perspective for the works of Bonahon--Liu and Baseilhac--Benedetti, as outlined in \cref{fig:invariants}. 
This theory can be generalized to any quantum cluster varieties at roots of unity. This will be discussed in a future work. 

\subsection{Cyclic quantum dilogarithm and its pentagon relation}

Let $N$ be a positive odd integer. Let $q^2$ be a primitive $N$-th root of unity, and choose the square-root $q$ so that $q^N=1$ (see \cref{sec:dilog}). The cyclic quantum dilogarithm we use in this paper is given by the following expression:
\begin{align}\label{introeq:q-dilog}
    \Psi_\bp(X):=& \prod_{j=0}^{N-1} ( p^- + q^{2j+1} p^+X )^{j/N},
\end{align} 
where the parameter $\bp:=(p^+,p^-)$ is taken from the Fermat curve 
\begin{align*}
    \cF_N:=\{\bp=(p^+,p^-) \in (\bC^\ast)^2 \mid (p^+)^N+(p^-)^N=1\},
\end{align*}
and $X$ is a variable such that $X^N=1$. 
It is a re-parametrization of the version introduced by Kashaev \cite{Kas00}. See \cref{subsec:dilog} for the relation with the versions of cyclic quantum dilogarithms used in the literature. Our key observation is that the tuple $\bp=(p^+,p^-)$ plays the role of \emph{coefficient tuples} in the cluster algebra, while the ratio $y:=p^+/p^-$ is the \emph{exchange ratio} (also known as $y$-variables). 

It is well-known that the cyclic quantum dilogarithm \eqref{introeq:q-dilog} satisfies a pentagon relation \cite{FK}, where the parameters $\bp$ must satisfy several monomial relations. See \cref{thm:pentagon} in our notation. We show that these seemingly mysterious relations are, in fact, equivalent to the mutation relations of coefficients \cite[(1.5)]{FZ-CA2}:

\begin{introthm}[Reformulation of pentagon relation: \cref{thm:parameter_coeff}]\label{introthm:pentagon}
Let $\bp,\br \in \overline{\cF}_N$, and $\opU,\opP$ be operators satisfying $\opU^N=\opP^N=1$, $\opU\opP=q^2 \opP\opU$, acting on the standard representation $V_N$ \eqref{eq:std_rep}. Then the pentagon relation
\begin{align*}
    \Psi_\bp(\opU) \Psi_{\br'}(\opP) = \Psi_\br (\opP) \Psi_{\bp'}(q^{-1}\opU\opP) \Psi_{\bp''}(\opU) 
\end{align*}
holds on $V_N$ if and only if $\bp',\br',\bp'' \in \overline{\cF}_N$ satisfy the mutation relations
\begin{align}\label{introeq:pentagon_coeff_rel}
    \begin{aligned}
    \frac{{r'}^+}{{r'}^-} &= \frac{{r}^+}{{r}^-} (p^-)^{-1}, &
    \frac{{p''}^-}{{p''}^+} &= \frac{{p}^-}{{p}^+} ({r'}^-)^{-1}, &  
    \frac{{p'}^-}{{p'}^+} &= \frac{{r'}^-}{{r'}^+} ({p''}^+)^{-1}, \\
    \frac{{p''}^+}{{p''}^-} &= \frac{{r}^-}{{r}^+} {p'}^+, &
    \frac{{p'}^+}{{p'}^-} &= \frac{{p}^+}{{p}^-} {r}^+.
    \end{aligned}
\end{align}
\end{introthm}
The five relations in \eqref{introeq:pentagon_coeff_rel} are exactly the mutation relations among the coefficient variables along the pentagon cycle (also known as the $A_2$-cycle): see \cref{lem:coeff_rel=mutation}. 
Therefore, the pentagon relation of the cyclic quantum dilogarithm fits in with the spirit of \cite{KN}.


\subsection{Finite-dimensional projective representation of the dotted Ptolemy groupoid}
Using the cyclic quantum dilogarithm, we construct a finite-dimensional projective representation of the dotted Ptolemy groupoid $\mathrm{Pt}_\Sigma^{\tdot}$ (\cref{def:Ptolemy_groupoid}). To each dotted triangulation $\bD$ of $\Sigma$, we assign a finite-dimensional vector space $V_{\bD}$, and to each dot rotation (resp. flip), we assign a linear operator $\opA_v$ (resp. $\opT_{vw}$) between these spaces. See \cref{subsec:functor} for details. The space $V_{\bD}$ quantizes the Kashaev coordinates on an extension of the \Teich\ space \cite{Kas98}.
The operator $\opA_v$ acts as a kind of `Fourier transformation' of order $3$ changing the polarizations, and $\opT_{vw}$ is constructed using the cyclic quantum dilogarithm. 

We introduce the concept of \emph{global coefficients} $\bbp_\Sigma \in \cF_\Sigma$ (\cref{def:coeff_space}) based on the reformulation in \cref{introthm:pentagon}, which is needed to ensure the pentagon identity `everywhere' in the dotted Ptolemy groupoid. A global coefficient contains the data of framed $PSL_2(\bC)$-local systems on $\Sigma$ (\cref{cor:proj_cluster}), as well as a consistent choice of $N$-th roots of its cross ratio parameters. 

The following result was already stated by Kashaev \cite[Example 3]{Kas00}, although our global treatment of the parameters (coefficients) appears to be novel:

\begin{introthm}[\cref{thm:AT_relations}]\label{introthm:functor}
For any global coefficient $\bbp_\Sigma=\{\bp_\tri\}_{\tri \in \mathrm{Pt}_\Sigma} \in \cF_\Sigma$ (\cref{def:coeff_space}), we have a projective representation $V_\ast(\bbp_\Sigma): \mathrm{Pt}^{\tdot}_\Sigma \to \mathrm{Vect}_\bC^\op$: 
\begin{enumerate}
    \item $\opA_v^3=\mathsf{1}$. 
    \item $\opT_{uv}\opT_{uw}\opT_{vw}=\opT_{vw}\opT_{uv}$. 
    \item $\opA_v \opT_{vw} \opA_w =\opA_w \opT_{wv} \opA_v$.
    \item $\opT_{vw}\opA_v \opT_{wv} = \zeta \opA_v \opA_w \opP_{(12)}$, where $\zeta =e^{\frac{\pi i}{6}(1-\frac 1N)}$. 
\end{enumerate}
\end{introthm}
As already noted by Kashaev \cite{Kas00} (see also \cite{Kas99}), the phase factor $\zeta$ agrees with the central charge of the $SU(2)$ Wess--Zumino--Witten model of level $k=N-2$. 
The projective functor $V_\ast(\bbp_\Sigma)$ is referred to as a \emph{cyclic representation} of $\mathrm{Pt}^{\tdot}_\Sigma$ or a \emph{cyclic quantum \Teich\ theory}. In view of Bai--Bonahon--Liu \cite{BBL}, it corresponds to a local representation of the quantum \Teich\ space, as we will mention later.

\subsection{Chekhov--Fock algebras and quantum cluster Poisson transformations}
Fix a global coefficient $\bbp_\Sigma \in \cF_\Sigma$. 
To each ideal triangulation $\tri$ of $\Sigma$, we associate a quantum torus $\X_{\by^{\tri}}^{\tri}$, referred to as the \emph{Chekhov--Fock algebra} (or the \emph{quantum cluster Poisson torus}). See \cref{def:CF-alg}. Here we fix central characters as $X_\alpha^N=y_\alpha^N$ for each generator $X_\alpha$, where $y_\alpha=p^+_\alpha/p^-_\alpha$ is the exchange ratio defined from $\bbp_\Sigma$. 

In this setting, the quantum cluster Poisson transformations among these quantum tori can be expressed using the cyclic quantum dilogarithm:
\begin{align*}
    \mu_\kappa^\ast= \Ad_{\Psi_{\bp_\kappa}(\overline{X}_\kappa)} \circ \mu'_{\kappa,+},
\end{align*}
providing a root of unity analogue of the Fock--Goncharov decomposition \cite{FG09,KN}. Here $\overline{X}_\kappa:=y_\kappa^{-1}X_\kappa$, which satisfies $\overline{X}_\kappa^N=1$. See \cref{lem:mutation_decomp} for a signed version. 
Notably, the $\overline{X}$-variables are related by \emph{quantum cluster Poisson transformations with coefficients} (\cref{lem:cluster_transf_Xbar}) first introduced in \cite{CMM}, demonstrating that the parameter $\bbp_\Sigma$ serves dual roles as both coefficients and central characters. 

This algebraic structure naturally generalizes to any quantum cluster variety \cite{FG09}, including the quantum higher \Teich\ theory \cite{FG06,GS19}. Furthermore, one can develop the root of unity version of the Fock--Goncharov representation described in \cite{FG08}, where the intertwiners are also constructed using the cyclic quantum dilogarithm. In this context, a quantum invariant for mutation loops is produced, generalizing \eqref{introeq:q-trace} below. 
This direction will be explored in future work. 
Another representation is obtained by the irreducible decomposition of $V_\ast(\bbp_\Sigma)$, which is discussed in the next subsection. 

\subsection{$V_\ast(\bbp_\Sigma)$ as modules over the Chekhov--Fock algebra}
Given a dotted triangulation $\bD$ over $\tri$, we construct a canonical embedding $\iota_{\bD}: \X_{\by^\tri}^{\tri} \to \cW_{\bD}$ (\cref{lem:Weyl_embedding}). 
The projective functor $V_\ast(\bbp_\Sigma)$ is compatible with the quantum cluster Poisson transformations in the following way:

\begin{introthm}[Compatibility of transformations]
In the dotted Ptolemy groupoid of any marked surface $\Sigma$, we have the following commutative diagrams:
\begin{enumerate}
    \item For any dot rotation $A_v=[\tri_1,\tri_2]$, 
    \begin{equation*}
    \begin{tikzcd}
        \X^\tri_{\by^\tri} \ar[r,"\iota_{\bD_2}"] \ar[d,phantom,"=",sloped] & \cW_{\bD_2} \ar[d,"\Ad({\opA_v})"] \\
        \X^\tri_{\by^\tri} \ar[r,"\iota_{\bD_1}"'] & \cW_{\bD_1},
    \end{tikzcd}
    \end{equation*}
    where $\tri:=\tri_1=\tri_2$ is the common underlying triangulation. 
    \item  For any flip $T_{vw}=[\tri_1,\tri_2]$,
    \begin{equation*}
    \begin{tikzcd}
        \X^{\tri_2}_{\by^{\tri_2}} \ar[r,"\iota_{\bD_2}"] \ar[d,"\mu_\kappa^\ast"'] & \cW_{\bD_2} \ar[d,"\Ad(\opT_{vw})"] \\
        \X^{\tri_1}_{\by^{\tri_1}} \ar[r,"\iota_{\bD_1}"'] & \cW_{\bD_1}.
    \end{tikzcd}
    \end{equation*}
    Here $\mu_\kappa: \tri_1 \to \tri_2$ denotes the flip of the underlying triangulation.
\end{enumerate}
\end{introthm}
See \cite{GuoLiu} for a related statement. 
In particular, the Chekhov--Fock algebras also act on the representation spaces $V_{\bD}$ in a compatible way. Although $V_{\bD}$ is not irreducible as an $\X_{\by^\tri}^{\tri}$-module, the modules $V_{\bD}$ yield a \emph{local representation} of the quantum \Teich\ space $\mathcal{T}_\Sigma^q$ in the sense of Bai--Bonahon--Liu \cite{BBL}. Instead of irreducibility, it exhibits a good behavior under the cutting and gluing of marked surfaces. 

It has turned out that an appropriate selection of an intertwiner between two local representations is a subtle problem, as pointed out and carefully analyzed by Mazzoli \cite{Mazzoli}. 
In our approach, the intertwiners among the modules are now explicitly given. See \cref{subsec:BBL} for a detailed discussion. 

\subsection{Irreducible decomposition of $V_\ast(\bbp_\Sigma)$}\label{introsec:irred_decomp}
Following the ideas of \cite{Kas98} and \cite[Section 10.3]{Tes07}, we further obtain an irreducible decomposition
\begin{align*}
    V_{\bD} = \bigoplus_{\boldsymbol{\lambda}} V_{\bD}(L,\boldsymbol{\lambda})
\end{align*}
by fixing a Lagrangian sub-lattice $L \subset H_1(\Sigma^{\mathrm{cl}};\bZ)$. Here, $\Sigma^{\mathrm{cl}}$ is the closed surface of genus $g \geq 0$ obtained from $\Sigma$ by capping a disk to each boundary component, and closing interior punctures, and $\boldsymbol{\lambda}:\widehat{L} \to \bZ_N$ runs over additive characters on the maximal isotropic sub-lattice $\widehat{L} \subset H_1(\Sigma^\ast,\partial\Sigma^\ast;\bZ)$ generated by elements of $L$ and certain elements associated with punctures and boundary components of $\Sigma$. Each irreducible component is a simultaneous eigenspace for specific \emph{homology operators} associated with $\widehat{L}$ (see \cref{thm:irrep}). These operators commute with all the generators of the Chekhov--Fock algebra. 
Note that an irreducible decomposition of a local representation has been obtained by Toulisse \cite{Toulisse}. Our formulation clarifies the geometric meaning of the decomposition indices, specifying them as additive characters $\boldsymbol{\lambda}$.

Since the homology operators are compatible with elementary operators $\opA_v$ and $\opT_{vw}$, we obtain: 
\begin{introthm}[\cref{thm:functor_irrep}]
For any global coefficient $\bbp_\Sigma \in \cF_\Sigma$, Lagrangian sub-lattice $L \subset H_1(\Sigma^{\mathrm{cl}};\bZ)$ and additive character $\boldsymbol{\lambda}:\widehat{L} \to \bZ_N$, the projective representation $V_\ast(\bbp_\Sigma)$ restricts to the projective representation
\begin{align*}
    V_\ast(\bbp_\Sigma;L,\boldsymbol{\lambda}): \mathrm{Pt}^{\tdot}_\Sigma \to \mathrm{Vect}_\bC^\op, \quad \bD \mapsto V_\bD(L,\boldsymbol{\lambda}).
\end{align*}
\end{introthm}
We call $V_\ast(\bbp_\Sigma;L,\boldsymbol{\lambda})$ the \emph{irreducible cyclic quantum \Teich\ theory} associated with a global coefficient $\bbp_\Sigma \in \cF_\Sigma$, Lagrangian sub-lattice $L \subset H_1(\Sigma^{\mathrm{cl}};\bZ)$ and additive character $\boldsymbol{\lambda}:\widehat{L} \to \bZ_N$. 

In \cref{subsec:example}, we describe a few quantum geometric structures that the irreducible cyclic quantum \Teich\ theory carries:

\begin{itemize}
    \item Thrice-punctured sphere and `Conformal block': Assign a triple $(\lambda,\mu,\nu)$ of $\bZ_N$-weights to the punctures. The irreducible component $V_{\bD}(\lambda,\mu,\nu)$ is always $1$-dimensional (\cref{prop:conformal_block}). Its generator resembles the `conformal block', though the Clebsch--Gordan rule is absent. 
    \item Once-punctured disk and the quantum group: A canonical embedding of the quantum group $U_q(\mathfrak{sl}_2)$ into the Chekhov--Fock algebra is constructed by Ip \cite{Ip} and Schrader--Shapiro \cite{SS19}. We show that our irreducible theory recovers the cyclic representations of $U_q(\mathfrak{sl}_2)$ \cite{DJMM91} through this embedding.
\end{itemize}

\subsection{Quantum invariant of mapping classes I: Local quantum trace}
The mapping class group $MC(\Sigma)$ naturally acts on the dotted Ptolemy groupoid. 
This action in turn induces an ``action'' on the cyclic quantum \Teich\ theory $V_\ast(\bbp_\Sigma)$, in the following sense. A mapping class $\phi \in MC(\Sigma)$ induces a canonical linear map $V_{\phi}^{\bD}(\bbp_\Sigma): V_{\bD} \xrightarrow{\sim} V_{\bD}$ defined up to powers of $\zeta$ in \cref{introthm:functor}, which we refer to as the \emph{local intertwiner} of $\phi$.
The absolute value of its trace 
\begin{align}\label{introeq:q-trace}
    Z_N(\phi;\bbp_\Sigma):= \big|\mathrm{Tr}(V_{\phi}^{\bD}(\bp_\Sigma))\big|
\end{align}
is independent of $\bD$, provided that $\bbp_\Sigma$ satisfies the $\phi$-invariance condition \eqref{eq:fixed}. This condition in particular implies that the image of $\bbp_\Sigma$ in the cluster Poisson variety under the map \eqref{eq:proj_cluster} is fixed under the natural action of $\phi$. The latter space is birationally isomorphic to a framed version of the $PSL_2(\bC)$-character variety. We call the invariant $Z_N(\phi;\bbp_\Sigma)$ the \emph{local quantum trace} of $\phi$. The local intertwiner is automatically normalized so that $|\det V_{\phi}^\bD(\bp_\Sigma)|=1$.
In \cref{subsec:MCG}, we provide explicit computations of the quantum intertwiners for Dehn twists and a pseudo-Anosov mapping class whose mapping torus gives the figure-eight knot complement. 

Though the existence of a global $\phi$-invariant coefficient is a subtle problem (see \cref{thm:hyp_str}), we only need to find a $\phi$-invariant coefficient $\bbp_{\underline{\omega}}$ over a specified flip sequence $\underline{\omega}$ representing $\phi$ (\cref{def:coeff_local}) to compute the local intertwiner. We thus use the notation $V_\phi^\bD(\bbp_{\underline{\omega}})$ and $Z_N(\phi;\bbp_{\underline{\omega}})$.
This `local' choice includes the data of a $PSL_2(\bC)$-local system on the mapping torus $M_\phi$ (\cref{def:mapping_torus}), and is also the data used in the related earlier works (e.g. \emph{QH gluing variety} of \cite{BB-AGT}). 


\subsection{Relation to Baseilhac--Benedetti's quantum hyperbolic field theory (QHFT)}
Kashaev \cite{Kas95} introduced a state-sum invariant of links, now known as the \emph{Kashaev invariant}, by assembling the tetrahedron operator associated with each tetrahedron in a triangulated link exterior. Murakami--Murakami \cite{MM} proved that is coincides with the value at $e^{2\pi i/N}$ of the Jones polynomial. 
Kashaev's original construction was further clarified by Baseilhac--Benedetti, and generalized to the \emph{quantum hyperbolic field theory (QHFT)} \cite{BB-GT,BB-AGT}. The latter produces a quantum invariant of 3-manifolds with $PSL_2(\bC)$-characters (with a bit more additional structures), called the \emph{quantum hyperbolic invariant} (QHI for short), which recovers the Kashaev invariant at the trivial character \cite{BB-GGT}. If the 3-manifold has a non-empty boundary, it provides a \emph{quantum hyperbolic operator} (QH operator for short) between certain vector spaces associated to the boundary surfaces. 
In \cite{BB-QT}, it is shown that the QHI splits into two parts, the \emph{symmetry defect} and the \emph{reduced QHI}. 
In \cite{BB-GD}, it is shown that the transpose of the reduced QH operator of a mapping cylinder gives a canonical intertwiner of local representations of the quantum \Teich\ space. 

In \cref{sec:invariant}, we will prove that it also coincides with our quantum intertwiner:

\begin{introthm}[\cref{thm:QHFT}]\label{introthm:QHFT}
Take a mapping class $\phi \in MC(\Sigma)$, a sequence $\omega: \bD \to \phi^{-1}(\bD)$ of elementary moves in $\mathrm{Pt}^{\tdot}_\Sigma$ representing $\phi$, and a $\phi$-invariant coefficient $\bbp_{\underline{\omega}}$ over the underlying flip sequence $\underline{\omega}$. Then
\begin{itemize}
    \item $\bbp_{\underline{\omega}}$ determines a quantum shape assignment $\mathbf{w}$ (\cref{prop:coeff_shape}).
    \item $\omega$ determines a 3D triangulation $\Delta_{C_\phi}^{(3)}$ of the mapping cylinder $C_\phi$ with a weak branching $\widetilde{b}$. 
\end{itemize}
Upon these correspondences, the transposed reduced QH operator $\mathcal{H}_N^{\mathrm{red}}(\Delta_{C_\phi}^{(3)},\widetilde{b},\mathbf{w})^\top$ is conjugate to the quantum intertwiner $\opV_\phi^{\bD}(\bbp_{\underline{\omega}})$. In particular, the absolute value of the reduced QHI of the mapping torus $M_\phi$ (\cref{def:mapping_torus}) and the $PSL_2(\bC)$-character $\rho$ determined by $\bbp_{\underline{\omega}}$ coincides with the local quantum trace $Z_N(\phi;\bbp_{\underline{\omega}})$. 
\end{introthm}
Note that $\mathbf{w}$ is a `quantum' (or the $N$-th root) version of shape parameters of $\rho$. 
\cref{prop:coeff_shape} relates the mutation rule of coefficients to the quantum shape assignment, giving a `quantum' version of \cite[Proposition 4.1]{NTY}. An explicit conjugation of our projective representation $V_\ast(\bbp_\Sigma)$ is given in \cref{prop:two_realizations}.

\subsection{Quantum invariant of mapping classes II: Irreducible quantum trace}
Recall the irreducible decomposition of $V_\ast(\bbp_\Sigma)$ discussed in \cref{introsec:irred_decomp}.
The local intertwiner of a mapping class undergoes the following reduction: 

\begin{introthm}[\cref{thm:polarized_intertwiner}]
Let $\phi \in MC(\Sigma)$ be a mapping class, and $\bbp_{\underline{\omega}}$ a $\phi$-invariant coefficient over a flip sequence $\underline{\omega}$ representing $\phi$. 
Given a Lagrangian sub-lattice $L \subset H_1(\Sigma^{\mathrm{cl}};\bZ)$ and an additive character $\boldsymbol{\lambda}:\widehat{L} \to \bZ_N$, the local intertwiner $\opV_\phi^{\bD}(\bbp_{\underline{\omega}})$ induces an intertwiner of $\X_{\by^\tri}^{\tri}$-modules
\begin{align*}
    \opV_{\phi}^{\bD}(\bbp_{\underline{\omega}};L,\boldsymbol{\lambda}): V_{\bD}(L,\boldsymbol{\lambda}) \xrightarrow{\sim} V_{\bD}(\phi(L),\phi_\ast\boldsymbol{\lambda}),
\end{align*}
where $\phi_\ast\boldsymbol{\lambda}: \phi(\widehat{L}) \to \bZ_N$ is given by $\phi_\ast\boldsymbol{\lambda}(\phi([c]))=\boldsymbol{\lambda}([c])$ for $[c] \in \widehat{L}$.
\end{introthm}
In order to obtain an automorphism on $V_{\bD}(L,\boldsymbol{\lambda})$, we need to adjust the polarization from $\phi(L)$ to $L$. For this purpose, we introduce the \emph{transvection operators} \eqref{eq:transvection} that implement the transvection on the homology. Using an appropriate composite $\opF_\phi^{\bD}$ of transvection operators, we define the \emph{irreducible intertwiner} (\cref{def:reduced_intertwiner}) of $\phi$:
\begin{align*}
    \overline{\opV}_\phi^{\bD}(\bbp_{\underline{\omega}};L,\lambda):=\opF_\phi^{\bD} \circ \opV_\phi^{\bD}(\bbp_{\underline{\omega}};L,\boldsymbol{\lambda}): V_{\bD}(L,\boldsymbol{\lambda}) \xrightarrow{\sim} V_{\bD}(L,\boldsymbol{\lambda}).
\end{align*}
The absolute value of its trace 
\begin{align*}
    \overline{Z}_N(\phi;\bbp_{\underline{\omega}};L,\boldsymbol{\lambda}) := |\mathrm{Tr} (\overline{\opV}_\phi^{\bD}(\bbp_{\underline{\omega}};L,\boldsymbol{\lambda}))|
\end{align*}
is called the \emph{irreducible quantum trace} of $\phi$. 

\subsubsection*{Relation to the works of Bonahon--Liu and Bonahon--Wong--Yang}
Bonahon--Liu initiated the study of intertwiners associated with mapping classes for irreducible representations of the quantum \Teich\ space \cite{BL}. 
Since our reduced intertwiner $\opV_\phi^{\bD}(\bbp_{\underline{\omega}};L,\lambda)$ still intertwines the action of the Chekhov--Fock algebra, it provides (up to a constant factor) the Bonahon--Liu's intertwiner. 
We state the following:
\begin{introconj}[\cref{conj:BW_intertwiner}]\label{introconj:BW_intertwiner}
Each reduced intertwiner satisfies $|\det \overline{\opV}_\phi^{\bD}(\bbp_{\underline{\omega}};L,\boldsymbol{\lambda})|=1$. In particular, it coincides with the Bonahon--Liu intertwiner of $\phi$ associated with the irreducible representation $V_{\bD}(L,\boldsymbol{\lambda})$. 
\end{introconj}
This conjecture is verified for Dehn twists (\cref{subsub:example_torus}) by explicit computations, and for any mapping class on a once-punctured closed surface (\cref{cor:conj_once_punc}). 

The Bonahon--Liu invariant coincides \cite[Theorem 16]{BWY-1} with the Bonohon--Wong--Yang invariant \cite{BW15,BWY-1,BWY-2} arising from the skein theory. For the volume conjecture regarding the Bonahon--Wong--Yang invariant, see \cite{BWY-1,BWY-2}. A work of Garoufalidis--Yu \cite{GY} conjectures that a suitable ratio of Bonahon--Wong--Yang invariants coincides with the 1-loop invariant of the mapping torus. Thus our conjecture (and the partial results) provides a connection to these topics. 


The (partially conjectural) relations among several quantum invariants are summarized in \cref{fig:invariants}. 

Recent work of Garoufalidis--Yu \cite{GYrel} proves a refined version of \cref{introconj:BW_intertwiner}, clarifying the precise relationship between the Baseilhac--Benedetti invariant (=refined version of local quantum trace) and the Bonahon--Liu--Wong--Yang invariant (=refined version of irreducible quantum trace). Their work provides a refined control on the multiplicative ambiguity by roots of unity, without taking the absolute values.

\begin{figure}[ht]
    \centering
\begin{tikzpicture}
\node[draw,rounded corners=5pt,align=center](QT) at (-0.5,0) {Local quantum trace \\ $Z_N(\phi;\bbp_{\underline{\omega}})$}; 
\node[draw,rounded corners=5pt,align=center](RQT) at (-0.5,-3) {Irreducible quantum trace \\ $\overline{Z}_N(\phi;\bbp_{\underline{\omega}};L,\boldsymbol{\lambda})$};
\node[draw,rounded corners=5pt,align=right](RQHI) at (5,0) {Reduced QHI};
\node[draw,rounded corners=5pt] (SD) at (10,0) {Symmetry defect};
\node[draw,rounded corners=5pt,align=right](QHI) at (7.5,2) {QHI};
\node[draw,rounded corners=5pt,align=right](K) at (12,2){Kashaev invariant}; 
\node[draw,rounded corners=5pt,align=center](BWY) at (7,-3) {Bonahon--Liu inv. \\ =Bonahon--Wong--Yang inv.};
\node[draw,rounded corners=5pt,align=right](loop) at (13,-3){1-loop invariant}; 

\draw[<-] (RQHI.north) --++(0,0.5) coordinate(x);
\draw[<-] (SD.north)--++(0,0.5) coordinate(y);
\draw (x) -- (y);
\draw (QHI.south) --++(0,-0.83);
\draw[double distance=2pt,shorten >=2pt,shorten <=2pt] (QT.east) --node[above,scale=0.8]{\cref{introthm:QHFT}} (RQHI.west);
\draw[->] (QT.south) --node[right,align=center,scale=0.8]{irreducible \\ decomposition} (RQT.north);
\draw[->] (QHI.east) --node[above,align=center,scale=0.8]{trivial \\ character}node[below,scale=0.8]{\cite{BB-GGT}}  (K.west);
\draw[double distance=2pt,shorten >=2pt,shorten <=2pt] (RQT.east) --node[below]{?} node[above,scale=0.8]{\cref{introconj:BW_intertwiner}}  (BWY.west);
\draw[double distance=2pt,shorten >=2pt,shorten <=2pt] (BWY.east) --node[below]{?}node[above,scale=0.8]{\cite{GY}}  (loop.west);
\end{tikzpicture}
    \caption{Relations among several quantum invariants.}
    \label{fig:invariants}
\end{figure}

Since our formulation aligns with the Fock--Goncharov representation in the infinite-dimensional setting, this relation would also provide insight to the Hikami--Inoue's result \cite{HI14}, which connects the Kashaev's $R$-matrix to a 'root-of-unity limit' of the Fock--Goncharov representation in \cite{FG08}. Here we do not take an actual limit. The relation to the Kashaev's $6j$-symbol has been explained by Baseilhac \cite{Baseilhac} and Bai \cite{Bai}.

\subsection*{Organization of the paper}
This paper can be divided into three parts:
\begin{description}
    \item[Part I (Sections \ref{sec:dilog}--\ref{sec:Fourier})] \textbf{Cyclic quantum dilogarithm}. We investigate basic functional identities of cyclic quantum dilogarithm, and recall the discrete Fourier analysis. We provide all the proofs in our convention.  
    \item[Part II  (Sections \ref{sec:Kashaev}--\ref{sec:cluster})] \textbf{Representations of the dotted Ptolemy groupoid}. We present a fully explicit construction of the cyclic quantum \Teich\ theory $V_\ast(\bbp_\Sigma)$. Then we provide its irreducible decomposition as a module over the Chekhov--Fock algebra. 
    \item[Part III  (Sections \ref{sec:trace_local}--\ref{sec:trace_irred})] \textbf{Quantum invariants of mapping classes}. We introduce two versions of intertwiners associated with a mapping class: the local intertwiner and the irreducible intertwiner. Their traces define local quantum trace and the irreducible quantum trace, respectively. We discuss their relationship with the earlier works outlined in \cref{fig:invariants}. 
\end{description}
Necessary background on the mapping torus and the moduli space $\P_{PGL_2,\Sigma}$ is summarized in \cref{sec:mapping_torus}.
Necessary modifications in the case $q^N=-1$ is briefly discussed in \cref{sec:root}.

\subsection*{Acknowledgements}
The author expresses his gratitude to Yuji Terashima for inspiring him to study the root-of-unity quantum cluster transformations and for providing insightful comments. He is deeply grateful to St\'ephane Baseilhac for illuminating discussion on the relation to the QHFT and sharing ideas towards \cref{introconj:BW_intertwiner}. He also thanks Yoshiyuki Kimura for his valuable comments on the cyclic representations of quantum groups. His thanks also go to Yuma Mizuno for his crucial comment regarding the choice of $q$. He appreciates the helpful comments and suggestions by anonymous reviewers. 
The author is supported by JSPS KAKENHI Grant Number~JP20K22304 and JP24K16914.



%% file: 2_dilogarithm.tex
\section{Cyclic quantum dilogarithm}\label{sec:dilog}
Let $N$ be a positive odd integer. Let $q^2$ be a primitive $N$-th root of unity, and choose its square-root to be $q:=(q^2)^{(N+1)/2}$ so that $q^N=1$. Our standard choice is $q^2=e^{2\pi i/N}$ and $q=-e^{\pi i/N}$. 

\begin{lem}\label{lem:cyclic_prod}
We have 
\begin{align*}
    \prod_{k=0}^{N-1} (y-q^{2k}x)=y^N-x^N, \quad \prod_{k=0}^{N-1} (y\pm q^{2k+1}x)=y^N\pm x^N.
\end{align*}
\end{lem}

\begin{proof}
The first equality follows from $1-z^N=\prod_{k=0}^{N-1} (1-q^{2k}z)$ by setting $z=x/y$. Then the second one is obtained by replacing $x \mapsto \pm qx$, with a notice that $q^N=1$ and $(-1)^N=1$.
\end{proof}

\begin{rem}
Our choice of $q$ follows that of e.g. \cite{Kas94}. When $N$ is even, we have no choice satisfying $q^N=1$. 
Another natural option is to take $\zeta:=e^{\pi i/N}$ so that $\zeta^N=-1$, which is adopted e.g. in \cite{BR,IY} and works for any positive integer $N$. 
Most of the theory in this paper also works for this choice, together with extra signs in formulas. See \cref{sec:root}. 
\end{rem}

\subsection{Definition}\label{subsec:dilog}

Let 
\begin{align*}
    \cF_N:=\{\bp=(p^+,p^-) \in (\bC^\ast)^2 \mid (p^+)^N+(p^-)^N=1\}
\end{align*}
denote the (punctured) Fermat curve, and consider the projection 
\begin{align*}
    \pi:\cF_N \to \bC^\ast, \quad (p^+,p^-) \mapsto y:=p^+/p^-,
\end{align*}
which is a principal $\bZ_N$-bundle with the action given by $k.(p^+,p^-):=(q^{2k}p^+,q^{2k}p^-)$ for $k \in \bZ_N$. 
In \cref{sec:cluster}, the pair $(p^+,p^-)$ is interpreted as the coefficient tuple, while $y$ is the exchange ratio. We also consider the partial compactification $\overline{\cF}_N:=\cF_N \cup \{(0,1)\}$.

We use the following version of the cyclic quantum dilogarithm:

\begin{dfn}
Given a point $\bp=(p^+,p^-) \in \overline{\cF}_N$, we define the \emph{cyclic quantum dilogarithm} by the expression
\begin{align}\label{eq:q-dilog}
    \Psi_\bp(X):=& \prod_{j=0}^{N-1} ( p^- + q^{2j+1} p^+X )^{j/N}.
\end{align}

\end{dfn}
Here, we need a care about the choice of a branch of the $N$-th root and the domain of the function $\Psi_{\bp}$. Fixing $\bp \in \overline{\cF}_N$, the function \eqref{eq:q-dilog} is firstly defined as a convergent power series on the domain $|yX|<1$ with $y=\pi(\bp) \in \bC$, and then analytically continued to the complex domain 
\begin{align*}
    \cD_N(y):= \bC \setminus \bigcup_{k=0}^{N-1} \bigg\{ |yX|>1,~ \arg(yX)=\frac{\pi (2k+1)}{N}\bigg\}.
\end{align*}
See \cref{fig:domain}. Then the cyclic quantum dilogarithm is regarded as a holomorphic function $\Psi_{\bp}: \cD_N(y) \to \bC$. Note that $\cD_N(q^2 y) = \cD_N(y)$.

\begin{figure}[ht]
    \centering
\begin{tikzpicture}
\filldraw[fill=gray!30,draw=black,dashed] (0,0) circle(1cm);
\draw[->] (-3,0) -- (3,0) node[right]{$\Re z$};
\draw[->] (0,-3) -- (0,3) node[above]{$\Im z$};
\foreach \i in {0,1,2,3,4}
\draw[thick,decorate,decoration={snake,amplitude=1pt,pre length=1pt,post length=1pt}] (72*\i+36:1) node{$\times$} -- (72*\i+36:2.8);
\end{tikzpicture}
    \caption{The domain $\cD_5(\bp)$ with $\arg(y)=0$.}
    \label{fig:domain}
\end{figure}

\begin{rem}\label{rem:use_of_dilog}
In practice, we will use $\Psi_{\bp}(\opX)$ for some cyclic operator $\opX$ via `functional analysis' (\cref{def:funct_analysis}). Since the only eigenvalues of such an operator are $q^{2k}$ with $k \in \bZ_N$, we only need the values $\Psi_{\bp}(q^{2k})$. The parameter $\bp$ is generically chosen so that $q^{2k} \in \cD_N(y)$. 
\end{rem}

Let us mention the relations to several versions of cyclic quantum dilogarithm used in the literature. 
Take a generic section 
\begin{align}\label{eq:generic_section}
    s: \bC^\ast \to \cF_N, \quad y \mapsto (p^+,p^-)=\bigg(\left(\frac{y^N}{1+y^N}\right)^{1/N},\left(\frac{1}{1+y^N}\right)^{1/N}\bigg)
\end{align}
by choosing a branch of the $N$-th root. Then we may also write
\begin{align*}
    \Psi_y(X):=\Psi_{s(y)}(X) =& (1+y^N)^{\frac{1-N}{2N}} \prod_{j=0}^{N-1} (1 +q^{2j+1}yX)^{j/N}.
\end{align*}

\begin{itemize}
    \item The specialization 
\begin{align*}
    \Psi_{-y}(1)^{-1}= (1+(-y)^N)^{\frac{N-1}{2N}} \prod_{j=0}^{N-1} (1 +\zeta^{2j+1}y)^{-j/N} 
\end{align*}
coincides with the cyclic quantum dilogarithm $d_N(y)$ used by Ip--Yamazaki \cite{IY} with $\zeta=-q$. They have obtained the quantum dilogarithm identities for $d_N(y)$, where the parameters $y$ are upgraded into quantum variables. 
    \item The specialization 
    \begin{align*}
    \Psi_{-q^{-1}y}(1)^{-1}= (1-y^N)^{\frac{N-1}{2N}} \prod_{j=0}^{N-1} (1 -\omega^{-j}y)^{-j/N}
\end{align*}
coincides with $d(y)=w(y,0)^{-1}$ used by Bazhanov--Reshetikhin \cite{BR} with $\omega=q^{-2}$ and $\omega^{1/2}=-q^{-1}$. 
    \item Our cyclic quantum dilogarithm is most closely related to the one used in \cite{Kas00} 
\begin{align}\label{eq:CQL_Kas00}
    \Psi^K_x(X):=\prod_{j=0}^{N-1}((1-x)^{1/N} -x^{1/N}\theta^{-2j}X)^{j/N}.
\end{align}
Here the relations between the parameter is $(p^+,p^-)=(x^{1/N},(1-x)^{1/N})$, and we have
\begin{align*}
    \Psi_\bp(X) = \Psi^K_{(p^+)^N}(\theta^{-1}X)
\end{align*}
with $\theta=-q^{-1}$. 
\end{itemize}



\subsection{$q$-difference relation}

\begin{lem}\label{lem:q-diff_eq}
If $X$ is a variable such that $X^N=1$, then $\Psi_\bp$ satisfies the difference equations 
\begin{align*}
    \Psi_\bp(q^{2}X) &= (p^- + qp^+X ) \Psi_\bp(X), \\
    \Psi_\bp(q^{-2}X) &= (p^- + q^{-1}p^+ X )^{-1} \Psi_\bp(X).
\end{align*}
\end{lem}

\begin{proof}
For the first equation, we compute
\begin{align*}
    \Psi_\bp(q^2X) &= \prod_{j=0}^{N-1}(p^- + q^{2j+3} p^+ X)^{j/N} \\
    &= \prod_{j=1}^{N}(p^- + q^{2j+1} p^+ X)^{(j-1)/N} \\
    &=\Psi_\bp(X) \prod_{j=1}^{N}(p^- + q^{2j+1} p^+ X)^{-1/N} \cdot (p^- + q^{2N+1}p^+ X) \\
    &= \Psi_\bp(X) (p^- + qp^+ X).
\end{align*}
Here in the last line, we use \cref{lem:cyclic_prod} to compute the cyclic product as $\prod_{j=1}^{N}(p^- + q^{2j+1} p^+ X)=(p^-)^N + (p^+ X)^N = (p^-)^N + (p^+)^N =1$. Then $\prod_{j=1}^{N}(p^- + q^{2j+1} p^+ X)^{-1/N}$ is an $N$-th root of unity, which is constant for $\bp \in \overline{\cF}_N$. Taking $(p^+,p^-)=(0,1)$, it is indeed seen to be $1$. 

The second equation is obtained from the first one by substituting $X \mapsto q^{-2}X$, where the condition $X^N=1$ is preserved. 
\end{proof}

\subsection{Values on $q$-powers}
Let us use the notation 
\begin{align*}
    \wt(\bp|k):=\prod_{j=1}^{k} (p^- +q^{2j-1}p^+),
\end{align*}
following \cite{FK,BR}. 
By using the $q^2$-Pochhammer symbol $(a;q^2)_k:=\prod_{j=1}^k (1-a q^{2j-2})$, we can also write $w(\bp|k)=(p^-)^k \cdot (-qy;q^2)_k$.

\begin{lem}
We have $w(\bp|k+N)=w(\bp|k)$. 
\end{lem}
Therefore we regard $w(\bp|-)$ as a function of $k \in \bZ_N$.

\begin{proof}
\begin{align*}
    w(\bp|k+N)= \prod_{j=1}^{k+N} (p^- +q^{2j-1}p^+) = \prod_{j=k+1}^{k+N} (p^- +q^{2j-1}p^+) \cdot \prod_{j=1}^{k} (p^- +q^{2j-1}p^+) = w(\bp|k).
\end{align*}
Here we have used $\prod_{j=k+1}^{k+N} (p^- +q^{2j-1}p^+)=\prod_{j \in \bZ_N} (p^- +q^{2j-1}p^+)=1$ by \cref{lem:cyclic_prod}.
\end{proof}

\begin{lem}\label{lem:Psi_diagonal}
For $k \in \bZ_N$, we have $\Psi_\bp(q^{2k}) = \Psi_\bp(1) \wt(\bp|k) = \Psi_\bp(1)(p^-)^k \cdot (-qy;q^2)_k$.
\end{lem}

\begin{proof}
\begin{align*}
    \Psi_\bp(q^{2k})
    &= \prod_{j=0}^{N-1} (p^- + q^{
    2(j+k)+1}p^+)^{j/N} \\
    &= \prod_{j=k}^{N+k-1} (p^- + q^{
    2j+1}p^+)^{(j-k)/N} \\
    &= \Psi_\bp(1)\cdot \frac{\prod_{j=N}^{N+k-1} (p^- + q^{
    2j+1}p^+)^{j/N}}{\prod_{j=0}^{k-1} (p^- + q^{
    2j+1}p^+)^{j/N}}\cdot \prod_{j \in \bZ_N} (p^- + q^{
    2j+1}p^+)^{-k/N}.
\end{align*}
Here the cyclic product $\prod_{j \in \bZ_N}$ vanishes, thanks to \cref{lem:cyclic_prod} and an argument similar to the proof of \cref{lem:q-diff_eq}. 
Moreover, 
\begin{align*}
    \frac{\prod_{j=N}^{N+k-1} (p^- + q^{
    2j+1}p^+)^{j/N}}{\prod_{j=0}^{k-1} (p^- + q^{
    2j+1}p^+)^{j/N}} = \frac{\prod_{j=0}^{k-1} (p^- + q^{
    2j+1}p^+)^{(j+N)/N}}{\prod_{j=0}^{k-1} (p^- + q^{
    2j+1}p^+)^{j/N}} = \prod_{j=0}^{k-1} (p^- + q^{
    2j+1}p^+).
\end{align*}
Therefore we get
\begin{align*}
    \Psi_\bp(q^{2k}) = \Psi_\bp(1) \prod_{j=1}^{k} (p^- + q^{
    2j-1}p^+) = \Psi_\bp(1) \wt(\bp|k) 
\end{align*}
as desired. 
\end{proof}




\begin{lem}\label{lem:psi_product}
$\prod_{k \in \bZ_N} \Psi_\bp(q^{2k}) = 1$.
\end{lem}

\begin{proof}
\begin{align*}
    \prod_{k \in \bZ_N} \Psi_\bp(q^{2k}) &= \Psi_\bp(1)^N \prod_{k \in \bZ_N}\wt(\bp|k) \\
    &= \prod_{j=0}^{N-1}(p^- +q^{2j+1}p^+)^j \cdot\prod_{k =1}^{N}\prod_{j=0}^{k-1} (p^- +q^{2j+1}p^+)\\
    &= \prod_{j=0}^{N-1}(p^- +q^{2j+1}p^+)^j \cdot\prod_{j=0}^{N-1} (p^- +q^{2j+1}p^+)^{N-j}\\
    &=1.
\end{align*}
\end{proof}

\subsection{Inversion relation}
Consider the involution on the Fermat curve
\begin{align*}
    \cF_N \to \cF_N, \quad \bp=(p^+,p^-)\mapsto \bp^{-1}:=(p^-,p^+),
\end{align*}
which results in $y \mapsto y^{-1}$ on the ratio. 

\begin{thm}[Inversion relations]\label{thm:inv}
For any $\bp \in \cF_N$ and $k \in \bZ_N$, we have the following:
\begin{enumerate}
    \item $\wt(\bp|k)\cdot\wt(\bp^{-1}|-k) = \gamma(k)$.
    \item $\Psi_\bp(1) \cdot \Psi_{\bp^{-1}}(1) = \zeta_\inv$.
    \item $\Psi_\bp(q^{2k})\cdot \Psi_{\bp^{-1}}(q^{-2k})=\gamma(k)\zeta_\inv$. 
\end{enumerate}
Here $\gamma(k):=q^{k^2}$ and $\zeta_\inv:= e^{\frac{\pi i }{6}(N-\frac 1 N)}$.
\end{thm}

\begin{proof}

(1): 
\begin{align*}
    \wt(\bp|k)\cdot \wt(\bp^{-1}|-k) 
    &= \prod_{j=1}^k(p^- + q^{2j-1}p^+)\cdot \prod_{j=1}^{N-k}(p^+ + q^{2j-1}p^-) \\
    &= \prod_{j=1}^k(p^- + q^{2j-1}p^+)\cdot \prod_{j=1}^{N-k}(p^- + q^{-2j+1}p^+)\cdot \prod_{j=1}^{N-k}q^{2j-1}
\end{align*}
The last scalar term is computed as 
\begin{align*}
    \prod_{j=1}^{N-k}q^{2j-1}=q^{(N-k)(N-k+1)-(N-k)}=q^{(N-k)^2} = q^{k^2}
\end{align*}
by using $q^N=1$. The other two products cancel with each other as 
\begin{align*}
    \prod_{j=1}^k(p^- + q^{2j-1}p^+)\cdot \prod_{j=1}^{N-k}(p^- + q^{-2j+1}p^+)
    &= 
    \prod_{j=1}^k(p^- + q^{2j-1}p^+)\cdot \frac{\prod_{j=1}^{N}(p^- + q^{-2j+1}p^+)}{\prod_{j=N-k+1}^{N}(p^- + q^{-2j+1}p^+)}\\
    &= \frac{\prod_{j=1}^k(p^- + q^{2j-1}p^+)}{\prod_{j=1}^{k}(p^- + q^{-2(N+1-j)+1}p^+)}=1.
\end{align*}
Here we used $\prod_{j=1}^{N}(p^- + q^{-2j+1}p^+)=(p^-)^N+(p^+)^N=1$.

(2): 
\begin{align*}
    \Psi_\bp(1) \cdot \Psi_{\bp^{-1}}(1) 
    &= \prod_{j=0}^{N-1}(p^- + q^{2j+1}p^+)^{j/N} \cdot \prod_{j=0}^{N-1}(p^+ + q^{2j+1}p^-)^{j/N} \\
    &= \prod_{j=0}^{N-1}(p^- + q^{2j+1}p^+)^{j/N} \cdot \prod_{j=0}^{N-1}(p^- + q^{-2j-1}p^+)^{j/N} \cdot \prod_{j=0}^{N-1}(q^{2j+1})^{j/N}. 
\end{align*}
Using $\zeta:=-q=e^{\pi i/N}$, the last scalar term is computed as\footnote{Here, the sign shown in red needs a care. We are taking the standard branch cut $\bR_{<0}$ for the function $\log$. Therefore we need to take $\log(q^{2j+1})=\log(-\zeta^{2j+1}) = \log(\exp(\frac{\pi i}{N}(2j+1)\textcolor{red}{-\pi i}))$ so that $\frac{\pi}{N}(2j+1)-\pi \in (-\pi,\pi)$ for $j=0,\dots,N-1$.}
\begin{align*}
    \prod_{j=0}^{N-1}(q^{2j+1})^{j/N}&= \prod_{j=0}^{N-1}(-\zeta^{2j+1})^{j/N} = \prod_{j=0}^{N-1}(\zeta^{2j+1\textcolor{red}{-N}})^{j/N} = \prod_{j=0}^{N-1}\zeta^{\frac{2j^2-(N-1)j}{N}} \\
    &=\zeta^{\frac 1 3 (N-1)(2N-1)-\frac 1 2 (N-1)^2} = \zeta^{\frac 1 6 (N-1)(N+1)} =e^{\frac{\pi i}{6}(N-\frac 1 N)}.
\end{align*}
The other two products cancel with each other as 
\begin{align*}
    &\prod_{j=0}^{N-1}(p^- + q^{2j+1}p^+)^{j/N} \cdot \prod_{j=0}^{N-1}(p^- + q^{-2j-1}p^+)^{j/N} \\
    &= \prod_{j=0}^{N-1}(p^- + q^{2j+1}p^+)^{j/N} \cdot \prod_{j=0}^{N-1}(p^- + q^{-2(N-1-j)-1}p^+)^{(N-1-j)/N} =1.
\end{align*}
Here we again used 
$\prod_{j=1}^{N}(p^- + q^{-2j+1}p^+)=(p^-)^N+(p^+)^N=1$ and an argument similar to the proof of \cref{lem:q-diff_eq}.

(3): Follows from (1), (2) and \cref{lem:Psi_diagonal}.
\end{proof}

\section{Cyclic representations of the Weyl algebra and the discrete Fourier analysis}\label{sec:Fourier}
Let $V_N$ denote the $N$-dimensional vector space with a basis $\ket{k}$ parameterized by $k \in \bZ_N$. The $q$-Weyl algebra $\cW=\langle U,P \mid UP = q^2 PU \rangle$ acts on $V_N$ by
\begin{align}
    \begin{aligned}\label{eq:std_rep}
    \opU\ket{k}&:= q^{2k}\ket{k}, \\
    \opP\ket{k}&:=\ket{k+1}.
    \end{aligned}
\end{align}
We call \eqref{eq:std_rep} the \emph{standard representation} of $\cW$. In general, an operator $\opX \in \mathrm{End}(V_N)$ satisfying $\opX^N=\mathsf{1}$ is called a \emph{cyclic operator}. Then the eigenvalues of $\opX^N$ are $q^{2k}$, $k \in \bZ_N$. Any cyclic operator can be conjugated into $\opU$. 

\begin{dfn}\label{def:funct_analysis}
For any cyclic operator $\opX$ on $V_N$, choose a basis $\ket{\mb{k}_\opX}$ such that $\opX\ket{\mb{k}_\opX}=q^{2k}\ket{\mb{k}_\opX}$ for $k \in \bZ_N$. Then we define the linear operator $\Psi_\bp(\opX) \in \mathrm{End}(V_N)$ by
\begin{align}\label{eq:funct_analysis}
    \Psi_\bp(\opX)\ket{\mb{k}_\opX} := \Psi_\bp(q^{2k})\ket{\mb{k}_\opX}
\end{align}
for all $k \in \bZ_N$. 
\end{dfn}
Note that the definition \eqref{eq:funct_analysis} does not depend on the choice of eigenvectors of $\opX$. The parameter $\bp \in \cF_N$ must be chosen so that the values $\Psi_\bp(q^{2k})$ are well-defined (\cref{rem:use_of_dilog}).

\subsection{Pentagon relation}

Proofs of the following statements are given in \cref{sec:proof_pentagon}.

\begin{thm}[Faddeev--Kashaev]\label{thm:pentagon}
For any $\bp,\br \in \overline{\cF}_N$, 
the pentagon relation
\begin{align*}
    \Psi_\bp(\opU) \Psi_{\br'}(\opP) = \Psi_\br (\opP) \Psi_{\bp'}(q^{-1}\opU\opP) \Psi_{\bp''}(\opU) 
\end{align*}
holds on the standard representation $V_N$. Here $\bp',\br',\bp'' \in \overline{\cF}_N$ satisfy
\begin{align}
    \begin{aligned}\label{eq:pentagon_parameter_rel}
    &p^- = {p'}^- {p''}^-, \quad p^+{r'}^- = {p'}^-{p''}^+, \quad p^+{r'}^+ = {p'}^+, \\
    &{r'}^- = r^-{p'}^-, \quad {r'}^+{p''}^- = r^+, \quad {r'}^+{p''}^+ = r^-{p'}^+.
    \end{aligned}
\end{align}
\end{thm}

The following gives a new interpretation of these parameter relations: 

\begin{thm}\label{thm:parameter_coeff}
The relations \eqref{eq:pentagon_parameter_rel} are equivalent to the relations
\begin{align}\label{eq:pentagon_coeff_rel}
    \begin{aligned}
    \frac{{r'}^+}{{r'}^-} &= \frac{{r}^+}{{r}^-} (p^-)^{-1}, &
    \frac{{p''}^-}{{p''}^+} &= \frac{{p}^-}{{p}^+} ({r'}^-)^{-1}, &  
    \frac{{p'}^-}{{p'}^+} &= \frac{{r'}^-}{{r'}^+} ({p''}^+)^{-1}, \\
    \frac{{p''}^+}{{p''}^-} &= \frac{{r}^-}{{r}^+} {p'}^+, &
    \frac{{p'}^+}{{p'}^-} &= \frac{{p}^+}{{p}^-} {r}^+.
    \end{aligned}
\end{align}
\end{thm}

The five relations in \eqref{eq:pentagon_coeff_rel} is exactly the mutation relations among the coefficient variables along the pentagon cycle: see \cref{lem:coeff_rel=mutation}. 

\begin{rem}\label{rem:pentagon_ambiguity}
Given initial parameters $\bp,\br \in \cF_N$, the remaining parameters $\bp',\br',\bp''$ are uniquely determined by the relations \eqref{eq:pentagon_coeff_rel} up to the $\bZ_N$-action
\begin{align}\label{eq:pentagon_ambiguity}
    ({p'}^+,{p'}^-, {r'}^+,{r'}^-, {p''}^+,{p''}^-) \mapsto (q^2{p'}^+,q^2{p'}^-, q^{-2}{r'}^+,q^2{r'}^-, q^2{p''}^+,q^{-2}{p''}^-),
\end{align}
as follows. 

From the first and the last relations, we obtain $\pi(\br')$ and $\pi(\bp')$, respectively. Then the third relation gives ${p''}^+$. Using this, the second and the fourth relations determine the products ${p''}^-{r'}^-$ and ${p'}^+{p''}^-$, respectively. If we fix one of $\{{p'}^+,{r'}^-, {p''}^-\}$, say ${p'}^+$, it determines the remaining two. Then we obtain $\pi(\bp'')$ from the second relation. Hence all the remaining parameters are determined. Recall that the ratio $\pi(\bp')$ being determined, the only freedom for the pair $({p'}^+,{p'}^-)$ is the $\bZ_N$-action on the fiber, namely the simultaneous rescaling by $q^2$. Thus the only ambiguity is summarized as in \eqref{eq:pentagon_ambiguity}.
\end{rem}

\subsection{Fourier analysis over $\bZ_N$}
Note that $V_N$ can be viewed as $L^2(\bZ_N)$. 
The following table summarizes the dictionary with the Fourier analysis on $L^2(\bR)$.

\begin{table}[ht]
    \centering
    \begin{tabular}{c|cccc}
    State space & Standard operators & Measure & Fourier kernel & Gaussian kernel \\\hline
    $L^2(\bR)$ & $e^x$, $e^{(2 \pi i)^{-1} \partial/\partial x}$  & $\int_\bR \mathrm{d}x$ & $e^{2\pi ixy}$   & $e^{\pi ix^2}$  \\
    $V_N$     & $\opU$, $\opP$ & $\sum_{k \in \bZ_N}$ & $q^{2k \ell}$ & $\gamma(k)$ 
    \end{tabular}
    \vspace{5mm}
    \caption{Dictionary between the Fourier analysis over $\bR$ and $\bZ_N$.}
    \label{tab:Fourier}
\end{table}

Let $\bra{k} \in V_N^\ast$ denote the dual basis of $\ket{k}$. As is usual in the quantum physics, we denote the dual pairing $V_N^\ast \times V_N \to \bC$ by the Dirac's ``bra-ket'' notation so that $\braket{k}{\ell} = \delta_{k,\ell}$. 

\subsubsection{Diagonalization of the operator $\opP$}

In order to diagonalize the operator $\opP$, we introduce the ``momentum basis''
\begin{align*}
    \ket{\mb{\ell}}:= \sum_{k \in \bZ_N} q^{-2k\ell} \ket{k} \in  V_N
\end{align*}
and the (rescaled) dual basis $\bra{\mb{\ell}}:= \sum_{k \in \bZ_N} q^{2k\ell} \bra{k} \in V^\ast$. 
Then we have
\begin{align*}
    \opP\ket{\mb{\ell}} = \sum_{k \in \bZ_N} q^{-2k\ell} \ket{k+1} = \sum_{k \in \bZ_N} q^{-2(k-1)\ell} \ket{k} = q^{2\ell} \ket{\mb{\ell}}
\end{align*}
as desired. 
We call the coefficient $\braket{\mb{\ell}}{k} =q^{2k\ell}=\overline{\braket{k}{\mb{\ell}}}$ the \emph{Fourier kernel}. They are normalized so that
\begin{align}\label{eq:momentum_normalization}
    \frac 1 N \braket{\mb{m}}{\mb{\ell}}=\delta_{m\ell},\quad   \frac 1 N\sum_{\ell \in \bZ_N} \ket{\mb{\ell}}\bra{\mb{\ell}} = \mathsf{1},
\end{align}
where the latter is an equality in $V_N^\ast \otimes V_N = \mathrm{End}(V_N)$.

The base-change matrix is the Vandermonde matrix $V(q):= (q^{2ij})_{i,j=0}^{N-1}$. 
Its inverse matrix is $\frac 1 N V(q^{-1})$. The following is well-known:

\begin{lem}\label{lem:Vandermonde_det}
$\det V(q^{\pm 1}) = N^{N/2}e^{\pm\pi i \frac{(3N-2)(N-1)}{4}} = N^{N/2}i^{\pm \frac{(3N-2)(N-1)}{2}}$.
\end{lem}
Here note that $(3N-2)(N-1)/2 \in \bZ$.

\subsubsection{The operator $\dbra{\opP\opU^{-1}}$ and its diagonalization}

The following notation will be useful. 
\begin{dfn}[Weyl ordering]\label{def:Weyl}
For $q$-commuting variables/operators $X_1,\dots,X_n$ such that $X_i X_j=q^{2\epsilon_{ij}}X_jX_i$, we define
\begin{align*}
    \dbra{X_1\dots X_n}:= q^{-\sum_{i<j}\epsilon_{ij}}X_1\dots X_n.
\end{align*}
It is designed so that $\dbra{X_{\sigma(1)}\dots X_{\sigma(n)}}=\dbra{X_1\dots X_n}$ for any permutation $\sigma \in \mathfrak{S}_n$. 
\end{dfn}

\begin{lem}\label{lem:Weyl}
If $X_i^N=c_i$ for some $c_i \in \bC^\ast$ for all $i=1,\dots,n$, then $\dbra{X_1\dots X_n}^N=c_1\dots c_n$. 
\end{lem}

\begin{proof}
Observe the associativity $\dbra{X \dbra{YZ}} = \dbra{XYZ}$. 
Therefore it suffices to prove the case $n=2$.
If $X_1X_2=q^{2a}X_2X_1$, we have
\begin{align*}
    \dbra{X_1X_2}^N = q^{-aN} (X_1X_2)(X_1X_2)\dots (X_1X_2).
\end{align*}
Move all the $X_1$ to the left, getting the $q$-factor $\prod_{j=1}^{N-1}q^{2aj}=q^{a(N-1)N}=1$ by $q^N=1$. Therefore $\dbra{X_1X_2}^N = X_1^N X_2^N = c_1c_2$. 
\end{proof}




Recall $\gamma(k):=q^{k^2}$ (cf.~\cref{thm:inv}). 
\begin{lem}\label{lem:slant_diagonal}
Let $\ket{\mbb{k}}:=\sum_{m \in \bZ_m} \gamma(m)^{-1} q^{-2km}\ket{m}$. Then $\dbra{\opP\opU^{-1}}\ket{\mbb{k}}=q^{2k}\ket{\mbb{k}}$. 
\end{lem}

\begin{proof}
\begin{align*}
    q^{-1}\opP\opU^{-1}\ket{\mbb{k}} &= \sum_{m \in \bZ_m}  q^{-m(m+2k)}q^{-1-2m}\ket{m+1} \\
    &=\sum_{m \in \bZ_m} q^{-(m-1)(m+2k-1)-2m+1}\ket{m} \\
    &=\sum_{m \in \bZ_m} q^{-m^2-2km+2k}\ket{m}  =q^{2k}\ket{\mbb{k}}.
\end{align*}
\end{proof}

We also use the covector $\bra{\uuline{\ell}}:=\sum_{m \in \br_m} \gamma(m) q^{2m\ell}\bra{m}$, which is normalized so that
\begin{align}\label{eq:slant_normalization}
    \frac 1 N \braket{\mbb{\ell}}{\mbb{k}}=\delta_{\ell k}, \quad \frac 1 N\sum_{\ell \in \bZ_N} \ket{\mbb{\ell}}\bra{\mbb{\ell}} = \mathsf{1}.
\end{align}

\subsection{Gaussian kernel over $\bZ_N$}\label{subsec:Gaussian}
The factor $\gamma(k)= q^{k^2}$
has appeared in the inversion relations (\cref{thm:inv}) and in the eigenvectors of the operator $\dbra{\opP\opU^{-1}}$. It is the cyclic analogue of the Gaussian kernel $e^{\pi i x^2}$. In terms of Andersen--Kashaev (e.g. \cite{AK-CS}), the abelian group $\bZ_N$ equipped with the function $\gamma:\bZ_N \to U(1)$ is an example of \emph{Gaussian group}. 
We collect here its basic properties.

\begin{rem}
In \cite{AK-CS}, the notion of a \emph{quantum dilogarithm} over any Gaussian group is defined by the inversion relation and the pentagon relation. While our inversion relation (\cref{thm:inv}) fits into their paradigm, the pentagon relation (\cref{thm:pentagon}) deviates as it involves extra parameters. See \cite{Marziori} for a quantum dilogarithm over the Gaussian group $\bR \times \bZ_N$. 
\end{rem}

\begin{lem}
$\gamma(k+N)=\gamma(k)$, 
$\gamma(-k)=\gamma(k)$, $\gamma(k+\ell)=\gamma(k)\gamma(\ell)q^{2k\ell}$.
\end{lem}

The following reciprocity formula is useful to calculate the Fourier transforms of $\gamma(k)^m$:

\begin{thm}[{\cite[Theorem 1.2.2]{BEW}}]\label{thm:Gauss_reciprocity}
For three integers $a,b,c \in \bZ$ such that $ac \neq 0$ and $ac+b$ is even, the generalized Gauss sum
\begin{align*}
    S(a,b,c):=\sum_{k=0}^{|c|-1} e^{\pi i(ak^2+bk)/c}.
\end{align*}
satisfies the reciprocity
\begin{align}\label{eq:reciprocity}
    S(a,b,c) = \left|\frac{c}{a} \right|^{1/2} e^{\frac{\pi i}{4}(\sgn(ac)-\frac{b^2}{ac})} S(-c,-b,a).
\end{align}
\end{thm}

\begin{lem}\label{lem:gamma_Fourier}
We have the Fourier transformation formulae
\begin{align*}
    \sum_{\ell \in \bZ_N}\gamma(\ell)q^{2k \ell} = \zeta_+\cdot \gamma(k)^{-1}, \quad \sum_{\ell \in \bZ_N}\gamma(\ell)^{-1}q^{2k \ell} = \zeta_-\cdot \gamma(k),
\end{align*}
where $\zeta_\pm = N^{1/2} i^{\mp \frac{N-1}{2}}$.
\end{lem}

\begin{proof}
Let $\zeta_\pm:= \sum_{m \in \bZ_N} \gamma(m)^{\pm 1}$. Then we get
\begin{align*}
    \sum_{\ell \in \bZ_N}\gamma(\ell)q^{2k \ell} 
    = \sum_{\ell \in \bZ_N} q^{\ell^2+2k \ell} 
    = \sum_{\ell \in \bZ_N} q^{(\ell+k)^2-k^2} 
    = q^{-k^2}\sum_{\ell \in \bZ_N} q^{\ell^2} 
    = \gamma(k)^{-1}\zeta_+,
\end{align*}
and similarly for the second equation. 
It remains to determine the value
\begin{align}\label{eq:Gauss_rewrite}
    \zeta_\pm = \sum_{k \in \bZ_N} (-e^{\pi i/N})^{\pm k^2} = \sum_{k \in \bZ_N} (-1)^k e^{\pm \pi ik^2/N} = \sum_{k \in \bZ_N} e^{\pm \pi i(k^2+Nk)/N}.
\end{align}
By setting $a=1,b=N,c=\pm N$ in \cref{thm:Gauss_reciprocity}, it is calculated as
\begin{align*}
    \zeta_\pm = S(1,N,\pm N) = N^{1/2} e^{\mp \frac{\pi i}{4}(1-N)} = N^{1/2}i^{\mp \frac{N-1}{2}}.
\end{align*}
\end{proof}

\begin{lem}\label{lem:gamma_prod}
We have
\begin{align*}
    \prod_{k \in \bZ_N} \gamma(k)^{-1}=e^{\frac{\pi i}{6}(N^2-1)}=\zeta_\inv^N.
\end{align*}
\end{lem}

\begin{proof}
\begin{align*}
    \prod_{k \in \bZ_N} \gamma(k)^{-1} 
    = e^{-\frac{\pi i}{N} \sum_{k=0}^{N-1}(k^2-Nk)}=e^{\frac{\pi i}{N} (-\frac{(N-1)N(2N-1)}{6}+\frac{(N-1)N^2}{2})} = e^{\frac{\pi i}{N} \frac{(N-1)N}{6}(3N-(2N-1))} = e^{\frac{\pi i}{6} (N-1)(N+1)}.
\end{align*}
\end{proof}

\begin{lem}\label{lem:mbb-mb}
$\braket{\mbb{k}}{\mb{\ell}}=\zeta_+ \gamma(k-\ell)^{-1}$, $\braket{\mb{k}}{\mbb{\ell}}=\zeta_- \gamma(k-\ell)$.
\end{lem}

\begin{proof}
By using the Fourier transformation formula (\cref{lem:gamma_Fourier}), we obtain
\begin{align*}
    \braket{\mbb{k}}{\mb{\ell}} = \sum_{m \in \bZ_N} \gamma(m)q^{2km} q^{-2m\ell} = \zeta_+ \gamma(k-\ell)^{-1}.
\end{align*}
\end{proof}

\subsection{Square-root and $\gamma$-operators}
Recall our choice $q:=(q^2)^{M}$, where $M:=(N+1)/2$. Note that $M$ is again an integer when $N$ is odd. 

\begin{dfn}[square-root operator]\label{def:square-roots}
For a cyclic operator $\opX$ with eigenvalues $\{q^{2k}\}_{k \in \bZ_N}$, we define $\opX^{1/2} :=\opX^M$. 
\end{dfn}
If $\ket{\mb{k}_\opX}$ is a $q^{2k}$-eigenvector of $\opX$, then we have $\opX^{1/2}\ket{\mb{k}_\opX}=q^k\ket{\mb{k}_\opX}$. The square-roots of the standard operators $\opU,\opP$ are explicitly written on $V_N$ as
\begin{align*}
    \opU^{1/2}\ket{k}=q^k \ket{k}, \quad \opP^{1/2}\ket{k} = \ket{k+M}. 
\end{align*}
Observe that $q^{2M}=q$, and hence $M$ plays the role of $1/2$ in the exponent of $q$. More generally, we can define $\opX^{1/2^n} := \opX^{M^n}$
for $n \in \bZ_{>0}$. 

\begin{dfn}[$\gamma$-operator]\label{def:gamma_op}
For any cyclic operator $\opX$, we define
\begin{align*}
    \gamma(\opX):=\frac 1 N \sum_{i,j \in \bZ_N} q^{-ij}\opX^{\frac{i+j}{2}}.
\end{align*}
\end{dfn}

\begin{lem}\label{lem:gamma_op}
\begin{enumerate}
    \item If $\ket{\mb{k}_\opX}$ is a $q^{2k}$-eigenvector of $\opX$, then $\gamma(\opX)\ket{\mb{k}_\opX}=\gamma(k)\ket{\mb{k}_\opX}$. In particular, $\gamma(\opX)$ is again a cyclic operator. 
    \item The inverse operator is given by $\gamma(\opX)^{-1}=N^{-1} \sum_{i,j \in \bZ_N} q^{ij}\opX^{\frac{i+j}{2}}$. 
    \item If $\opY$ is another operator such that $\opX\opY=q^{2a} \opY\opX$, then $\gamma(\opX)\opY= [\opX^a\opY] \gamma(\opX)$. 
\end{enumerate}
\end{lem}

\begin{proof}
(1) and (2): Follow from the computation
\begin{align*}
    &\gamma(\opX)^{\pm 1}\ket{\mb{k}_\opX} = \frac 1 N \sum_{i,j \in \bZ_N} q^{\mp ij}q^{(i+j)k}\ket{\mb{k}_\opX} = \sum_{j \in \bZ_N} \delta_{j, \pm k} q^{jk}\ket{\mb{k}_\opX} = q^{\pm k^2}\ket{\mb{k}_\opX}. 
\end{align*}
(3): For another operator $\opY$, we get
\begin{align*}
    \gamma(\opX) \opY =  \opY\cdot  \frac 1 N \sum_{i,j \in \bZ_N} q^{-ij}q^{(i+j)a} \opX^{\frac{i+j}{2}}= \opY\cdot  \frac 1 N \sum_{i,j \in \bZ_N} q^{-(i-a)(j-a)} q^{a^2} \opX^{\frac{i+j}{2}} = q^{a^2}\opY \opX^a \cdot \gamma(\opX).
\end{align*}
Observe that $q^{a^2}\opY \opX^a=[\opX^a \opY]$. 
\end{proof}

\begin{lem}\label{lem:Gauss_2}
We have the Fourier transformation formulae
\begin{align*}
    \sum_{k \in \bZ_N} \gamma(k)^2 q^{2k\ell} =  \zeta^{(2)}_+ \gamma(\ell)^{-M}, \quad \sum_{k \in \bZ_N} \gamma(k)^{-2} q^{2k\ell} = \zeta^{(2)}_- \gamma(\ell)^{M}.
\end{align*}
Here $M=(N+1)/2$, 
\begin{align*}
    \zeta^{(2)}_+:=\begin{cases}
        N^{1/2}i & \mbox{if $M$ is even}, \\
        N^{1/2} & \mbox{if $M$ is odd}, 
    \end{cases}
\end{align*}
and $\zeta^{(2)}_-:=(\zeta^{(2)}_+)^\ast$.
\end{lem}
Observe that $\gamma(\ell)^{\pm M}$ plays the role of $\gamma(\ell)^{\pm 1/2}$. 

\begin{proof}
Note that $\sum_{k \in \bZ_N} \gamma(k)^2 q^{2k\ell}=\sum_{k \in \bZ_N} q^{2k(k+\ell)}= \sum_{k=0}^{N-1} e^{\frac{\pi i}{N}(2k^2+2k\ell)}=S(2,2\ell,N)$. 
Then by \eqref{eq:reciprocity}, we get
\begin{align*}
    \sum_{k \in \bZ_N} \gamma(k)^2 q^{2k\ell} &= \big(\frac N 2\big)^{1/2} e^{\frac{\pi i}{4}(1-\frac{2\ell^2}{N})} S(-N,-2\ell,2) \\
    &= N^{1/2} \frac{1+i}{2}e^{-\frac{\pi i}{2N} \ell^2} \sum_{k=0}^1 e^{\frac{\pi i }{2}(-Nk^2-2k\ell)} \\
    &= N^{1/2} \frac{1+i}{2}e^{-\frac{\pi i}{2N} \ell^2} (1+ e^{\frac{\pi i }{2}(-N-2\ell)}) \\
    &= N^{1/2} \frac{1+i}{2}e^{-\frac{\pi i}{2N} \ell^2} (1+ (-1)^{M+\ell}i) \\
    &= N^{1/2} e^{-\frac{\pi i}{2N} \ell^2}\sigma(N;\ell).
\end{align*}
Here 
\begin{align*}
    \sigma(N;\ell) := \frac{(1+i)(1+ (-1)^{M+\ell}i)}{2} = \begin{cases}
        i & \mbox{if $M+\ell$ is even}, \\
        1 & \mbox{if $M+\ell$ is odd}. 
    \end{cases}
\end{align*}

On the other hand, $\gamma(\ell)^M = (-1)^{\ell M} e^{\frac{\pi i}{N}\frac{N+1}{2}\ell^2}=(-1)^{\ell M}i^{\ell^2} e^{\frac{\pi i}{2N}\ell^2}$. The case-by-case consideration on the parities of $\ell$ and $M$ shows that $(-1)^{\ell M}i^{\ell^2}\sigma(N;\ell)= \sigma(N;0)$. See Tables \ref{tab:sign_1}--\ref{tab:sign_3}. Therefore 
\begin{align*}
    \gamma(\ell)^M\sum_{k \in \bZ_N} \gamma(k)^2 q^{2k\ell} = N^{1/2}\sigma(N;0) = \zeta^{(2)}_+,
\end{align*}
proving the first equation. Then the second equation follows by taking the complex conjugate. 
\end{proof}
\begin{table}[ht]
\begin{minipage}{0.3\textwidth}
    \centering
    \caption{$\sigma(N;\ell)$}
    \begin{tabular}{c|c|c}
    \diagbox[height=2em,width=3em]{$\ell$}{$M$}    & even & odd  \\ \hline
     even    & $i$  & $1$ \\ \hline
     odd & $1$ & $i$
    \end{tabular}
    \label{tab:sign_1}
\end{minipage}
\begin{minipage}{0.3\textwidth}
    \centering
    \caption{$(-1)^{\ell M}i^{\ell^2}$}
    \begin{tabular}{c|c|c}
    \diagbox[height=2em,width=3em]{$\ell$}{$M$}    & even & odd  \\ \hline
     even    & $1$ & $1$ \\ \hline
     odd & $i$ & $-i$
    \end{tabular}
    \label{tab:sign_2}
\end{minipage}  
\begin{minipage}{0.3\textwidth}
    \centering
    \caption{Product}
    \begin{tabular}{c|c|c}
    \diagbox[height=2em,width=3em]{$\ell$}{$M$}    & even & odd  \\ \hline
     even    & $i$ & $1$ \\ \hline
     odd & $i$ & $1$
    \end{tabular}
    \label{tab:sign_3}
\end{minipage} 
\end{table}

%% file: 3_Kashaev.tex
\section{Cyclic quantum \Teich\ theory}\label{sec:Kashaev}
In this section, we present the construction of a finite dimensional projective representation of the dotted Ptolemy groupoid $\mathrm{Pt}^{\bullet}_\Sigma$ associated with any marked surface $\Sigma$. 

\subsection{Marked surfaces and their triangulations}\label{subsec:surface}
A marked surface $\Sigma$ is a compact oriented surface equipped with a fixed non-empty finite set $\bM \subset \Sigma$ of \emph{marked points}. 
Let $\Sigma^*:=\Sigma \setminus \bM$, and $\bM_\partial :=\bM \cap \partial\Sigma$. 
We always assume the following conditions:
\begin{enumerate}
    \item[(S1)] Each boundary component (if exists) has at least one marked point.
    \item[(S2)] $-2\chi(\Sigma^\ast)+|\bM_\partial| >0$.
\end{enumerate}
An \emph{ideal triangulation} is an isotopy class of a triangulation $\tri$ of $\Sigma$ whose set of $0$-cells (vertices) coincides with $\bM$. 
The conditions above ensure the existence of such an ideal triangulation. 
For an ideal triangulation $\tri$, denote the set of edges (resp. interior edges, triangles) of $\tri$ by $e(\tri)$ (resp. $e_{\interior}(\tri)$, $t(\tri)$). By a computation of Euler characteristics, we get
\begin{align*}
    &|e(\tri)|=-3\chi(\Sigma^*)+2|\bM_\partial|, \quad |e_{\interior}(\tri)|=-3\chi(\Sigma^*)+|\bM_\partial|, \\
    &|t(\tri)|=-2\chi(\Sigma^*)+|\bM_\partial|.
\end{align*}
A triangle with two of its sides identified is called a \emph{self-folded triangle}. See \cref{fig:self_folded}. 

\begin{figure}[ht]
    \centering
\begin{tikzpicture}
\draw[blue](0,0) -- (0,2);
\draw[blue] (0,0) ..controls (45:0.5) and (1.5,2.5).. (0,2.5);
\draw[blue] (0,0) ..controls (135:0.5) and (-1.5,2.5).. (0,2.5);
\filldraw[fill=white](0,2) circle(2pt);
\fill(0,0) circle(1.5pt);
\node[blue] at (-0.3,1.2) {$\alpha$};
\node[blue] at (1.4,1) {$\pi_\tri(\alpha)$};

\end{tikzpicture}
    \caption{A self-folded triangle.}
    \label{fig:self_folded}
\end{figure}

Given an ideal triangulation $\tri \in \mathrm{Pt}_\Sigma$, define $\ve^\tri=(\ve_{\alpha\beta}^\tri)_{\alpha,\beta \in e(\tri)}$ by the following rule. 
For an edge $\alpha \in e(\tri)$, let $\pi_\tri(\alpha)$
be the edge defined as follows: if $\alpha$ is the interior edge of a self-folded triangle, then $\pi_\tri(\alpha)$ is the encircling edge (see \cref{fig:self_folded}); otherwise $\pi_\tri(\alpha):=\alpha$. 
For a non-self-folded triangle $T \in t(\tri)$ and two edges $\alpha,\beta \in e(\tri)$, let
\begin{align*}
    \ve_{\alpha\beta}(T):= \begin{cases}
        1 & \mbox{if $T$ has $\pi_\tri(\alpha)$ and $\pi_\tri(\beta)$ as its consecutive edges in the clockwise order}, \\
        -1 & \mbox{if the same holds with the counter-clockwise order}, \\
        0 & \mbox{otherwise}.
    \end{cases}
\end{align*}
Then we set $\ve_{\alpha\beta}^\tri:=\sum_T \ve_{\alpha\beta}(T)$, where $T$ runs over all non-self-folded triangles of $\tri$. 

It is known that any marked surface in our consideration admits an ideal triangulation without self-folded triangles \cite[Corollary 3.9]{FST}. Moreover, any two ideal triangulations without self-folded triangles can be connected by a sequence of flips within this class \cite[(6.4)]{LF}.

A \emph{dotted triangulation} $\bD$ of $\Sigma$ consists of an ideal triangulation $\tri$ of $\Sigma$ together with a choice of a corner of each triangle. We call $\tri$ the \emph{underlying triangulation} of $\bD$. 
In figures, the chosen corner is shown by the symbol $\bullet$, and called the \emph{dot}. 

\subsection{Dotted Ptolemy groupoid}
Recall that a \emph{full groupoid} over a set $S$ is the groupoid $\mathcal{G}_S$ where the set of objects is $S$, and there is exactly one morphism $[s,s'] \in \Hom_{\mathcal{G}_S}(s,s')$ for any objects $s,s' \in \mathcal{G}_S$. In particular, $[s,s]=\mathrm{id}_s$ and $[s,s'] \circ [s',s'']=[s,s'']$ for any $s,s',s'' \in \mathcal{G}_S$. Note that the composition of morphisms is read from the left to the right. 

\begin{dfn}[Ptolemy groupoids \cite{Kas98,Penner}]\label{def:Ptolemy_groupoid}
The \emph{Ptolemy groupoid} $\mathrm{Pt}_\Sigma$ of $\Sigma$ is the full groupoid over the set of ideal triangulations of $\Sigma$. 

The \emph{dotted Ptolemy groupoid} (also known as the \emph{Kashaev groupoid}) $\mathrm{Pt}^{\scalebox{0.4}{$\bullet$}}_\Sigma$ is the full groupoid over the set of dotted triangulations of $\Sigma$. There is a functor $\mathrm{Pt}^{\tdot}_\Sigma \to \mathrm{Pt}_\Sigma$, $\bD \mapsto \tri$ forgetting the dots.
\end{dfn}
We are going to mainly deal with the dotted Ptolemy groupoid $\mathrm{Pt}^{\tdot}_\Sigma$. 

\begin{dfn}[elementary morphisms]
We call the following morphisms in $\mathrm{Pt}^{\tdot}_\Sigma$ the \emph{elementary morphisms}.
\begin{enumerate}
    \item We call a morphism $[\bD_1,\bD_2]$ the \emph{dot rotation} at $v \in t(\tri)$ if $\bD_1,\bD_2$ have a common ideal triangulation and their dots only differ on the triangle $v$ as shown in the left of \cref{fig:elem_morph}. We write $A_v=[\bD_1,\bD_2]$ in this case.
    \item We call a morphism $[\bD_1,\bD_2]$ the \emph{flip} at a square $(v,w) \in t(\tri)^2$ if $\bD_1,\bD_2$ have a common ideal triangulation except for the square formed by $v,w$, where they are related as shown in the right of \cref{fig:elem_morph}. We write $T_{vw}=[\bD_1,\bD_2]$ in this case.
\end{enumerate}
\end{dfn}

\begin{figure}[ht]
    \centering
\begin{tikzpicture}
\draw(-30:1) -- (90:1) -- (210:1) --cycle;
\node at (0,0) {$v$};
\dast{(90:0.8)};
\node at (0,-1) {$\bD_1$};
\draw[->] (1,0.3) --node[midway,above]{$A_v$} (2,0.3);
\draw[shift={(3,0)}] (-30:1) -- (90:1) -- (210:1) --cycle;
\node at (3,0) {$v$};
\node at (3,-1) {$\bD_2$};
\dast{(3,0)++(210:0.8)};

\begin{scope}[xshift=6cm,yshift=-0.5cm]
\draw (0,0) -- (1.5,0) -- (1.5,1.5) -- (0,1.5) --cycle; 
\draw (0,1.5) -- (1.5,0);
\node at (0.5,0.5) {$v$};
\node at (1,1) {$w$};
\dast{(0.15,0.15)};
\dast{(1.4,0.25)};
\node at (0.75,-0.5) {$\bD_1$};
\draw[->] (2,0.75) --node[midway,above]{$T_{vw}$} (3,0.75);
\end{scope}
\begin{scope}[xshift=9.5cm,yshift=-0.5cm]
\draw (0,0) -- (1.5,0) -- (1.5,1.5) -- (0,1.5) --cycle; 
\draw (1.5,1.5) -- (0,0);
\node at (0.5,1) {$v$};
\node at (1,0.5) {$w$};
\dast{(0.1,0.25)};
\dast{(1.35,0.15)};
\node at (0.75,-0.5) {$\bD_2$};
\end{scope}
\end{tikzpicture}
    \caption{Elementary morphisms.}
    \label{fig:elem_morph}
\end{figure}

\begin{thm}[\cite{Kas01,Tes07,Kim16}]\label{thm:Ptolemy_presentation}
The dotted Ptolemy groupoid $\mathrm{Pt}^{\tdot}_\Sigma$ is generated by elementary moves. Namely, any morphism can be written as a finite composition of elementary morphisms. Moreover, any relations among them is generated by 
\begin{align*}
    A_v^3=\mathrm{id}, \quad T_{uv}T_{uw}T_{vw}=T_{vw}T_{uv}, \quad A_v T_{vw} A_w = A_w T_{wv} A_v, \quad T_{vw}A_v T_{wv} = A_v A_w \opP_{(vw)}.
\end{align*}
Here, $\opP_{(vw)}$ denotes the permutation of labels $v$ and $w$. 
\end{thm}
See \cref{fig:pentagon_proof} for the illustration of the pentagon relation $T_{uv}T_{uw}T_{vw}=T_{vw}T_{uv}$. We refer the reader to \cite{Kas01} for illustrations of the other relations. 
Strictly speaking, we should consider dotted triangulations with labelings on the triangles. See \cite{Kim16} for a detailed treatment.

Note that under the projection $\mathrm{Pt}^{\tdot}_\Sigma \to \mathrm{Pt}_\Sigma$, the morphism $T_{vw}$ (resp. $A_v$) descends to the ordinary flip of the underlying triangulations (resp. the identity morphism). The latter flip is denoted by $\mu_\kappa: \tri_1 \to \tri_2$, where $\kappa \in e_{\interior}(\tri)$ denotes the unique edge shared by the triangles $v$ and $w$.

\subsection{Global coefficients}
For each ideal triangulation $\tri$, consider the space
\begin{align*}
    \cF_{\tri}:= \prod_{\alpha \in e_{\mathrm{int}}(\tri)} \cF_N \times \prod_{\alpha \in e(\tri) \setminus e_{\mathrm{int}}(\tri)} \overline{\cF}_N.
\end{align*}
We call a point $\bp^\tri=(\bp_{\alpha}^\tri) \in \cF_\tri$ a \emph{coefficient tuple}. For a flip $\mu_\kappa: \tri \to \tri'$, two coefficient tuples $\bp^\tri=(\bp_\alpha) \in \cF_\tri$, $\bp^{\tri'}=(\bp'_\alpha) \in \cF_{\tri'}$ are said to be \emph{mutation-compatible} if they satisfy the relation 
\begin{align}\label{eq:coeff_relation}
    {p'_\kappa}^{\pm} = p_\kappa^{\mp}, \qquad 
    \frac{{p'}^+_\alpha}{{p'}^-_\alpha} &= \begin{cases}
        \displaystyle\frac{p^+_\alpha}{p^-_\alpha} (p_\kappa^+)^{\ve_{\alpha\kappa}} & \mbox{if $\alpha \neq \kappa$ and $\ve_{\alpha\kappa} \geq 0 $}, \vspace{3mm}\\ 
        \displaystyle\frac{p^+_\alpha}{p^-_\alpha}  (p_\kappa^-)^{\ve_{\alpha\kappa}} & \mbox{if $\alpha \neq \kappa$ and $\ve_{\alpha\kappa} <0 $}.
    \end{cases}
\end{align}

\begin{rem}\label{rem:coefficient_orbit}
If $\bp^\tri\in \cF_\tri$ and $\bp^{\tri'} \in \cF_{\tri'}$ are mutation-compatible, so are $\hat\bp^\tri$ and $\bp^{\tri'}$, where the former is
given by 
\begin{align*}
    (\hat{p}_\kappa^+,\hat{p}_\kappa^-):=(p_\kappa^+,p_\kappa^-) \quad \mbox{  and  } \quad (\hat{p}_\alpha^+,\hat{p}_\alpha^-):=(q^{2k}p_\alpha^+,q^{2k}p_\alpha^-), \quad \alpha \neq \kappa.
\end{align*}
In particular, the relation \eqref{eq:coeff_relation} does not uniquely determine $\bp^{\tri}$ from $\bp^{\tri'}$, and vice versa. Our mutation compatibility is closely related to the concept of \emph{seed orbits} in \cite{Fraser} (see also \cref{rem:semifield} below). 
\end{rem}


\begin{dfn}[Global coefficient]\label{def:coeff_space}
A \emph{global coefficient} is a pairwise mutation-compatible collection $\bbp_\Sigma=\{\bp^\tri\}$ of coefficient tuples, one for each $\tri \in \mathrm{Pt}_\Sigma$. 
Let $\cF_\Sigma$ be the set of global coefficients. 
\end{dfn}

\begin{figure}[ht]
    \centering
\begin{tikzpicture}[scale=1.2]
\foreach \i in {0,1,2,3,4} {
\draw(\i*72+90:3) coordinate(A\i);
}
\draw[->,shorten >=1.44cm, shorten <=1.44cm] (A0) -- (A1);
\draw[->,shorten >=1.44cm, shorten <=1.44cm] (A1) -- (A2);
\draw[->,shorten >=1.44cm, shorten <=1.44cm] (A0) -- (A4);
\draw[->,shorten >=1.44cm, shorten <=1.44cm] (A4) -- (A3);
\draw[->,shorten >=1.44cm, shorten <=1.44cm] (A3) -- (A2);
\draw(A0) node[fill=white,circle,inner sep=0.1pt]{
\tikz[scale=0.36]{
\foreach \i in {0,1,2,3,4} {
\fill(\i*72+90:3) circle(1.3pt) coordinate(B\i);
\draw[blue] (\i*72+90:3) -- (\i*72+90+72:3);
\draw(72+90:2.5) node[scale=0.5]{$\ast$};
\draw(144+90:2.5)++(0.3,0.2) node[scale=0.5]{$\ast$};
\draw(216+90:2.5)++(0.25,0.5) node[scale=0.5]{$\ast$};
}
\draw[blue] (B0)--(B2);
\draw[blue] (B0)--(B3);
\node[scale=0.8] at (-2,0){$u$};
\node[scale=0.8] at (0,-0.5){$v$};
\node[scale=0.8] at (2,0){$w$};
\node[blue,scale=0.8] at (1.1,1){$\bp$};
\node[blue,scale=0.8] at (-1.1,1){$\br$};
}};
\draw(A1) node[fill=white,circle,inner sep=0.1pt]{
\tikz[scale=0.36]{
\foreach \i in {0,1,2,3,4} {
\fill(\i*72+90:3) circle(1.3pt) coordinate(B\i);
\draw[blue] (\i*72+90:3) -- (\i*72+90+72:3);
\draw(72+90:2.5) node[scale=0.5]{$\ast$};
\draw(144+90:2.5)++(0.3,0.3) node[scale=0.5]{$\ast$};
\draw(216+90:2.5) node[scale=0.5]{$\ast$};
}
\draw[blue] (B0)--(B2);
\draw[blue] (B2)--(B4);
\node[scale=0.8] at (-2,0){$u$};
\node[scale=0.8] at (0,0){$v$};
\node[scale=0.8] at (0.5,-1.8){$w$};
\node[blue,scale=0.8] at (1.5,0.5){$\bp^{-1}$};
\node[blue,scale=0.8] at (-1.1,1){$\br'$};
}};
\draw(A2) node[fill=white,circle,inner sep=0.1pt]{
\tikz[scale=0.36]{
\foreach \i in {0,1,2,3,4} {
\fill(\i*72+90:3) circle(1.3pt) coordinate(B\i);
\draw[blue] (\i*72+90:3) -- (\i*72+90+72:3);
\draw(72+90:2.5)++(0.5,0.35) node[scale=0.5]{$\ast$};
\draw(144+90:2.5)++(-0.2,0.1) node[scale=0.5]{$\ast$};
\draw(216+90:2.5) node[scale=0.5]{$\ast$};
}
\draw[blue] (B1)--(B4);
\draw[blue] (B2)--(B4);
\node[scale=0.8] at (0,1.5){$u$};
\node[scale=0.8] at (-1,-0.5){$v$};
\node[scale=0.8] at (0.5,-1.8){$w$};
\node[blue,scale=0.8] at (1,0){${\bp''}^{-1}$};
\node[blue,scale=0.8] at (-1.1,0.4){${\br'}^{-1}$};
}};
\draw(A3) node[fill=white,circle,inner sep=0.1pt]{
\tikz[scale=0.36]{
\foreach \i in {0,1,2,3,4} {
\fill(\i*72+90:3) circle(1.3pt) coordinate(B\i);
\draw[blue] (\i*72+90:3) -- (\i*72+90+72:3);
\draw(72+90:2.5)++(0.5,0.4) node[scale=0.5]{$\ast$};
\draw(144+90:2.5) node[scale=0.5]{$\ast$};
\draw(216+90:2.5)++(0.2,0.3) node[scale=0.5]{$\ast$};
}
\draw[blue] (B1)--(B3);
\draw[blue] (B1)--(B4);
\node[scale=0.8] at (0,1.5){$u$};
\node[scale=0.8] at (-1.5,-1){$v$};
\node[scale=0.8] at (-0.5,0){$w$};
\node[blue,scale=0.8] at (1,-1){$\bp''$};
\node[blue,scale=0.8] at (1.5,0.4){${\bp'}^{-1}$};
}};
\draw(A4) node[fill=white,circle,inner sep=0.1pt]{
\tikz[scale=0.36]{
\foreach \i in {0,1,2,3,4} {
\fill(\i*72+90:3) circle(1.3pt) coordinate(B\i);
\draw[blue] (\i*72+90:3) -- (\i*72+90+72:3);
\draw(72+90:2.5)++(0.2,0.1) node[scale=0.5]{$\ast$};
\draw(144+90:2.5) node[scale=0.5]{$\ast$};
\draw(216+90:2.5)++(0.25,0.5) node[scale=0.5]{$\ast$};
}
\draw[blue] (B1)--(B3);
\draw[blue] (B0)--(B3);
\node[scale=0.8] at (-0.5,1){$u$};
\node[scale=0.8] at (-1.5,-1){$v$};
\node[scale=0.8] at (2,0){$w$};
\node[blue,scale=0.8] at (1.1,1){$\bp'$};
\node[blue,scale=0.8] at (-0.5,0){$\br^{-1}$};
}};
\end{tikzpicture}
    \caption{Pentagon relation}
    \label{fig:pentagon_proof}
\end{figure}

The following is easily verified by inspection:

\begin{lem}\label{lem:coeff_rel=mutation}
Under the assignment of coefficient tuples $\bp,\br,\bp',\br',\bp'' \in \cF_N$ as in \cref{fig:pentagon_proof}, the relations \eqref{eq:pentagon_coeff_rel} coincides with the mutation rule \eqref{eq:coeff_relation}. 
\end{lem}

\begin{rem}[coefficient semifield]\label{rem:semifield}
Consider the semifield $\mathbb{P}_N:=\bC^\ast/\bZ_N$, where $\bZ_N$ acts on $\bC^\ast$ by $k.z:=q^{2k}z$ for $k \in \bZ_N$, equipped with the operations
\begin{align*}
    z \oplus w :=(z^N + w^N)^{1/N}, \quad z\cdot w := zw.
\end{align*}
These operations are well-defined on $\mathbb{P}_N$, regardless of the choice of $N$-th roots. Then the Fermat condition $(p^+)^N+(p^-)^N=1$ can be written as $p^+ \oplus p^- =1$. The transformation rule \eqref{eq:coeff_relation} is exactly the mutation formula 
\cite[(1.5)]{FZ-CA2} of coefficients. 

In our theory, however, it is crucial to regard the coefficients $p^\pm$ as points of $\bC^\ast$ rather than their orbits in $\mathbb{P}_N$, since the multiplication by $q^{2k}$ affects the value of $\Psi_{\bp}(X)$.
\end{rem}

\subsubsection*{Projection to the classical cluster Poisson variety}

\begin{lem}[cf. {\cite[Remark 2.7]{FZ-CA4}}]\label{lem:cluster_transf_y}
Take a generic section \eqref{eq:generic_section}
by choosing a branch of the $N$-th root. 
Then the relation \eqref{eq:coeff_relation} implies the relation among the exchange ratios:
\begin{align}
    y'_\alpha&=
    \begin{cases}
        y_{\kappa}^{-1} & \mbox{if $\alpha=\kappa$},\\
        y_{\alpha}\big(1+y_{\kappa}^{-N\sgn (\ve_{\alpha \kappa})}\big)^{-\ve_{\alpha \kappa}/N} & \mbox{if $\alpha\neq \kappa$}.
    \end{cases}\label{eq:cluster_transf_y}
\end{align}
\end{lem}

\begin{proof}
We have $p_\kappa^\pm = s((1+y_\kappa^{\mp N})^{-1/N})$. 
Then from \eqref{eq:coeff_relation}, we get
\begin{align*}
    y'_\alpha =\begin{cases}
        y_\alpha (p_\kappa^+)^{\ve_{\alpha\kappa}} = y_\alpha (1+y_\kappa^{-N})^{-\ve_{\alpha\kappa}/N} & \mbox{if $\ve_{\alpha\kappa}\geq 0$}, \\
        y_\alpha (p_\kappa^-)^{\ve_{\alpha\kappa}} = y_\alpha (1+y_\kappa^{N})^{-\ve_{\alpha\kappa}/N} & \mbox{if $\ve_{\alpha\kappa}< 0$}.
    \end{cases}
\end{align*}
for $\alpha \neq \kappa$.
\end{proof}
Let $\overline{\X}_\Sigma= \bigcup_{\tri \in \mathrm{Pt}_\Sigma} \overline{\X}_\tri$ denote the (partially compactified) cluster $\X$-variety (also known as the cluster Poisson variety) of Fock--Goncharov \cite{FG09}, where $\overline{\X}_\tri\cong (\bC^\ast)^{e_{\interior}(\tri)} \times \bC^{e(\tri)\setminus e_{\interior}(\tri)}$ are algebraic tori parameterized by $\tri \in \mathrm{Pt}_\Sigma$,\footnote{More precisely, we need to include the tori associated with tagged triangulations \cite{FST} to get a genuine cluster variety.} which are partially compactified in the frozen direction. Let $\X_\Sigma \subset \overline{\X}_\Sigma$ be the (non-compactified) cluster Poisson variety. 

\begin{cor}\label{cor:proj_cluster}
The projections $\pi_N^\tri: \cF_\tri \to \overline{\X}_\tri$ given by $(\bp_\alpha^\tri)_\alpha \mapsto (Y_\alpha^\tri)_{\alpha}$ with $Y_\alpha^\tri:=(y_\alpha^\tri)^N$ combine to give a well-defined map
\begin{align}\label{eq:proj_cluster}
    \pi_N: \cF_\Sigma \to \overline{\X}_\Sigma.
\end{align}
In particular, a global coefficient $\bbp_\Sigma \in \cF_\Sigma$ determines a $PSL_2(\bC)$-character $\rho:\pi_1(\Sigma^\ast) \to PSL_2(\bC)$.
\end{cor}

\begin{proof}
From \eqref{eq:cluster_transf_y}, we get
\begin{align}
    Y'_\alpha&=
    \begin{cases}
        Y_{\kappa}^{-1} & \mbox{if $\alpha=\kappa$},\\
        Y_{\alpha}\big(1+Y_{\kappa}^{-\sgn (\ve_{\alpha \kappa})}\big)^{-\ve_{\alpha \kappa}} & \mbox{if $\alpha\neq \kappa$},
    \end{cases}\label{eq:cluster_transf_Y}
\end{align}
which is exactly the cluster Poisson transformation.

It is known \cite[Section 9]{AB} that the cluster Poisson variety $\X_\Sigma$ is identified with a Zariski open subspace of the moduli space $\P_{PGL_2,\Sigma}$ of framed $PSL_2(\bC)$-local systems with pinnings \cite{GS19} (see also \cref{sec:mapping_torus}). Forgetting the data of pinnings, we get a projection $\varpi:\mathcal{P}_{PGL_2,\Sigma} \to \X_{PGL_2,\Sigma}$ to the moduli space of Fock--Goncharov \cite{FG06}. It amounts to forget the frozen coordinates. 
Then we obtain $\varpi(\pi_N(\bbp_\Sigma)) \in \X_{PGL_2,\Sigma}$, which contains the data of a $PSL_2(\bC)$-local system on $\Sigma^\ast$, which is equivalent to a $PSL_2(\bC)$-character.  
\end{proof}


\begin{rem}
Recall from \cref{rem:use_of_dilog} that we need to choose $\bp$ generically so that $q^{2k} \in \mathcal{D}_N(y)$ for all $k \in \bZ_N$. The ``bad'' locus for this choice only depends on $y=\pi(\bp) \in \bC^\ast$, which is a real-analytic subset of $\bC^\ast$. When we vary $\tri \in \mathrm{Pt}_\Sigma$, since the mutation-compatibility is a monomial relation, the union of ``bad'' loci for the choices of $\bp^\tri$ determines an at most countable union of real-analytic subsets in the initial plane $(\bC^\ast)^{e(\tri)}$, which is also easy to avoid by a generic perturbation.   
\end{rem}


\subsection{Representation functor}\label{subsec:functor}
We fix a global coefficient $\bbp_\Sigma \in \cF_\Sigma$. 
For each dotted triangulations $\bD$ of $\Sigma$, we associate the $q$-Weyl algebra $\cW_{\bD}$ generated by $U_v,P_v$ for $v \in t(\tri)$. To simplify the notation, let us label the triangles by $\{1,\dots,T\}$ with $T:=|t(\tri)|$. 

We construct a projective representation of the dotted Ptolemy groupoid $\mathrm{Pt}^{\bullet}_\Sigma$, as follows. 

\begin{description}
\item[State space] For a dotted triangulations $\bD$ of $\Sigma$, we associate the vector space  
    \begin{align*}
        V_{\bD}:= \bigotimes_{v \in t(\tri)} V_N = \bigoplus_{k_1,\dots,k_T \in \bZ_N}\bC \ket{k_1,\dots,k_T},
    \end{align*}
    where $\ket{k_1,\dots,k_T}:=\ket{k_1}\otimes \dots \otimes \ket{k_T}$.\footnote{We will also use a `mixed' basis like $\ket{k_1,\mb{k_2}}=\ket{k_1}\otimes \ket{\mb{k_2}}=\sum_{m_2 \in \bZ_N} q^{-2k_2m_2} \ket{k_1,m_2}$.}
    The Weyl algebra $\cW_{\bD}$ acts on $V_{\bD}$ by
\begin{align*}
    \opU_v\ket{k_1,\dots,k_T}&:= q^{2k_v}\ket{k_1,\dots,k_T}, \\
    \opP_v\ket{k_1,\dots,k_T}&:=\ket{k_1,\dots,k_v+1,\dots,k_T}.
\end{align*}
    In particular, $\opU_v^N=\opP_v^N=\mathsf{1}$ holds. 
    
\item[Quantum coordinate transformations] 
For any morphism $\omega=[\bD_1,\bD_2]$ in $\mathrm{Pt}^{\tdot}_\Sigma$, we contravariantly associate a linear isomorphism $V_\omega: V_{\bD_2} \to V_{\bD_1}$ as follows.
\begin{enumerate}
    \item For a dot rotation $A_v=[\bD_1,\bD_2]$, associate the linear map $\opA_v: V_{\bD_2} \to V_{\bD_1}$ given by
    \begin{align*}
        \opA_v \ket{k_1,\dots,k_T}:=\zeta_A\ket{k_1,\dots,\mbb{k_v},\dots,k_T}
    \end{align*}
    for $k_1,\dots,k_T \in \bZ_N$, where $\zeta_A:=(N\zeta_-)^{-1/3}=N^{-1/2}i^{(1-N)/6}$. Recall the basis $\ket{\mbb{k}}=\sum_{m \in \bZ_m} \gamma(m)^{-1} q^{-2km}\ket{m}$ of $V_N$ from \cref{lem:slant_diagonal}. 
    \item For a flip $t_{vw}=[\bD_1,\bD_2]$, associate the linear map $\opT_{vw}: V_{\bD_2} \to V_{\bD_1}$ given by
    \begin{align}\label{eq:flip_op}
        \opT_{vw}:= \Psi_{\bp_{\alpha}^{\tri_1}}(\dbra{\opP_v^{-1}\opU_v\opP_w})\opS_{vw}, \quad \opS_{vw}:=\frac 1 N \sum_{i,j=0}^{N-1} q^{2ij}\opU_w^i \opP_v^j .
    \end{align}
    Here $\alpha \in e_{\interior}(\tri)$ is the flipped edge.
\end{enumerate}
For a general morphism $\omega=[\bD_1,\bD_2]$, we define $V_\omega:V_{\bD_2} \to V_{\bD_1}$ to be the composite of these elementary maps.
\end{description}
The operator $\opA_v$ is designed so that it satisfies
\begin{align*}
    \opA_v\opU_v \opA_v^{-1} = \dbra{\opP_v \opU_v^{-1}},\quad \opA_v\opP_v \opA_v^{-1} = \opU_v^{-1}
\end{align*} 
and $\opA_v^3=\mathsf{1}$. See \cref{lem:A_comm}. In this sense, $\opA_v$ is a `slant Fourier transform' of order 3.

\begin{thm}\label{thm:AT_relations}
The following relations hold.
\begin{enumerate}
    \item $\opA_v^3=\mathsf{1}$. 
    \item $\opT_{uv}\opT_{uw}\opT_{vw}=\opT_{vw}\opT_{uv}$. 
    \item $\opA_v \opT_{vw} \opA_w =\opA_w \opT_{wv} \opA_v$.
    \item $\opT_{vw}\opA_v \opT_{wv} = \zeta \opA_v \opA_w \opP_{(12)}$, where $\zeta = \zeta_\inv(\zeta_A \zeta_-)^{-1}=e^{\frac{\pi i}{6}(1-\frac 1N)}$. 
\end{enumerate}
In particular, for any global coefficient $\bbp_\Sigma=\{\bp^\tri\}_{\tri \in \mathrm{Pt}_\Sigma} \in \cF_\Sigma$, we get a projective representation $V_\ast(\bbp_\Sigma): \mathrm{Pt}^{\tdot}_\Sigma \to \mathrm{Vect}_\bC^\op$.
\end{thm}
The proof is given in \cref{subsec:AT_relations_proof}. In short, these relations essentially come from (2) the pentagon relation (\cref{thm:pentagon}), (3) a symmetry of $\opS_{vw}$ (\cref{lem:S_symmetry}), and (4) the inversion relation (\cref{thm:inv}), respectively.  

We call the projective functor $V_\ast(\bbp_\Sigma)$ a \emph{cyclic representation} of $\mathrm{Pt}^{\tdot}_\Sigma$ or a \emph{cyclic quantum \Teich\ theory}. 

\begin{rem}\label{rem:module}
Let $\mathrm{Tor}_q$ denote the category where the objects are quantum tori and morphisms are quantum `rational maps', namely algebra homomorphisms between the skew-fields of fractions. The assignment $\bD \mapsto \cW_{\bD}$ gives a functor
\begin{align*}
    \cW_\ast(\bbp_\Sigma): \mathrm{Pt}^{\tdot}_\Sigma \to \mathrm{Tor}_q^\op,
\end{align*}
where elementary morphisms $A_v,T_{vw}$ are send to the adjoint actions of $\opA_v,\opT_{vw}$. 
See \cref{lem:A_comm,prop:T_comm} for explicit formulae for these morphisms. 
The projective representation $V_\ast(\bbp_\Sigma)$ is a $\cW_\ast(\bbp_\Sigma)$-module: each $V_\bD$ is a $\cW_\bD$-module, and compatible with morphisms.     
\end{rem}

%% file: 4_Cluster.tex
\section{Representations of Chekhov--Fock algebras}\label{sec:cluster}
In this section, we introduce the Chekhov--Fock algebra $\X_{\by^\tri}^\tri$ with specified values of central characters, and describe the quantum cluster Poisson transformations among them by using the cyclic quantum dilogarithm. Then we discuss the embedding $\iota_\bD: \X_{\by^\tri}^\tri \to \cW_{\bD}$ determined by a dotted triangulation. Via this embedding, the space $V_{\bD}$ can be viewed as a module over $\X_{\by^\tri}^\tri$, which is reducible. The goal is to describe its irreducible decomposition in a geometric way.



\subsection{Chekhov--Fock algebras and quantum cluster Poisson transformations}\label{subsec:FG_Kashaev}
To each ideal triangulation of a marked surface $\Sigma$, we associate a quantum torus as follows. 
Fix a global coefficient $\bbp_\Sigma=(\bp^\tri)_{\tri \in \mathrm{Pt}_\Sigma} \in \cF_\Sigma$, and let $\by^\tri=(y_\alpha^\tri)$ denote the exchange ratio derived from $\bp^\tri$. 
\begin{dfn}\label{def:CF-alg}
The \emph{Chekhov--Fock algebra} (also known as the \emph{quantum cluster Poisson torus}) is defined to be the quantum torus algebra $\X^\tri$ over $\bC$ with generators $(X_\alpha^\tri)^{\pm 1}$ for $\alpha \in e(\tri)$ and the relations
\begin{align*}
    X_\alpha^\tri X_\beta^\tri = q^{2\ve_{\alpha\beta}^\tri} X_\beta^\tri X_\alpha^\tri.
\end{align*}
Each generator $X_\alpha^\tri$ is called a \emph{cluster Poisson variable}. Note that each $(X_\alpha^\tri)^N$ is a central element. 
Let $\X^\tri_{\by^\tri}$ be the quotient algebra of $\X^\tri$ obtained by further imposing the relation $(X_\alpha^\tri)^N = (y_\alpha^\tri)^N \in \bC^\ast$ for $\alpha \in e(\tri)$.  
\end{dfn}
Note that $\X^\tri$ is an Ore domain (see, for instance, \cite{MR}), while $\X^\tri_{\by^\tri}$ is not. 

For a flip $\mu_\kappa: \tri \to \tri'$ along $\kappa \in e_{\interior}(\tri)$, the exchange matrices $\ve_{\alpha\beta}=\ve^\tri_{\alpha\beta}$ and $\ve'_{\alpha\beta}=\ve^{\tri'}_{\alpha\beta}$ are related by the \emph{matrix mutation rule}
\begin{align}
    \ve'_{\alpha\beta}&=
    \begin{cases}
        -\ve_{\alpha\beta} & \mbox{if $\alpha=\kappa'$ or $\beta=\kappa$,}\\
        \varepsilon_{\alpha\beta}+[\ve_{\alpha\kappa}]_+[\ve_{\kappa\beta}]_+ - [-\ve_{\alpha\kappa}]_+[-\ve_{\kappa\beta}]_+ & \mbox{otherwise},
    \end{cases} \label{eq:matrix_mutation}
\end{align}
where $[a]_+:=\max(a,0)$ for $a \in \bR$. The Chekhov--Fock algebras are connected by `quantum rational maps' as follows. 

\begin{dfn}[quantum cluster Poisson transformation]
We define a $\bC$-algebra isomorphism $\mu_\kappa^\ast: \Frac \X^{\tri'} \xrightarrow{\sim} \Frac \X^\tri$ between the skew-fields of fractions by
\begin{align}
    \mu_\kappa^\ast (X'_\alpha):=
    \begin{cases}
        X_{\kappa}^{-1} & \mbox{if $\alpha=\kappa'$},\\
        X_{\alpha}\displaystyle\prod_{j=1}^{|\ve_{\alpha\kappa}|}\big(1 + q^{2j-1}X_{\kappa}^{-\sgn (\ve_{\alpha \kappa})}\big)^{-\ve_{\alpha \kappa}} & \mbox{if $\alpha\neq \kappa'$}.
    \end{cases}\label{eq:cluster_transf_X}
\end{align}
Here $X'_\alpha=X_\alpha^{\tri'}$ and $X_\alpha=X_\alpha^\tri$, and $\sgn(a) \in \{0,\pm \}$ denotes the sign of $a\in \bR$.
\end{dfn}
These constructions give rise to a functor $\X^\ast: \mathrm{Pt}^{\tdot}_\Sigma \to \mathrm{Tor}_q^\op$.

\begin{conv}\label{conv:flip}
As long as a single triangulation $\tri$ or a single flip $\mu_\kappa: \tri \to \tri'$ matter in the consideration, we drop the superscripts $\tri$ and $\tri'$ and abbreviate as $Z'_\alpha:=Z_\alpha^{\tri'}$, $Z_\alpha:=Z_\alpha^{\tri}$ for any variables/matrices (as we already did).
\end{conv}

For $\alpha \in e(\tri)$, we introduce the \emph{rescaled cluster Poisson variables} 
\begin{align*}
    \overline{X}_\alpha:=y_\alpha^{-1} X_\alpha.
\end{align*}
It satisfies $\overline{X}_\alpha^N=1$ in the quotient algebra $\X^\tri_{\by^\tri}$. 

\begin{lem}\label{lem:cluster_transf_Xbar}
In terms of the rescaled cluster Poisson variables, the map \eqref{eq:cluster_transf_X} is rewritten as
\begin{align}
    \mu_\kappa^\ast(\overline{X}'_\alpha) &= \begin{cases}
        \overline{X}_{\kappa}^{-1} & \mbox{if $\alpha=\kappa'$},\\
        \overline{X}_{\alpha}\displaystyle\prod_{j=1}^{|\ve_{\alpha\kappa}|}(p_\kappa^- + q^{2j-1}p_\kappa^+ \overline{X}_{\kappa})^{-\ve_{\alpha\kappa}} & \mbox{if $\alpha\neq \kappa'$ and $\ve_{\alpha\kappa} < 0$}, \\
        \overline{X}_{\alpha}\displaystyle\prod_{j=1}^{|\ve_{\alpha\kappa}|}(p_\kappa^+ + q^{2j-1}p_\kappa^- \overline{X}_{\kappa}^{-1})^{-\ve_{\alpha\kappa}} & \mbox{if $\alpha\neq \kappa'$ and $\ve_{\alpha\kappa} \geq 0$}.
    \end{cases}\label{eq:cluster_transf_Xbar} 
\end{align}
\end{lem}

\begin{proof}
The case $\alpha=\kappa'$ is obtained from \eqref{eq:cluster_transf_X} by dividing the both sides by ${y'}_\kappa=y_\kappa^{-1}$. The other cases are similarly obtained by the division by $y'_\alpha=y_\alpha (p^{\sgn(\ve_{\alpha\kappa})}_\kappa)^{\ve_{\alpha\kappa}}$. For example if $\ve_{\alpha\kappa}< 0$, then
\begin{align*}
    \mu_\kappa^\ast(\overline{X}'_\alpha) &={y'}_\alpha^{-1}X_{\alpha}\prod_{j=1}^{|\ve_{\alpha\kappa}|}\big(1 + q^{2j-1}X_{\kappa}\big)^{-\ve_{\alpha \kappa}} \\
    &= \overline{X}_\alpha \prod_{j=1}^{|\ve_{\alpha\kappa}|}\big(p_\kappa^- + q^{2j-1}p_\kappa^- y_\kappa\overline{X}_{\kappa}\big)^{-\ve_{\alpha \kappa}} = \overline{X}_\alpha \prod_{j=1}^{|\ve_{\alpha\kappa}|}\big(p_\kappa^- + q^{2j-1}p_\kappa^+\overline{X}_{\kappa}\big)^{-\ve_{\alpha \kappa}}.
\end{align*}
\end{proof}

\begin{rem}
The transformation rule \eqref{eq:cluster_transf_Xbar} coincides with the quantum cluster Poisson transformation with coefficients \cite[Proposition 1.1]{CMM}. The parameter $\bp_\alpha$ (or the ratio $y_\alpha$) plays the dual roles as `coefficients' in the quantum cluster algebra (when we use $\overline{X}_\alpha$) and as the central characters for the cluster Poisson variables (when we use $X_\alpha$). 
Both of these perspectives will be useful.
\end{rem}

\subsubsection*{Fock--Goncharov decomposition at roots of unity}
It is known that the transformation \eqref{eq:cluster_transf_X} can be written by using the compact quantum dilogarithm \cite{FG09} when $q$ is a formal parameter. We are going to give a similar description by using the cyclic quantum dilogarithm in the root of unity case.

For a flip $\mu_\kappa: \tri \to \tri'$ and a sign $\epsilon \in \{+,-\}$, define an monomial isomorphism $\mu'_{\kappa,\epsilon}: \X^{\tri'}\xrightarrow{\sim} \X^\tri$ by
\begin{align*}
    \mu'_{\kappa,\epsilon}(\overline{X}'_\alpha):=
    \begin{cases}
        \overline{X}_\kappa^{-1} & \mbox{if $\alpha=\kappa'$},\\
        \big[\overline{X}_\alpha \overline{X}_\kappa^{[\epsilon \ve_{\alpha\kappa}]_+} \big] & \mbox{otherwise}.
    \end{cases}
\end{align*}


\begin{lem}\label{lem:Xbar_N_power}
The monomial isomorphism $\mu'_{\kappa,\epsilon}$ descends to $\mu'_{\kappa,\epsilon}: \X^{\tri'}_{\by^{\tri'}}\xrightarrow{\sim} \X^\tri_{\by^\tri}$. 
\end{lem}

\begin{proof}
It suffices to prove $(\mu'_{\kappa,\epsilon}(\overline{X}'_\alpha))^N=1$ for $\alpha \in e(\tri')$, which follows from \cref{lem:Weyl}. 
\end{proof}

The following connects the cluster Poisson transformation with the cyclic quantum dilogarithm. 
\begin{prop}\label{lem:mutation_decomp}
The quantum cluster Poisson transformation \eqref{eq:cluster_transf_Xbar} descends to an isomorphism $\mu_\kappa^\ast: \X^{\tri'}_{\by'} \xrightarrow{\sim} \X^\tri_{\by}$. It is decomposed as $\mu_\kappa^\ast= \mu^\#_{\kappa,\epsilon} \circ \mu'_{\kappa,\epsilon}$
for any sign $\epsilon \in \{+,-\}$, where
\begin{align*}
    \mu^\#_{\kappa,\epsilon} := \Ad_{\Psi_{\bp_\kappa^\epsilon}(\overline{X}_\kappa^\epsilon)}^\epsilon 
\end{align*}
is an automorphism on $\X^\tri_{\by}$. 
\end{prop}

\begin{proof}
The first assertion follows from the second assertion and \cref{lem:Xbar_N_power}. 

Observe that the factor $(p_\kappa^- + q p_\kappa^+\overline{X}_\kappa)$ is invertible in $\X^\tri_\by$, since 
\begin{align*}
    1 + y_\kappa^N = (1+qy_\kappa \overline{X}_\kappa) (1- qy_\kappa \overline{X}_\kappa + q^2 (y_\kappa \overline{X}_\kappa)^2 -\dots +(-1)^{N-1}q^{N-1} (y_\kappa \overline{X}_\kappa)^{N-1})
\end{align*}
and $y_\kappa^N \neq -1$ by the Fermat condition. Similarly, each factor $(p_\kappa^- + q^{2j+1} p_\kappa^+\overline{X}_\kappa)$ is invertible. 

Now we verify the composite $\mu_{k,\epsilon} \circ \mu'_{k,\epsilon}$ reproduces the formula \eqref{eq:cluster_transf_Xbar}. 
The case $\alpha=\kappa'$ is clear. Consider the case $\ve_{\alpha\kappa} < 0$ and $\epsilon=+$. Then $\mu'_{\kappa,+}(\overline{X}_\alpha)=\overline{X}_\alpha$, and the right-hand side is computed as
\begin{align*}
    \Ad_{\Psi_{\bp_\kappa}(\overline{X}_\kappa)}(\overline{X}_\alpha) &= \overline{X}_\alpha \Psi_{\bp_\kappa}(q^{2\ve_{\kappa\alpha}}\overline{X}_\kappa)\Psi_{\bp_\kappa}(\overline{X}_\kappa)^{-1} \\
    &= \overline{X}_\alpha(p_\kappa^- + qp_\kappa^+ q^{2(\ve_{\kappa\alpha}-1)} \overline{X}_\kappa) \Psi_{\bp_\kappa}(q^{2(\ve_{\kappa\alpha}-1)}\overline{X}_\kappa)\Psi_{\bp_\kappa}(\overline{X}_\kappa)^{-1} \\
    &= \overline{X}_\alpha(p_\kappa^- + q^{2\ve_{\kappa\alpha}-1} p_\kappa^+ \overline{X}_\kappa)(p_\kappa^- + q^{2\ve_{\kappa\alpha}-3} p_\kappa^+ \overline{X}_\kappa)\dots (p_\kappa^- + q p_\kappa^+\overline{X}_\kappa),
\end{align*}
which coincides with the second equation in \eqref{eq:cluster_transf_Xbar}.  
The other cases are proved similarly. 
\end{proof}



Identify all the skew-fields of fractions $\Frac \X^\tri$ via the quantum cluster Poisson transformations and denote it by $\mathscr{F}$. Then each Chekhov--Fock algebra is a subalgebra of $\mathscr{F}$. 
Define the \emph{quantum cluster Poisson algebra} to be the intersection:
\begin{align*}
    &\cO_q(\X_\Sigma) := \bigcap_{\tri \in \mathrm{Pt}_\Sigma} \X^{\tri} \subset \mathscr{F}. 
\end{align*}
Fixing a global coefficient $\bbp_\Sigma \in \cF_\Sigma$, its specialization $(X_\alpha^\tri)^N=(y_\alpha^\tri)^N$ for all $\tri \in \mathrm{Pt}_\Sigma$, $\alpha \in e(\tri)$ is denoted by $\cO_q(\X_\Sigma;\bbp_\Sigma)$. We have an isomorphism $\cO_q(\X_\Sigma;\bbp_\Sigma) \cong \X_{\by^\tri}^\tri$ for each $\tri \in \mathrm{Pt}_\Sigma$. 



\begin{rem}
The algebra $\cO_q(\X_\Sigma;\bbp_\Sigma)$ is the same as the specialization $q=-e^{i\pi/N}$ of $\cO_q(\X_\Sigma)$, which was introduced in \cite[Section 3.3]{FG09}. Closely related is the \emph{polynomial core} of the \emph{quantum \Teich\ space} \cite{BL,BBL}, where the latter was formulated as the family $\{\X^\tri\}_{\tri \in \mathrm{Pt}_\Sigma}$. 
\end{rem}

\subsection{Embedding into the Weyl algebra}

We relate the algebra structures in the previous subsection to the one developed in \cref{sec:Kashaev}. 
Given a decorated triangulation $\bD \in \mathrm{Pt}^{\tdot}_\Sigma$ with the underlying triangulation $\tri$, we construct an algebra embedding 
\begin{align*}
    \iota_{\bD}: \X^\tri_{\by^\tri} \to \cW_{\bD}
\end{align*}
as follows (cf. \cite[Section 7.2]{Tes07}). Let $\tau$ be the dual trivalent graph of $\tri$. Since each vertex $v$ of $\tau$ corresponds to a triangle of $\tri$, one of the corners around $v$ inherits the decoration from $\bD$. 
We assign the \emph{Weyl weight} $W_{v;\alpha}^{\bD}$ to a half-edge $\alpha$ incident to a vertex $v$ as shown in \cref{fig:weight}. In particular, $W_{v;\alpha}^{\bD}$ acts on the $v$-th component of $V_{\bD}$ as a non-commutative Laurent monomial of $U_v$ and $P_v$.

\begin{figure}[ht]
    \centering
\begin{tikzpicture}
\foreach \i in {30,150,270} \draw(0,0) -- (\i:1);
\node[above] at (30:1) {$P_v$};
\node[above] at (150:1) {$U_v^{-1}$};
\node[below] at (270:1) {$\dbra{U_v P_v^{-1}}$};
\node[below right] at (0,0) {$v$};
\dast{(0,0.2)};
\end{tikzpicture}
    \caption{Weyl weights around a vertex $v$ of the dual graph {$\tau$}.}
    \label{fig:weight}
\end{figure}

\begin{lem}\label{lem:Weyl_embedding}
There is a unique algebra embedding $\iota_{\bD}: \X^\tri_{\by^\tri} \to \cW^{\bD}$ satisfying the following conditions. 
\begin{itemize}
    \item For a non-self-folded edge $\alpha \in e(\tri)$, 
    \begin{align*}
        \iota_{\bD}(\overline{X}_\alpha^\tri) = W_{v;\alpha}^{\bD} W_{w;\alpha}^{\bD}, 
    \end{align*}
    where $v,w$ denote the endpoints of the edge in $\tau$ dual to $\alpha$. 
\item For a self-folded edge $\alpha \in e(\tri)$, $\iota_{\bD}(\overline{X}_\alpha^\tri)$ is the Weyl-ordered product of the two factors associated to $\alpha$ and $\pi_\tri(\alpha)$ by the rule above.\footnote{This product rule resembles \cite[Definition 9.2 (ii)]{AB}.} 
\end{itemize}
\end{lem}

\begin{proof}
We drop the notations $\tri$ and $\bD$. 
The condition $\iota(\overline{X}_\alpha)^N =1$ holds by \cref{lem:Weyl}. 
In order to check the $q$-commutation relation, 
fix $v$ and let $\alpha_1,\alpha_2,\alpha_3$ be the three half-edges incident to $v$ in this counter-clockwise order so that $W_{v;\alpha_1}=P_v$, $W_{v;\alpha_2}=U_v^{-1}$, $W_{v;\alpha_3}=\dbra{U_vP_v^{-1}}$. Then it is easy to see that
\begin{align*}
    W_{v;\alpha_i}W_{v;\alpha_{i+1}} =q^2 W_{v;\alpha_{i+1}}W_{v;\alpha_{i}}
\end{align*}
for $i =1,2,3$, considered modulo $3$. With a notice that the Weyl weights assigned to distinct vertices commute with each other, it follows that
\begin{align*}
    \iota(\overline{X}_\alpha)\iota(\overline{X}_\beta) =q^{2\ve_{\alpha\beta}} \iota(\overline{X}_\beta)\iota(\overline{X}_\alpha)
\end{align*}
for all $\alpha,\beta \in e(\tri)$. Thus the assertion is proved. See \cref{fig:self_folded_Weyl} for a few examples. 
\end{proof}

\begin{figure}[ht]
    \centering
\begin{tikzpicture}[scale=0.9]
\draw (0,0) -- (4,0) -- (4,4) -- (0,4) --cycle; 
\draw (0,4) -- (4,0); 
\draw[red](1,1) --++(-90:1.6) node[left,scale=0.8]{$\opU_v^{-1}$};
\draw[red](1,1) node[right,scale=0.8]{$\dbra{\opU_v\opP_v^{-1}}$} --++(180:1.6) node[above,scale=0.8]{$\opP_v$};
\draw[red](3,3) --++(0:1.6) node[above,scale=0.8]{$\opU_w^{-1}$};
\draw[red](3,3)++(-0.2,-0.2) node[left,scale=0.8]{$\opP_w$};
\draw[red](3,3) --++(90:1.6) node[right,scale=0.8]{$\dbra{\opU_w\opP_w^{-1}}$};
\draw[red](1,1) -- (3,3);
\node[above left,scale=0.8] at (1,1) {$v$};
\node[below right,scale=0.8] at (3,3) {$w$};
\dast{(0.3,0.3)};
\dast{(3.8,0.5)};
\node at (1.4,3){$\kappa$};
\node[left] at (0,2){$\alpha$};
\node[below] at (2,0){$\beta$};
\node[right] at (4,2){$\gamma$};
\node[above] at (2,4){$\delta$};
\begin{scope}[xshift=8cm]
\draw(0,0) -- (0,2);
\draw (0,0) ..controls (45:1) and (2,3).. (0,3);
\draw (0,0) ..controls (135:1) and (-2,3).. (0,3);
\draw (0,2) circle(2cm);
\fill(0,0) circle(2pt);
\fill(0,2) circle(2pt);
\fill(0,4) circle(2pt);
\node at (-0.2,0.8) {$\alpha$};
\node at (-1.4,2) {$\beta$};
\node[red,above right,scale=0.8] at (0,2.4) {$\opU_v^{-1}$};
\node[red,scale=0.8] at (0.6,1.5) {$\opP_v$};
\node[red,scale=0.8] at (-1.0,1.5) {$[\opU_v \opP_v^{-1}]$};
\node[red,below right,scale=0.8] at (0,3.6) {$[\opU_w \opP_w^{-1}]$};
\draw[red] (0,2) circle(0.5cm);
\draw[red] (0,2.5) -- (0,3.5);
\draw[red] (0,3.5) --++(30:1.5) node[red,right,scale=0.8]{$\opP_w$};
\draw[red] (0,3.5) --++(150:1.5) node[red,left,scale=0.8]{$\opU_w^{-1}$};
\dast{(0,3.7)};
\dast{(0.2,0.4)};
\end{scope}
\end{tikzpicture}
    \caption{Examples of Weyl weights. Left: $\iota(\overline{X}_\kappa)=[\opU_v\opP_v^{-1}\opP_w]$. Right: $\iota(\overline{X}_\alpha)=\opU_w \opP_w^{-1}$ and $\iota(\overline{X}_\beta)=[\opU_v^{-1}\opU_w\opP_w^{-1}]$.}
    \label{fig:self_folded_Weyl}
\end{figure}

\begin{thm}[Compatibility of transformations, cf.~\cite{GuoLiu}]\label{thm:K_FG}
In the decorated Ptolemy groupoid of any marked surface $\Sigma$, we have the following commutative diagrams:
\begin{enumerate}
    \item For any dot rotation $A_v=[\bD_1,\bD_2]$, 
    \begin{equation*}
    \begin{tikzcd}
        \X^\tri_{\by^\tri} \ar[r,"\iota_{\bD_2}"] \ar[d,phantom,"=",sloped] & \cW_{\bD_2} \ar[d,"\Ad({\opA_v})"] \\
        \X^\tri_{\by^\tri} \ar[r,"\iota_{\bD_1}"'] & \cW_{\bD_1},
    \end{tikzcd}
    \end{equation*}
    where $\tri$ is the common underlying triangulation. 
    \item  For any flip $T_{vw}=[\bD_1,\bD_2]$,
    \begin{equation*}
    \begin{tikzcd}
        \X^{\tri_2}_{\by^{\tri_2}} \ar[r,"\iota_{\bD_2}"] \ar[d,"\mu_\kappa^\ast"'] & \cW_{\bD_2} \ar[d,"\Ad(\opT_{vw})"] \\
        \X^{\tri_1}_{\by^{\tri_1}} \ar[r,"\iota_{\bD_1}"'] & \cW_{\bD_1}.
    \end{tikzcd}
    \end{equation*}
    Here $\mu_\kappa: \tri_1 \to \tri_2$ denotes the flip of the underlying triangulation.
\end{enumerate}
\end{thm}

\begin{figure}[ht]
    \centering
\begin{tikzpicture}[scale=1.2]
\begin{scope}
\draw (0,0) -- (2,0) -- (2,2) -- (0,2) --cycle; 
\draw (0,2) -- (2,0); 
\draw[red](0.5,0.5) --++(-90:0.8) node[below,scale=0.7]{$\opU_v^{-1}$};
\draw[red](0.5,0.5) node[right,scale=0.7]{$\dbra{\opU_v\opP_v^{-1}}$} --++(180:0.8) node[left,scale=0.7]{$\opP_v$};
\draw[red](1.5,1.5) --++(0:0.8) node[right,scale=0.7]{$\opU_w^{-1}$};
\draw[red](1.5,1.5)++(-0.2,-0.2) node[left,scale=0.7]{$\opP_w$};
\draw[red](1.5,1.5) --++(90:0.8) node[above,scale=0.7]{$\dbra{\opU_w\opP_w^{-1}}$};
\draw[red](0.5,0.5) -- (1.5,1.5);
\node[above left,scale=0.8] at (0.5,0.5) {$v$};
\node[below right,scale=0.8] at (1.5,1.5) {$w$};
\dast{(0.15,0.15)};
\dast{(1.9,0.25)};
\node[scale=0.8] at (0.7,1.5){$\kappa$};
\node[left,scale=0.8] at (0,1){$\alpha$};
\node[below,scale=0.8] at (1,0){$\beta$};
\node[right,scale=0.8] at (2,1){$\gamma$};
\node[above,scale=0.8] at (1,2){$\delta$};
\node at (1,-0.8) {$\bD_1$};
\draw[->] (3,0.75) --node[midway,above]{$T_{vw}$} (4,0.75);
\end{scope}
\begin{scope}[xshift=5cm]
\draw (0,0) -- (2,0) -- (2,2) -- (0,2) --cycle; 
\draw (2,2) -- (0,0);
\draw[red](1.5,0.5) --++(0:0.8) node[right,scale=0.7]{$\opU_w^{-1}$};
\draw[red](1.5,0.5) node[above=0.7em,scale=0.7]{$\dbra{\opU_w\opP_w^{-1}}$} --++(-90:0.8) node[below,scale=0.7]{$\opP_w$};
\draw[red](0.5,1.5) --++(180:0.8) node[left,scale=0.7]{$\opP_v$};
\draw[red](0.5,1.5) --++(90:0.8) node[above,scale=0.7]{$\dbra{\opU_v\opP_v^{-1}}$};
\draw[red](0.5,1.5)++(0.2,-0.2) node[below=0.3em,scale=0.7]{$\opU_v^{-1}$};
\draw[red](0.5,1.5) -- (1.5,0.5);
\node[above right,scale=0.8] at (0.5,1.5) {$v$};
\node[below left,scale=0.8] at (1.5,0.5) {$w$};
\dast{(0.1,0.25)};
\dast{(1.85,0.15)};
\node[scale=0.8] at (1.3,1.6){$\kappa'$};
\node[left,scale=0.8] at (0,1){$\alpha$};
\node[below,scale=0.8] at (1,0){$\beta$};
\node[right,scale=0.8] at (2,1){$\gamma$};
\node[above,scale=0.8] at (1,2){$\delta$};
\node at (1,-0.8) {$\bD_2$};
\end{scope}

\end{tikzpicture}
    \caption{A flip and the associated Weyl weights.}
    \label{fig:flip_proof}
\end{figure}

\begin{proof}
(1): From \cref{lem:A_comm}, we get $\opP_v \xmapsto{\Ad_{\opA_v}} \opU_v^{-1} \xmapsto{\Ad_{\opA_v}} \dbra{\opU_v\opP_v^{-1}} \xmapsto{\Ad_{\opA_v}} \opP_v$. 
Then the assertion immediately follows. 

(2): In order to simplify the notation, let us omit the symbol $\iota_{\bD_k}$ and regard $\X^{\tri_k}_{\by^{\tri_k}} \subset \cW_{\bD_k}$ as a subalgebra for $k=1,2$. We also employ the abbreviation in \cref{conv:flip}. 

Recall that $\opT_{vw}=\Psi_{\bp_\kappa}(\dbra{\opP_v^{-1}\opU_v\opP_w})\opS_{vw}$. Since $\dbra{\opP_v^{-1}\opU_v\opP_w}=\overline{X}_\kappa$, the adjoint action $\Ad(\Psi_{\bp_\kappa}(\dbra{\opP_v^{-1}\opU_v\opP_w}))$ restricts to 
the automorphism part $\mu_{\kappa,+}^\#$ in  \cref{lem:mutation_decomp}. Therefore it suffices to show that $\Ad(\opS_{vw})$ restricts to the monomial part $\mu'_{\kappa,+}$. Under the labeling of \cref{fig:flip_proof}, using \eqref{eq:S_comm}, the adjoint action of $\opS_{vw}$ is computed as 
\begin{align*}
    &\overline{X}'_\kappa= \opU_v^{-1}\dbra{\opU_w\opP_w^{-1}} &&\mapsto\ (\opU_v\opU_w)^{-1}\dbra{\opU_w(\opP_w\opP_v^{-1})^{-1}} = \dbra{\opU_v^{-1}\opP_v}\opP_w^{-1} = \overline{X}_{\kappa}^{-1}, \\
    &\overline{X}'_\alpha= \opP_v &&\mapsto\ \opP_v = \overline{X}_{\alpha}, \\
    &\overline{X}'_\beta= \opP_w &&\mapsto\ \opP_w\opP_v^{-1} = q\dbra{\opU_v\opP_v^{-1}}\opP_w \cdot \opU_v^{-1} = q \overline{X}_{\kappa}\overline{X}_{\beta}, \\
    &\overline{X}'_\gamma= \opU_w^{-1} &&\mapsto\ \opU_w^{-1} = \overline{X}_{\gamma}, \\
    &\overline{X}'_\delta= \dbra{\opU_v\opP_v^{-1}} &&\mapsto\ \dbra{(\opU_v\opU_w)\opP_v^{-1}} =q\dbra{\opU_v\opP_v^{-1}}\opP_w \cdot \dbra{\opU_w\opP_w^{-1}} = q \overline{X}_{\kappa}\overline{X}_{\delta}.
\end{align*}
Here, note that the Weyl weights coming from outside the square are omitted, since they are obviously invariant under $\opS_{vw}$.  
Thus we get the desired assertion. 
\end{proof}

In particular, the Chekhov--Fock algebra $\X_{\by^\tri}^\tri$ acts on $V_{\bD}$ through $\iota_{\bD}$. Thus we have:

\begin{cor}\label{cor:cluster_module}
The projective representation $V_\ast$ has a structure of $\widetilde{\X}^\ast$-module, where $\widetilde{\X}^\ast: \mathrm{Pt}^{\tdot}_\Sigma \to \mathrm{Tor}_q^\op$ is the composite functor
\begin{align*}
    \widetilde{\X}^\ast: \mathrm{Pt}^{\tdot}_\Sigma \to \mathrm{Pt}_\Sigma \xrightarrow{\X^\ast} \mathrm{Tor}_q^\op.
\end{align*}
\end{cor}

\begin{rem}
As the flip operator \eqref{eq:flip_op} is a direct analogue of the tetrahedron operator in the Andersen--Kashaev TQFT (see the right-most side of \cite[(10)]{AK-TQFT}) and the embedding $\iota_{\bD}$ is compatible with it, we call $\iota_{\bD}$ the \emph{Andersen--Kashaev realization}. In \cref{sec:invariant}, we will compare it with another realization used in the Baseilhac--Benedetti's QHFT \cite{BB-GD}.
\end{rem}

\subsection{Relation with the local representations of Bai--Bonahon--Liu}\label{subsec:BBL}
In terms of Bai--Bonahon--Liu \cite{BBL}, $V_{\bD}=\bigotimes_{T \in t(\tri)} V_N$ is a local representation of the Chekhov--Fock algebra $\X^\tri_{\by^\tri}$. In view of the classification result \cite[Proposition 6]{BBL}, the classifying data of $V_{\bD}$ are:
\begin{itemize}
    \item the central characters $(y_\alpha^{\tri})^N$ for $\alpha \in e(\tri)$,
    \item the central load $h=\prod_{\alpha \in e(\tri)} y_\alpha^{\tri}$.
\end{itemize}
In particular, our representations realize all the isomorphism classes of local representations when we vary $\bp^\tri$. Our particular realizations are specified by the data of dots of $\bD$. 

All the representations $V_{\bD}$ for $\bD \in \mathrm{Pt}^{\tdot}_\Sigma$ together yield a local representation of the quantum \Teich\ space $\mathcal{T}_\Sigma^q$, in the sense of \cite[Section 5]{BBL}. The required consistency condition is satisfied since $1+\opX$ is invertible for any cyclic operator $\opX$. Again, these representations realize all the isomorphism classes of local representations when we vary $\bbp_\Sigma$. 

It has turned out that an appropriate selection of an intertwiner between two local representations is a subtle problem, as pointed out and carefully analyzed by Mazzoli \cite{Mazzoli}. Mazzoli defined a finite collection of intertwiners canonically associated to a pair of local representations, where the redundancy of the choice is represented by $H_1(\Sigma;\bZ_N)$ (and powers of $q$). In our approach, we constructed the intertwiners $V_\omega$ among our realizations of local representations in a canonical way, which is independent of the choice of paths $\omega$ in $\mathrm{Pt}^{\tdot}_\Sigma$ up to powers of $q$.  

Baseilhac--Benedetti \cite{BB-GD} show that the transpose of the \emph{reduced quantum hyperbolic operator} (\emph{reduced QH operator} for short) associated with the mapping cylinder of $\phi$ gives an intertwiner of local representations \cite[Theorem 1.1 (1)]{BB-GD}. We will see in \cref{sec:invariant} that this gives the same intertwiner.


\subsection{Irreducible decomposition of the representation $V_\ast(\bbp_\Sigma)$}
The representation space $V_{\bD}$ is typically reducible as an $\X_\by^\tri$-module. In order to obtain an irreducible decomposition, we introduce specific operators which commute with the Chekhov--Fock algebra (cf. \cite[(12)]{Kas98}, \cite[(7.3)]{Tes07} in the infinite-dimensional setting). 

\subsubsection*{Homology operators}
Let us consider the relative homology group
\begin{align*}
    H_\Sigma:= H_1(\Sigma^\ast,\partial\Sigma^\ast;\bZ_N).
\end{align*}
Here recall $\Sigma^*=\Sigma \setminus \bM$, and hence $\partial \Sigma^\ast=\partial\Sigma \setminus \bM_\partial$. 

\begin{dfn}[Homology operator]
Define a monomial map $\oph^{\bD}: H_\Sigma \to \cW_{\bD}$ as follows. For a curve $c$ representing a homology class $[c]\in H_\Sigma$, deform it to an efficient edge path $(\alpha_1,v_1,\alpha_2,\dots,v_{m-1},\alpha_m,v_m)$ on the dual graph $\tau$, where $\alpha_i$ are edges and $v_i$ are vertices traversed in this order. Then we define
\begin{align*}
    \oph^{\bD}_{[c]}:=\bigg[\prod_{i=1}^{m} (W_{v_i;\alpha_i}^{\bD}W_{v_i;\alpha_{i+1}}^{\bD})^{\sigma_i}\bigg],
\end{align*}
where $\alpha_{m+1}:=\alpha_1$ and $\sigma_i \in \{\pm 1\}$ is $+1$ (resp. $-1$) if $\alpha_i \to \alpha_{i+1}$ forms a clockwise (resp. counter-clockwise) turn at $v_i$. Extend this assignment multiplicatively: 
\begin{align*}
    \oph_{[c_1]+[c_2]}^{\bD}:=\big[\oph_{[c_1]}^{\bD}\cdot \oph_{[c_2]}^{\bD}\big]
\end{align*}
for any homology classes $[c_1],[c_2] \in H_\Sigma$. 
\end{dfn}
We have the following properties of the homology operators. The proofs are just a matter of linear algebra, and identical to the corresponding statements in the infinite-dimensional setting \cite{Kas98,Tes07} (with a straightforward generalization to the marked surfaces with boundary).


\begin{lem}[intersection number, cf. {\cite[Lemma 2 (ii)]{Tes07}}]\label{lem:homology_intersection}
The map $\oph^{\bD}$ is well-defined, and satisfies
\begin{align*}
    \oph^{\bD}_{[c_1]}\oph^{\bD}_{[c_2]} = q^{4\boldsymbol{i}([c_1],[c_2])} \oph^{\bD}_{[c_2]}\oph^{\bD}_{[c_1]}
\end{align*}
for $[c_1], [c_2] \in H_\Sigma$, where $\boldsymbol{i}:H_\Sigma \times H_\Sigma \to \frac 1 2\bZ_N$ denotes an extension of the homology intersection pairing, due to \cite[(2.1)]{KQ22}. 
\end{lem}

\begin{lem}[naturality]\label{lem:homology_naturality}
For any morphism $\omega: \bD_1 \to \bD_2$ in $\mathrm{Pt}^{\tdot}_\Sigma$, we have the commutative diagram
\begin{equation*}
\begin{tikzcd}
\cW_{\bD_2} \ar[r,"\mathrm{Ad}_{V_\omega}"] & \cW_{\bD_1} \\
H_\Sigma \ar[u,"\oph^{\bD_2}"] \ar[ur,"\oph^{\bD_1}"'] &
\end{tikzcd}
\end{equation*}
\end{lem}

\begin{lem}[mapping class group equivariance]\label{lem:homology_MCG}
For any mapping class $\phi \in MC(\Sigma)$ (see \cref{sec:trace_local} below), we have the commutative diagram
\begin{equation*}
\begin{tikzcd}
    \cW_{\phi^{-1}(\bD)} \ar[r,"\phi_\ast"] & \cW_\bD \\
    H_\Sigma \ar[u,"\oph^{\phi^{-1}(\bD)}"] \ar[r,"\phi"'] & H_\Sigma \ar[u,"\oph^{\bD}"'].
\end{tikzcd}
\end{equation*}
Here $\phi_\ast(U_{\phi^{-1}(v)})=U_v$, $\phi_\ast(P_{\phi^{-1}(v)})=P_v$ for $v \in t(\tri)$. 
\end{lem}

For a puncture $p \in \bM \setminus \bM_\partial$, let $[c_p] \in H_\Sigma$ denote the class represented by a small loop around $p$ in the counter-clockwise direction. For a boundary component $b$ with marked points $m_1,\dots,m_n$, let $[c_{m_i}] \in H_\Sigma$ denote the class represented by a small corner arc at $m_i$ turning counter-clockwisely around $m_i$. See \cref{fig:boundary_sum}. Define $[c_b]:=\sum_{i=1}^n [c_{m_i}]$. 

\begin{figure}[ht]
    \centering
\begin{tikzpicture}
    \filldraw[gray!20](0,0) -- (0,-0.2) -- (4,-0.2) -- (4,0) --cycle;
    \node at (2,-0.4) {\scriptsize $m_i$};
    \foreach \i in {1,2,3} \draw[red,thick,->-] (\i+0.3,0) arc(0:180:0.3);
    \node[red] at (2,0.7) {$c_{m_i}$};
    \draw[thick] (0,0) -- (4,0);
    \foreach \i in {1,2,3} \filldraw[fill=white] (\i,0) circle(2pt);
\end{tikzpicture}
    \caption{The homology classes $[c_{m_i}]$.}
    \label{fig:boundary_sum}
\end{figure}


\begin{lem}[relationship with cluster Poisson variables, cf. {\cite[Lemma 4 (iii), (iv)]{Tes07}}]\label{lem:Casimir}
We have
\begin{align*}
    \iota_{\bD}(X_\alpha^\tri)\cdot \oph_{[c]}^{\bD} = \oph_{[c]}^{\bD} \cdot \iota_{\bD}(X_\alpha^\tri) 
\end{align*}
for any $\alpha \in e(\tri)$ and $[c] \in H_\Sigma$. Moreover, 
\begin{align*}
    \iota_{\bD}(\overline{\theta}_p^\tri)=\oph^{\bD}_{[c_p]}, \quad \iota_{\bD}(\overline{\theta}_b^\tri)=\oph^{\bD}_{[c_b]}
\end{align*}
for any puncture $p$ and boundary component $b$. Here $\theta_p,\theta_b \in Z(\X_\by^\tri)$ are defined by
\begin{align*}
    \overline{\theta}_p:= \bigg[\prod_{\alpha \in e(\tri)}(\overline{X}_\alpha^\tri)^{i_p(\alpha)}\bigg], \quad \overline{\theta}_b:= \bigg[\prod_{\alpha \in e(\tri)}(\overline{X}_\alpha^\tri)^{i_b(\alpha)}\bigg],
\end{align*}
where $i_p(\alpha)$ (resp. $i_b(\alpha)$) counts the number of ends of $\alpha$ incident to the puncture $p$ (resp. incident to a marked point on $b$).
\end{lem}


Let 
\begin{align*}
    \overline{H}:=\bigg[\prod_{\alpha \in e(\tri)} \overline{X}_\alpha^\tri \bigg] \in \X_\by^\tri
\end{align*}
denote the total product of rescaled cluster Poisson variables, which also lies in the center of $\X_\by^\tri$.

\begin{lem}[special central element]\label{lem:special_center}
We have 
\begin{align*}
    \overline{H}^2=\prod_p \overline{\theta}_p \cdot \prod_b \overline{\theta}_b,
\end{align*}
where in the right-hand side, $p$ (resp. $b$) runs over all punctures (resp. boundary components). Moreover, $\iota_{\bD}(\overline{H})=1$. 
\end{lem}

\begin{rem}\label{rem:central_load}
We have the element
\begin{align*}
    H:=\bigg[\prod_{\alpha \in e(\tri)} X_\alpha^\tri \bigg],
\end{align*}
which satisfies $\iota_{\bD}(H)=\prod_{\alpha \in e(\tri)} y_\alpha^\tri \cdot \mathsf{1}$. The scalar $h=\prod_{\alpha \in e(\tri)} y_\alpha^\tri$ is called the \emph{central load}. It is a part of the classifying data of the local representation of $\X^\tri_{\by^\tri}$(\cref{subsec:BBL}).
\end{rem}


\subsubsection*{Choice of Lagrangian}
Now we choose a `polarization' as follows. 
Let $\Sigma^{\mathrm{cl}}$ denote the closed surface of genus $g \geq 0$ obtained from $\Sigma$ by capping a disk to each boundary component, and closing interior punctures. Choose a Lagrangian sub-lattice $L=\langle a_1,\dots,a_g\rangle \subset H_1(\Sigma^{\mathrm{cl}};\bZ)$ with respect to the intersection pairing. Then the elements of $L$, the classes $[c_p]$ for the punctures $p$ and the classes $[c_b]$ for the boundary components $b$ generate a maximal isotropic sub-lattice $\widehat{L} \subset H_\Sigma$, where we have a relation $\sum_p [c_p] + \sum_b [c_b]=0$. Here we naturally regard $H_1(\Sigma^{\mathrm{cl}};\bZ) \subset H_\Sigma$. 
Consider the simultaneous eigenspace decomposition for $\oph^{\bD}(\widehat{L})$:
\begin{align*}
    V_{\bD} = \bigoplus_{\boldsymbol{\lambda}} V_{\bD}(L,\boldsymbol{\lambda}),
\end{align*}
where $\boldsymbol{\lambda}: \widehat{L} \to \bZ_N$ is an additive character. For any $[c] \in \widehat{L}$, the operator $\oph^{\bD}_{[c]}$ acts on $V_{\bD}(L,\boldsymbol{\lambda})$ by $\oph^{\bD}_{[c]} = q^{2\boldsymbol{\lambda}([c])} \mathsf{1}$. 
The dimension is
\begin{align}\label{eq:dim_irrep}
    \dim V_{\bD}(L,\boldsymbol{\lambda}) = N^{|t(\tri)|-g-(n_p+n_b-1)} = N^{3g-3+n_b+|\bM|},
\end{align}
where $n_p:=|\bM \cap \interior \Sigma|$ is the number of punctures, and $n_b:=|\pi_0(\partial\Sigma)|$ is the number of boundary components. Indeed, if $k$ mutually-commuting cyclic operators of order $N$ are acting on an $N^m$ dimensional complex vector space, then each simultaneous eigenspace has dimension $N^{m-k}$.

\begin{thm}\label{thm:irrep}
Each $V_{\bD}(L,\boldsymbol{\lambda})$ is an irreducible representation of $\X_{\by^\tri}^{\tri}$ acted through $\iota_{\bD}$.
\end{thm}

\begin{proof}
The action of each cluster variable $X_\alpha^\tri$ preserves the subspace $V_{\bD}(L,\boldsymbol{\lambda})$, since it commutes with the operators $\oph_{[c]}^{\bD}$ for $[c] \in \widehat{L}$ by \cref{lem:Casimir}. Thus the quantum torus $\X_{\by^\tri}^{\tri}$ acts preserving $V_{\bD}(L,\boldsymbol{\lambda})$.

By \cite[Theorem 1.2]{KQ22}, any irreducible representation of $\X_{\by^\tri}^{\tri}$ must have the same dimension as \eqref{eq:dim_irrep}. Therefore $V_{\bD}(L,\boldsymbol{\lambda})$ is an irreducible representation.
\end{proof}


\begin{rem}
Toulisse \cite{Toulisse} has obtained an irreducible decomposition of a local representation $V_{\bD}$, concluding that each irreducible component has multiplicity $N^g$. While his description of the multiplicity space was implicit, our \cref{thm:irrep} gives a refined geometric description by identifying the multiplicity space with $\Hom(L,\bZ_N)$, where each irreducible component is identified with the eigenspace $V_{\bD}(L,\boldsymbol{\lambda})$ for $\boldsymbol{\lambda} \in \Hom(\widehat{L},\bZ_N)$. 
\end{rem}

\begin{thm}\label{thm:functor_irrep}
For any global coefficient $\bbp_\Sigma \in \cF_\Sigma$, Lagrangian sub-lattice $L \subset H_1(\Sigma^{\mathrm{cl}};\bZ)$ and additive character $\boldsymbol{\lambda}:\widehat{L} \to \bZ_N$, the projective representation $V_\ast(\bbp_\Sigma)$ restricts to the projective representation
\begin{align*}
    V_\ast(\bbp_\Sigma;L,\boldsymbol{\lambda}): \mathrm{Pt}^{\tdot}_\Sigma \to \mathrm{Vect}_\bC^\op, \quad \bD \mapsto V_\bD(L,\boldsymbol{\lambda}).
\end{align*}
\end{thm}

\begin{proof}
The elementary operators $\opV_\omega$ respects the decomposition by \cref{lem:homology_naturality}.
\end{proof}
We call $ V_\ast(\bbp_\Sigma;L,\boldsymbol{\lambda})$ the \emph{irreducible cyclic quantum \Teich\ theory}  associated with a global coefficient $\bbp_\Sigma \in \cF_\Sigma$, a Lagrangian $L \subset H_1(\Sigma^\mathrm{cl};\bZ)$ and a character $\boldsymbol{\lambda}:\widehat{L} \to \bZ_N$. 

\subsection{Examples}\label{subsec:example}
Here we give a few quantum geometric structures that the irreducible cyclic quantum \Teich\ theory carries. 

\subsubsection{Thrice-punctured sphere and the `conformal block'}\label{subsub:example_sphere}
Let $\Sigma=\Sigma_0^3$ be a thrice-punctured sphere, and take the dotted triangulation $\bD$ as shown in \cref{fig:sphere}. 
In this case, the unique choice is $L=0$, and the relative homology $\widehat{L}:=H_1(\Sigma^\ast;\bZ)$ itself is maximally isotropic. 
Let us assign the weights $\lambda,\mu,\nu \in \bZ_N$ such that $\lambda+\mu+\nu=0$ to the three punctures. 

\begin{figure}[ht]
    \centering
\begin{tikzpicture}[scale=1.3]
\draw (0,0) -- (2,0) -- (2,2) -- (0,2) --cycle; 
\draw (0,2) -- (2,0); 
\draw[red](0.5,0.5) --++(-90:0.8) node[below,scale=0.7]{$\dbra{\opU_v\opP_v^{-1}}$};
\draw[red](0.5,0.5) node[right,scale=0.7]{$\opP_v$} --++(180:0.8) node[left,scale=0.7]{$\opU_v^{-1}$};
\draw[red](1.5,1.5) --++(0:0.8) node[right,scale=0.7]{$\opU_w^{-1}$};
\draw[red](1.5,1.5)++(-0.2,-0.2) node[left,scale=0.7]{$\opP_w$};
\draw[red](1.5,1.5) --++(90:0.8) node[above,scale=0.7]{$\dbra{\opU_w\opP_w^{-1}}$};
\draw[red](0.5,0.5) -- (1.5,1.5);
\node[above left,scale=0.8] at (0.5,0.5) {$v$};
\node[below right,scale=0.8] at (1.5,1.5) {$w$};
\dast{(0.1,1.75)};
\dast{(1.9,0.25)};
\node[scale=0.8] at (0.7,1.5){$2$};
\node[left,scale=0.8] at (0,1){$3$};
\node[below,scale=0.8] at (1,0){$1$};
\node[right,scale=0.8] at (2,1){$1$};
\node[above,scale=0.8] at (1,2){$3$};
\node[above left,scale=0.8] at (0,2) {$\lambda$};
\node[below left,scale=0.8] at (0,0) {$\mu$};
\node[below right,scale=0.8] at (2,0) {$\nu$};
\node[above right,scale=0.8] at (2,2) {$\mu$};
\end{tikzpicture}
    \caption{A dotted triangulation of $\Sigma_0^3$. Here the edges of the same label are identified.}
    \label{fig:sphere}
\end{figure}
The homology operators $\oph_\lambda,\oph_\mu,\oph_\nu$ associated with these punctures are computed as
\begin{align*}
    \oph_\lambda &= [\opU_v^{-1}\opP_v \opU_w], \\
    \oph_\mu &= \opP_v^{-1}\opP_w^{-1}, \\
    \oph_\nu &= [\opU_v\opU_w^{-1}\opP_w].
\end{align*}

\begin{prop}[`Conformal block']\label{prop:conformal_block}
The space $V_{\bD}(\lambda,\mu,\nu)$ is $1$-dimensional, and spanned by the vector
\begin{align*}
    \ket{\Psi_{\lambda,\mu,\nu}} := \sum_{m,n \in \bZ_N} \gamma(m-n)^{-1} q^{-2m\lambda -2n\nu} \ket{m,n},
\end{align*}
where the first (resp. second) component corresponds to $v$ (resp. $w$).
\end{prop}

\begin{proof}
It is straightforward to verify that
$\oph_\lambda \ket{\Psi_{\lambda,\mu,\nu}} =q^{2\lambda} \ket{\Psi_{\lambda,\mu,\nu}}$, $\oph_\mu \ket{\Psi_{\lambda,\mu,\nu}} =q^{2\mu} \ket{\Psi_{\lambda,\mu,\nu}}$, $\oph_\nu \ket{\Psi_{\lambda,\mu,\nu}} =q^{2\nu} \ket{\Psi_{\lambda,\mu,\nu}}$ holds. Since $\oph_\lambda \oph_\mu \oph_\nu=\mathsf{1}$, one of them follows from the other two. Observe that these relations impose two independent recurrence relations for the coefficients $\psi_{mn}$ of $\ket{\Psi_{\lambda,\mu,\nu}}=\sum_{m,n}\psi_{mn}\ket{m,n}$, whose solution space is $1$-dimensional, as it agrees with \eqref{eq:dim_irrep}.
\end{proof}
The cluster Poisson variables acts on $\ket{\Psi_{\lambda,\mu,\nu}}$ by 
\begin{align*}
    \opX_1^\tri \ket{\Psi_{\lambda,\mu,\nu}} &= y_1 q^{-2\lambda}\ket{\Psi_{\lambda,\mu,\nu}}, \\
    \opX_2^\tri \ket{\Psi_{\lambda,\mu,\nu}} &= y_2 q^{-2\mu}\ket{\Psi_{\lambda,\mu,\nu}}, \\
    \opX_3^\tri \ket{\Psi_{\lambda,\mu,\nu}} &= y_3 q^{-2\nu}\ket{\Psi_{\lambda,\mu,\nu}}.
\end{align*}

\begin{rem}
The level $k=N+2$ Clebsch--Gordan rule is missing in this setting. 
\end{rem}

\subsubsection{Once-punctured disk and the quantum group}\label{subsub:example_disk}
Let $\Sigma$ be a once-punctured disk with two marked points on the boundary, and take the dotted triangulation $\bD$ as shown in \cref{fig:disk}. In this case, the unique choice is $L=0$, and $\widehat{L}=H_\Sigma$ is generated by $[c_p]$. 
Assign the weight $p \in \bZ_N$ to the unique puncture. The corresponding homology operators are
\begin{align*}
    \oph_{[c_p]} = \opU_v \opP_w^{-1}, \quad
    \oph_{[c_b]} = \opU_v^{-1}\opP_w.
\end{align*}

\begin{figure}[ht]
    \centering
\begin{tikzpicture}
\draw(0,0) circle(2cm);
\draw(0,-2) -- (0,2);
\draw[red] (0,0) circle(1cm);
\draw[red] (1,0) -- (2,0);
\draw[red] (-1,0) -- (-2,0);
\node[red,scale=0.7] at (-1.5,0.3){$\opU_v^{-1}$};
\node[red,scale=0.7] at (135:1.3){$\opP_v$};
\node[red,scale=0.7] at (225:1.45){$[\opU_v\opP_v^{-1}]$};
\node[red,scale=0.7] at (1.5,0.3){$\opP_w$};
\node[red,scale=0.7] at (45:1.3){$\opU_w^{-1}$};
\node[red,scale=0.7] at (-45:1.45){$[\opU_w\opP_w^{-1}]$};
\node[scale=0.8] at (150:2.2) {$1$};
\node[left,scale=0.8] at (0,1.3) {$2$};
\node[scale=0.8] at (30:2.2) {$3$};
\node[left,scale=0.8] at (0,-1.3) {$4$};
\dast{(-0.2,1.85)};
\dast{(0.2,1.85)};
\filldraw(0,-2) circle(1.5pt);
\filldraw(0,2) circle(1.5pt);
\filldraw[fill=white](0,0) circle(2pt) node[right,scale=0.8]{$p$};
\node[right,scale=0.8] at (-1,0){$v$};
\node[left,scale=0.8] at (1,0){$w$};
\end{tikzpicture}
    \caption{A dotted triangulation of a once-punctured disk with two marked points on the boundary.}
    \label{fig:disk}
\end{figure}
Following \cite{Ip,SS19,GS19}, we introduce the elements
\begin{align*}
    E &:= X_1 + [X_1X_2], \\
    F &:= X_3 + [X_3X_4], \\
    K &:= [X_1X_2X_3], \quad K' := [X_1X_4X_3],
\end{align*}
which generates the Drinfeld double of the quantum Borel subalgebra of $U_q(\mathfrak{sl}_2)$ inside $\X^\tri$ \cite[Theorem 4.14]{Ip}, \cite[Theorem 1]{SS19}. In particular, they satisfy the relations
\begin{align*}
    KE= q^2 EK, \quad KF= q^{-2}FK, \quad [E,F] = (q-q^{-1})(K' -K)
\end{align*}
and the quantum Serre relations. 
If we rescale the generators by $E= q^{-1/2}(q-q^{-1})\widetilde{E}$, $F= -q^{1/2}(q-q^{-1})\widetilde{F}$, then we get (cf. \cite[(505)]{GS19})
\begin{align*}
    [\widetilde{E},\widetilde{F}] = \frac{K-K'}{q-q^{-1}}.
\end{align*}
The quantum group $U_q(\mathfrak{sl}_2)$ is obtained by imposing the condition $KK'=1$.  
Via the embedding $\iota_{\bD}$, these elements give rise to the operators 
\begin{align*}
    \opE &= y_1 \opU_v^{-1} + y_1y_2 [\opU_v^{-1}\opP_v\opU_w^{-1}], \\
    \opF &= y_3 \opP_w + y_3y_4 [\opU_v\opP_v^{-1}\opU_w], \\ 
    \opK &= y_1y_2y_3 [\opU_v^{-1}\opP_w \opP_v \opU_w^{-1}], \quad
    \opK' = y_1y_4y_3 [\opP_v^{-1}\opU_w].
\end{align*}
We have
\begin{align*}
    \opK \opK' = (y_1^2 y_2y_4 y_3^2) \opU_v^{-1}\opP_w = (y_1^2 y_2y_4 y_3^2) \oph_{[c_p]}^{-1}.
\end{align*}
Therefore by imposing the condition $y_1^2 y_2y_4 y_3^2 q^{-2p}=1$, we obtain the relation $\opK\opK'=1$ on $V_{\bD}(L;p)$. 

\begin{lem}
The subspace $V_{\bD}(L;p) \subset V_{\bD}$ is spanned by the vectors
\begin{align*}
    \ket{\psi_m} := \sum_{n \in \bZ_N} q^{2(p-m)n} \ket{m,n}, 
\end{align*}
where the first (resp. second) component corresponds to $v$ (resp. $w$).
\end{lem}

\begin{proof}
A direct calculation shows $\oph_{[c_p]}\ket{\psi_m} = q^{2p}\ket{\psi_m}$.
\end{proof}
It is known that the (unrestricted) quantum group $U_q(\mathfrak{sl}_2)$ has irreducible cyclic representations with three complex parameters \cite{DCK90,DJMM91}. See also \cite[2.13, Case 3]{Jan}.
Our representation $V_{\bD}(L;p)$ is identified with one of them by passing to another basis $\ket{\phi_\ell}:=\gamma(\ell)^M\sum_{m \in \bZ_N}q^{-2\ell m} \ket{\psi_m}$:

\begin{thm}\label{thm:cyclic_quantum_group}
Given complex parameters $r,s,\lambda \in \bC^\ast$, we set
\begin{align*}
    y_1:= -q^{-1}r^{-1}\lambda, \quad y_2:= -q r^2, \quad
    y_3:= s^{-1}\lambda^{-1}q^{2p}, \quad y_4:= -s^2 q^{-2p+1}.
\end{align*}
Then the actions of $\widetilde{\opE},\widetilde{\opF},\opK$ are given by
\begin{align*}
    &\widetilde{\opE} \ket{\phi_\ell} = \frac{rq^{\ell+1}-r^{-1}q^{-\ell-1}}{q-q^{-1}}\lambda\ket{\phi_{\ell+1}}, \\
    &\widetilde{\opF} \ket{\phi_\ell} = \frac{sq^{-\ell+1}-s^{-1}q^{\ell-1}}{q-q^{-1}}\lambda^{-1}\ket{\phi_{\ell-1}}, \\
    &\opK \ket{\phi_\ell} = \frac r s q^{2\ell}\ket{\phi_{\ell}}.
\end{align*}
\end{thm}
This representation is exactly the one constructed in \cite{DJMM91}, where $\lambda=g_{11}$ in their notation.

\begin{proof}
The actions on the basis $\ket{\psi_m}$ are computed as
\begin{align*}
    &\opE \ket{\psi_m} = y_1 q^{-2m} \ket{\psi_m} + y_1y_2 q^{-2m-1}\ket{\psi_{m+1}}, \\
    &\opF \ket{\psi_m} = y_3 q^{2(m-p)} \ket{\psi_m} + y_3y_4 q^{2m-1}\ket{\psi_{m-1}}, \\
    &\opK \ket{\psi_m} = y_1y_2y_3 q^{-2p} \ket{\psi_{m+1}}.
\end{align*}
Then by the base-change, we obtain
\begin{align*}
    &\opE \ket{\phi_\ell} = y_1 (q^{-\ell-1/2} + q^{\ell+1/2}y_2) \ket{\phi_{\ell+1}} = q^{-1/2}(rq^{\ell+1}-r^{-1}q^{-\ell-1})\lambda \ket{\phi_{\ell+1}}, \\
    &\opF \ket{\phi_\ell} = y_3(q^{\ell-2p-1/2} + q^{-\ell+1/2}y_4) \ket{\phi_{\ell-1}} = -q^{1/2}(sq^{-\ell+1}-s^{-1}q^{\ell-1})\lambda^{-1}\ket{\phi_{\ell-1}}, \\
    &\opK \ket{\phi_\ell} = y_1y_2y_3 q^{-2p}q^{2\ell} \ket{\phi_\ell} = \frac r s q^{2\ell} \ket{\phi_\ell}.
\end{align*}
Observe that the condition $y_1^2 y_2y_4 y_3^2 q^{-2p}=1$ is satisfied, so that $\opK'=\opK^{-1}$. 
\end{proof}

%% file: 5_MCG_local.tex
\section{Invariants of mapping classes I: Local quantum trace}\label{sec:trace_local}
Here we discuss the invariant of a mapping class arising from the cyclic quantum \Teich\ theory constructed in \cref{sec:Kashaev}. We call it the \emph{local quantum trace}, since it corresponds to the local representations as explained in \cref{subsec:BBL}. We provide Dehn twists and a pseudo-Anosov examples on a once-punctured torus in \cref{subsec:MCG}.

\subsection{Local quantum trace of mapping classes}\label{subsec:trace}
Consider the mapping class group $MC(\Sigma):=\pi_0(\mathrm{Homeo}^+(\Sigma,\bM))$, where $\mathrm{Homeo}^+(\Sigma,\bM)$ denotes the group of orientation-preserving homeomorphisms of $\Sigma$ that preserves the subsets $\bM$ and $\partial\Sigma$. Typical elements are Dehn twists along simple closed curves, braidings of punctures, and rotations of boundary marked points. It naturally acts on the dotted Ptolemy groupoid $\mathrm{Pt}^{\tdot}_\Sigma$. 

Let $V_\ast(\bbp_\Sigma): \mathrm{Pt}^{\tdot}_\Sigma \to \mathrm{Vect}_\bC^\op$ be the projective representation associated with a global coefficient $\bbp_\Sigma \in \cF_\Sigma$. Fix an initial dotted triangulation $\bD \in \mathrm{Pt}^{\tdot}_\Sigma$.
Given a mapping class $\phi \in MC(\Sigma)$, there exists a finite sequence 
\begin{align}\label{eq:rep_seq}
    \omega: \bD=\bD_0 \xrightarrow{E_1} \bD_1 \xrightarrow{E_2} \dots \xrightarrow{E_m} \bD_m = \phi^{-1}(\bD)
\end{align}
of elementary morphisms and permutations of labels. We call $\omega$ a \emph{representation path} of $\phi$, and its image $\underline{\omega}$ in $\mathrm{Pt}_\Sigma$ the \emph{underlying flip sequence}. 
Let 
\begin{align}\label{eq:rep_seq_op}
    V_\omega: V_{\phi^{-1}(\bD)} = V_{\bD_m} \xrightarrow{V_{E_m}} \dots \xrightarrow{V_{E_2}} V_{\bD_1} \xrightarrow{V_{E_1}} V_{\bD_1} = V_{\bD}
\end{align}
be the associated sequence of linear maps. Then we consider the linear isomorphism $\iota_\phi: V_{\bD}\xrightarrow{\sim} V_{\phi^{-1}(\bD)} $ identifying the basis vectors, according to the correspondence between the triangles under $\phi$. 




\begin{dfn}[$\phi$-invariance]
We say that $\bbp_\Sigma \in \cF_\Sigma$ is \emph{$\phi$-invariant} if it satisfies
\begin{align}\label{eq:fixed}
    \bp_{\phi^{-1}(\alpha)}^{\phi^{-1}(\tri)}=\bp_\alpha^{\tri}
\end{align}
for all $\tri \in \mathrm{Pt}_\Sigma$ and $\alpha \in e(\tri)$. 
\end{dfn}
Note that the $\phi$-invariance implies the fixed point condition $\phi(\pi_\Sigma^N(\bbp_\Sigma)) = \pi_\Sigma^N(\bbp_\Sigma)$
in $\overline{\X}_\Sigma$, where we recall \cref{cor:proj_cluster}.


\begin{dfn}[Local quantum trace]\label{def:intertwiner}
Under the condition \eqref{eq:fixed}, the composite
\begin{align*}
    \opV_{\phi}^{\bD}(\bbp_\Sigma):= V_\omega \circ \iota_\phi: V_\bD \xrightarrow{\sim} V_\bD
\end{align*}
is called the \emph{local intertwiner} of the mapping class $\phi \in MC(\Sigma)$. We call
\begin{align*}
    Z_N(\phi;\bbp_\Sigma):= |\mathrm{Tr}(\opV_{\phi}^{\bD}(\bbp_\Sigma))|
\end{align*}
the \emph{local quantum trace} of $\phi$. 
\end{dfn}


\begin{prop}[Naturality]
If $\bbp_\Sigma$ is $\phi$-invariant, then the local intertwiner $\opV_{\phi}^{\bD}(\bbp_\Sigma)$ is independent of the choice of $\omega$ up to a multiplication by $\zeta$ in \cref{thm:AT_relations}. Moreover, the local quantum trace $Z_N(\phi;\bbp_\Sigma)$ is independent of the choice of  $\bD$.
\end{prop}

\begin{proof}
If we choose another path $\omega'$ from $\bD$ to $\phi^{-1}(\bD)$, then $\omega \circ (\omega')^{-1}$ is a closed path in $\mathrm{Pt}^{\tdot}_\Sigma$. By the presentation in \cref{thm:Ptolemy_presentation} and by \cref{thm:AT_relations}, such a closed path gives rise to an identity operator up to multiplication by $\zeta$. 

For two labeled triangulation $\bD_1,\bD_2$, take an arbitrary paths $\omega_j: \bD_j \to \phi^{-1}(\bD_j)$ for $j=1,2$, and
$\rho: \bD_1 \to \bD_2$ in $\mathrm{Pt}^{\tdot}_\Sigma$. By translating by $\phi^{-1}$, we get a path $\phi^{-1}(\rho): \phi^{-1}(\bD_1) \to \phi^{-1}(\bD_2)$. We have the following diagram:
\begin{equation*}
\begin{tikzcd}
    \bD_1 \ar[r,"\omega_1"] \ar[d,"\rho"'] & \phi^{-1}(\bD_1) \ar[d,"\phi^{-1}(\rho)"] \\
    \bD_2 \ar[r,"\omega_2"'] & \phi^{-1}(\bD_2).
\end{tikzcd}
\end{equation*}
Via the identifications of labels between $\bD_j$ and $\phi^{-1}(\bD_j)$ for $j=1,2$, we see that $\rho$ and $\phi^{-1}(\rho)$ are the same composite $V_\rho$ of elementary morphisms and permutations. Here, the $\phi$-invariance condition of $\bbp_\Sigma$ ensures that the coefficients appearing in these sequences are the same. 
By the first statement, we have
\begin{align*}
    \opV_{\phi}^{\bD_1}(\bbp_\Sigma) = \zeta^\bullet\cdot V_\rho \circ \opV_{\phi}^{\bD_2}(\bbp_\Sigma) \circ V_\rho^{-1}
\end{align*}
up to a power of $\zeta$. Taking the absolute value of the trace of both sides, we get the desired statement.
\end{proof}

\begin{rem}\label{rem:ambiguity}
$\mathrm{Tr}(\opV_{\phi}^{\bD}(\bbp_\Sigma))$ is well-defined up to powers of $\zeta$.
\end{rem}

\begin{rem}\label{rem:concatenation}
For two mapping classes $\phi_1,\phi_2 \in MC(\Sigma)$, take a sequence $\omega_i:\bD \to \phi_i^{-1}(\bD)$ of elementary moves for $i=1,2$. Let $\phi_1^{-1}(\omega_2):\phi_1^{-1}(\bD) \to \phi_1^{-1}(\phi_2^{-1}(\bD))$ be the translate of $\omega_2$ by the action of $\phi_1^{-1}$ on $\mathrm{Pt}^{\tdot}_\Sigma$. Then the concatenation
\begin{align*}
    \omega_1 \ast \phi_1^{-1}(\omega_2): \bD \to \phi_1^{-1}(\bD) \to \phi_1^{-1}(\phi_2^{-1}(\bD))=(\phi_2\phi_1)^{-1}(\bD)
\end{align*}
gives a sequence for the product $\phi_2\phi_1$. 
\end{rem}

\begin{prop}[Normalization]\label{prop:normalization_total}
The local intertwiner satisfies $|\det \opV_{\phi}^{\bD}(\bbp_\Sigma)|=1$.
\end{prop}

\begin{proof}
The local intertwiner $\opV_{\phi}^{\bD}(\bbp_\Sigma)$ is a composite of the operators $\opA$, $\opT$ and permutations. Each of these operators has unital determinant by the lemma below. 
\end{proof}

\begin{lem}
We have
\begin{align*}
    \det \opA_v = \zeta_A^N \zeta_\inv^N \det V(q^{-1}), \quad
    \det \opT_{vw} = 1.
\end{align*}
In particular, $|\det \opA_v| = |\det \opT_{vw}|=1$. 
\end{lem}

\begin{proof}
Observe that $\bra{n}\opA_v\ket{m}=\zeta_A \gamma(n)^{-1}q^{-2mn}$, and its $n$-th row has the common factor $\gamma(n)^{-1}$ for each $n \in \bZ_N$. Hence
\begin{align*}
    \det \opA_v = \zeta_A^N \prod_{n \in \bZ_N}\gamma(n)^{-1} \cdot \det V(q^{-1}) = \zeta_A^N \zeta_\inv^N \det V(q^{-1}).
\end{align*}
Here we have used \cref{lem:gamma_prod}. In particular, $|\det \opA_v| = N^{-N/2} \cdot 1 \cdot N^{N/2} =1$.

By \cref{lem:psi_product}, we get
\begin{align*}
    \det \opT_{vw} = \det \Psi_{\bp_{vw}}(\dbra{\opP_v^{-1}\opU_v\opP_w})\cdot \det\opS_{vw} = \det \opS_{vw}.
\end{align*}
By taking the basis $\ket{\mb{k},\ell} \in V_N \otimes V_N$, we compute
\begin{align*}
    \opS_{vw}\ket{\mb{k},\ell} &= \sum_{i,j \in \bZ_N} q^{2ij}q^{2j\ell}q^{2ik}\ket{\mb{k},\ell} \\
    &=\sum_{j \in \bZ_N}q^{2j\ell}\bigg(\sum_{i \in \bZ_N} q^{2i(j+k)}\bigg)\ket{\mb{k},\ell} \\
    &= \sum_{j \in \bZ_N}q^{j\ell} N\delta_{j,-k}\ket{\mb{k},\ell}  = q^{-k\ell}\ket{\mb{k},\ell}. 
\end{align*}
Therefore
\begin{align*}
    \det\opS_{vw} = \prod_{k,\ell \in \bZ_N}q^{-k\ell} = q^{-2\sum_{k,\ell}k\ell} = q^{-(N-1)N\sum_k k} = (-1)^{\frac{(N-1)^2N}{2}} =1.
\end{align*}
Thus we get $\det\opT_{vw}=1$ as desired.
\end{proof}

\begin{rem}\label{rem:gluing} 
The projective functor $V_\ast(\bbp_\Sigma)$ only gives a projective representation of the stabilizer subgroup of $\bbp_\Sigma$ in $MC(\Sigma)$, since we need the $\phi$-invariance condition for the correct identification of vector spaces. 
\end{rem}



\subsubsection*{Action on the quantum cluster Poisson algebra}
The mapping class group $MC(\Sigma)$ acts on $\cO_q(\X_\Sigma)$ from the right as algebra automorphisms so that $\phi^\ast(X_\alpha^{\tri}):= X_{\phi^{-1}(\alpha)}^{\phi^{-1}(\tri)}$ for $\phi \in MC(\Sigma)$, $\tri \in \mathrm{Pt}_\Sigma$ and $\alpha \in e(\tri)$. More explicitly, the action on each cluster Poisson variable can be computed as follows. Fix $\tri \in \mathrm{Pt}_\Sigma$, and take a flip sequence $\underline{\omega}$ from $\tri$ to $\phi^{-1}(\tri)$ (compare with \eqref{eq:rep_seq}). Then the rational action $\phi^\ast: \Frac \X^\tri \to \Frac \X^\tri$ is the composite
\begin{align}\label{eq:MC_action_rational}
    \phi^\ast: \Frac \X^\tri \xrightarrow{\sim} \Frac \X^{\phi^{-1}(\tri)} \xrightarrow{\mu_{\underline{\omega}}^\ast} \Frac \X^\tri,
\end{align}
where the first map is induced from the identification $X_\alpha^\tri \mapsto X_{\phi^{-1}(\alpha)}^{\phi^{-1}(\tri)}$, and the latter is the reverse composite of quantum cluster Poisson transformations along $\underline{\omega}$.

\begin{prop}
Let $\phi \in MC(\Sigma)$ be a mapping class. The automorphism $\phi^\ast$ on $\cO_q(\X_\Sigma)$ descends to an automorphism on $\cO_q(\X_\Sigma;\bbp_\Sigma)$ if  $\bbp_\Sigma \in \cF_\Sigma$ is $\phi$-invariant. 
\end{prop}

\begin{proof}
In this case, the action $\phi^\ast(X_\alpha^{\tri})= X_{\phi^{-1}(\alpha)}^{\phi^{-1}(\tri)}$ is equivalent to $\phi^\ast(\overline{X}_\alpha^{\tri})= \overline{X}_{\phi^{-1}(\alpha)}^{\phi^{-1}(\tri)}$. Therefore it preserves the defining relations.
\end{proof}


\begin{prop}\label{prop:intertwiner}
Let $\bbp_\Sigma \in \cF_\Sigma$ be a global coefficient, which is invariant under a mapping class $\phi \in MC(\Sigma)$. Then for any $\bD \in \mathrm{Pt}^{\tdot}_\Sigma$, the local intertwiner $V_\phi^\bD(\bbp_\Sigma)$ is an $\X^\tri$-intertwiner:
\begin{align*}
    V_\phi^\bD(\bbp_\Sigma) \circ X_\alpha^{\tri} = \phi^\ast(X_\alpha^\tri) \circ V_\phi^\bD(\bbp_\Sigma)
\end{align*}
for $\alpha \in e(\tri)$.
\end{prop}

\begin{proof}
Recall that $V_\phi^\bD(\bbp_\Sigma)=V_\omega \circ \iota_\phi$, where $V_\omega$ is the composite \eqref{eq:rep_seq_op} of elementary operators along a representation path $\omega$ of $\phi$, and $\iota_\phi:V_\bD \xrightarrow{\sim} V_{\phi^{-1}(\bD)}$ is the identification of basis vectors associated with triangles. Compare it with the factorization \eqref{eq:MC_action_rational}. From \cref{thm:K_FG}, the adjoint action of $V_\omega$ on the Chekhov--Fock algebra coincides with $\mu_{\underline{\omega}}^\ast$. The adjoint action of $\iota_\phi$ also coincides with the identification $X_\alpha^\tri \mapsto X_{\phi^{-1}(\alpha)}^{\phi^{-1}(\tri)}$, since the $\phi$-action on the dotted triangulations preserves the incidence relation of triangles and edges. Thus the assertion is proved. 
\end{proof}

\subsection{Geometric content of a $\phi$-invariant global coefficient $\bbp_\Sigma$.}
We refer the reader to \cref{sec:mapping_torus} for the definition of the mapping torus $M_\phi$ and the Goncharov--Shen's moduli space $\P_{PGL_2,\Sigma}$. 

\begin{prop}
A generic $\phi$-invariant global coefficient $\bbp_\Sigma \in \cF_\Sigma$ determines a $PSL_2(\bC)$-character $\widehat{\rho}:\pi_1(M_\phi) \to PSL_2(\bC)$ on the mapping torus $M_\phi$.
\end{prop}

\begin{proof}
Recall \cref{cor:proj_cluster} and its proof. 
Assume that $\xi:=\pi_N(\bbp_\Sigma)$ belongs to the non-compactified cluster variety $\X_\Sigma$. Then $\xi$ determines a point of $\mathcal{P}_{PGL_2,\Sigma}$. If $\bbp_\Sigma$ is $\phi$-invariant, then $\xi \in \P_{PGL_2,\Sigma}^\phi$. By \cref{thm:mapping_torus}, it lifts to a data on $M_\phi$, which contains a character $\widehat{\rho}: \pi_1(M_\phi) \to PSL_2(\bC)$. Here the lift is unique if $\xi$ is sufficiently generic, in particular if $\bM_\partial \neq \emptyset$. 
\end{proof}

\begin{thm}\label{thm:hyp_str}
If the mapping torus $M_\phi$ has a complete hyperbolic structure (cf. \cref{rem:mapping_torus_hyperbolic}), 
then there exists a $\phi$-fixed point in $\X_\Sigma$ that satisfies $Y_\alpha^\tri \in \bC \setminus \{0,1\}$ for all $\tri \in \mathrm{Pt}_\Sigma$ and $\alpha \in e(\tri)$.
\end{thm}

\begin{proof}
Let $\widehat{\rho}_{\mathrm{hyp}}: \pi_1(M_\phi) \to PSL_2(\bC)$ denote the monodromy homomorphism of the complete hyperbolic structure. We refer the reader to \cite[Section 3.2]{BWY-1} for the properties of $\widehat{\rho}_{\mathrm{hyp}}$ that we need. Since $\widehat{\rho}_{\mathrm{hyp}}$ is boundary-unipotent, it admits a unique $\mathbb{P}^1$-framing $\widehat{\beta}$. If we restrict the data $(\widehat{\rho}_{\mathrm{hyp}},\widehat{\beta})$ to the surface $\Sigma^\ast=\{0\} \times \Sigma^\ast \subset M_\phi$, we get a $\phi$-fixed point $(\rho_{\mathrm{hyp}},\beta) \in \X_{PGL_2,\Sigma}$. Since the framing determines a generic configuration of points in $\mathbb{P}^1$, the cluster coordinates satisfy
$Y_\alpha^\tri(\rho_{\mathrm{hyp}},\beta) \in \bC \setminus \{0,-1\}$ for any ideal triangulation $\tri$ and $\alpha \in e_{\interior}(\tri)$\footnote{Here, our convention of cross ratio is $[\infty:-1:0:z]=z$.}. 

We can supply the $\widehat{\mathbb{P}}^1$-framing $\widehat{\alpha}$ (and hence the frozen coordinates), as follows. For simplicity, assume that $\phi$ rotates the marked points on each boundary component by one step. 
Note that the monodromy $\rho_{\mathrm{hyp}}(\gamma_b)$ around a boundary component of $\Sigma^\ast$ is loxodromic. By the relation \eqref{eq:rel_filling}, the monodromy $\widehat{\rho}_{\mathrm{hyp}}(t)$ along the fiber loop is also loxodromic. For each component $c$ of the boundary link $L_b$, choose any section $z_c$ of $\widehat{\mathcal{L}}\times_{PGL_2} \widehat{\mathbb{P}}^1$ at a point of $c$ (namely a point of $\widehat{\mathbb{P}}^1$), and then extend via parallel-transport along $L_b$ to obtain $\widehat{\alpha}$. It is possible since $L_b$ is a trivial loop in $M_\phi$. If we restrict $\widehat{\alpha}$ to the surface $\Sigma^\ast=\{0\} \times \Sigma^\ast \subset M_\phi$, we get a $\widehat{\mathbb{P}}^1$-framing $\alpha$ of $\mathcal{L}$. A consecutive pair is given by $(z_c,\widehat{\rho}_{\mathrm{hyp}}(t)(z_c))$. 
If we choose $z_c$ generically so that $\pi(\widehat{\rho}_{\mathrm{hyp}}(t)(z_c)) \neq \pi(z_c)$ in $\mathbb{P}^1$, we get a $\phi$-fixed point $\xi_{\mathrm{hyp}}:=(\rho_{\mathrm{hyp}},\beta,\alpha) \in \P_{PGL_2,\Sigma}$.
By further choosing $z_c$ generically, we may ensure $Y_\alpha^\tri(\xi_{\mathrm{hyp}}) \in \bC \setminus \{0,-1\}$ for any boundary edge $\alpha$.

\end{proof}
We do not know if there exists a $\phi$-invariant global coefficient $\bbp_\Sigma$ such that $\pi_N(\bbp_\Sigma)=\xi_\mathbf{hyp}$ in general. See \cref{prop:LR_local}. 

In practice, given a mapping class $\phi$, we only need the following data to compute the local quantum trace: 

\begin{dfn}\label{def:coeff_local}
Let $\underline{\omega}: \tri=\tri_0 \to \dots \to \tri_n=\phi^{-1}(\tri)$ be a flip sequence in $\mathrm{Pt}_\Sigma$ representing $\phi$. A \emph{$\phi$-invariant coefficient over $\underline{\omega}$} is the tuple $\bp_{\underline{\omega}}=\bp^{\tri_\nu} \in \cF_\Sigma$ of coefficients  for $\nu=0,\dots,n$ such that each consecutive pair is related by the mutation rule \eqref{eq:coeff_relation}, and the condition $\bp_{\phi^{-1}(\alpha)}^{\phi^{-1}(\tri)} = \bp_\alpha^\tri$ holds for all $\alpha \in e(\tri)$.  
\end{dfn}

Such $\bbp_{\underline{\omega}}$ can be found as follows. 
The mutation relation \eqref{eq:cluster_transf_Y} and the $\phi$-invariance condition for the variables $Y_\alpha^{\tri_\nu}:=\pi_N(\bbp_\alpha^{\tri_\nu})$
impose a set of algebraic equations, which can be easily solved. In the hyperbolic case (\cref{thm:hyp_str}), $Y_\alpha^\tri:=Y_\alpha^\tri(\xi_\mathrm{hyp}) \in \bC \setminus \{0,1\}$ is a preferred solution. Suppose that neither of these parameters are $0$ nor $-1$.

Choosing any lifts $\bbp_\alpha^{\tri_\nu}$, the condition $\bp_{\phi^{-1}(\alpha)}^{\phi^{-1}(\tri)} = \bp_\alpha^\tri$ holds up to powers of $q^2$. In order to match these powers, we may introduce phase shifts $(q^{2a_\alpha^\nu} p_\alpha^{\tri_\nu,+}, q^{2b_\alpha^\nu} p_\alpha^{\tri_\nu,-})$ with $(a_\alpha^\nu,b_\alpha^\nu) \in \bZ$. Then the mutation relation \eqref{eq:coeff_relation} and the condition $\bp_{\phi^{-1}(\alpha)}^{\phi^{-1}(\tri)} = \bp_\alpha^\tri$ impose a set of linear equations among $(a_\alpha^\nu,b_\alpha^\nu) \in \bZ$, whose number is less than or equal to the number of parameters, hence can be solved. 

The solution space is closely related to the \emph{QH gluing variety} of \cite{BB-AGT}.

\begin{rem}
Even in the non-hyperbolic case (e.g. Dehn twists), one may further compactify $\cF_\Sigma$ and $\overline{\X}_\Sigma$ in the directions corresponding to the edges not being flipped along $\omega$. We may find $\phi$-invariant coefficients over a flip sequence in this enlarged space. See \cref{subsec:MCG} for examples. 
\end{rem}

\subsection{Examples}\label{subsec:MCG}
Let $T_c \in MC(\Sigma)$ denote the left-hand Dehn twist along an essential simple closed curve $c \subset \Sigma$. 
Let us consider a once-punctured torus $\Sigma=\Sigma_1^1$, together with the dotted triangulation $\bD$ shown in the top-left of \cref{fig:Dehn_twist}. Consider the Dehn twists $T_a,T_b^{-1}$. 
We can take the sequences of elementary moves in $\mathrm{Pt}^{\tdot}_\Sigma$ as shown in \cref{fig:Dehn_twist} from $\bD$ to $T_a^{-1}(\bD)$ and $T_b(\bD)$, respectively. Therefore the local intertwiners of $T_a$ and $T_b^{-1}$ are given by
\begin{align}\label{eq:Dehn_operator}
    \opV_{T_a}^{\bD}(\bbp_a) = \opT_{vw} \opA_w^{-1} \opP_{(vw)}, \quad
    \opV_{T_b^{-1}}^{\bD}(\bbp_b) =\opT_{vw} \quad: V_{\bD} \to V_{\bD}
\end{align}
for some global coefficients $\bbp_a,\bbp_b \in \cF_\Sigma$, where the identifications $\iota_{T_a},\iota_{T_b^{-1}}$ are omitted. 

The underlying flip sequence $\underline{\omega}_a$ for $T_a$ is the flip $\mu_1$ followed by the transposition $P_{(12)}$ of edge labels, which imposes the fixed point condition
\begin{align*}
    Y_1 = Y_2(1+Y_1)^2, \quad Y_2 = Y_1^{-1}, \quad Y_3 = Y_3(1+Y_1^{-1})^{-2}
\end{align*}
for $Y$-variables $Y_i=((p^\tri_i)^+/(p^\tri_i)^-)^N$, $i=1,2,3$ in the initial triangulation $\tri$. 
A solution is $(Y_1,Y_2,Y_3) = (-1/2,-2,0)$.

The mutation-equivalence and invariance conditions on the coefficients are:
\begin{align*}
     &{p'}^\pm_1 = p^\mp_1, \quad \frac{{p'}^+_2}{{p'}^-_2} = \frac{p^+_2}{p^-_2}(p_1^-)^{-2}, \quad \frac{{p'}^+_3}{{p'}^-_3} = \frac{p^+_3}{p^-_3}(p_1^+)^{2}\\
     &\bp'_1 = \bp_2, \quad \bp'_2 = \bp_1, \quad \bp'_3=\bp_3.
\end{align*}
Here $\bp^\tri=(\bp_1,\bp_2,\bp_3)$ and $\bp^{\mu_1(\tri)}=(\bp'_1,\bp'_2,\bp'_3)$. The unique solution over $(Y_1,Y_2,Y_3) = (-1/2,-2,0)$ is
\begin{align*}
    \bp_1=(e^{\pi i/N},\sqrt[N]{2})=\bp'_2, \quad \bp_2=(\sqrt[N]{2},e^{\pi i/N})=\bp'_1, \quad \bp_3=(0,1)=\bp'_3.
\end{align*}
These data define a $T_a$-invariant coefficient $\bbp_a$ over $\underline{\omega}_a$. A $T_b^{-1}$-invariant coefficient $\bbp_b$ for $T_b^{-1}$ can be similarly found. 

\begin{figure}[ht]
    \centering
\begin{tikzpicture}
\begin{scope}
\draw (0,0) -- (2,0) -- (2,2) -- (0,2) --cycle; 
\draw (0,2) -- (2,0); 
\node[scale=0.9] at (0.5,0.5) {$v$};
\node[scale=0.9] at (1.5,1.5) {$w$};
\dast{(0.15,0.15)};
\dast{(1.9,0.25)};
{\color{blue}
\node[scale=0.8] at (0.7,1.5){$1$};
\node[left,scale=0.8] at (0,1){$2$};
\node[right,scale=0.8] at (2,1){$2$};
\node[above,scale=0.8] at (1,2){$3$};
\node[below,scale=0.8] at (1,0){$3$};
}
\draw[red,thick] (0,1) -- (2,1) node[below left]{$a$};
\node at (2,-0.5) {$\bD$};
\draw[<-,thick] (2.5,1) --node[midway,above]{$T_a$} ++(1,0);
\draw[->] (0.75,-0.5) --node[midway,left]{$T_{vw}$} node[blue,midway,right]{$\mu_1$} ++(0,-1);
\end{scope}
\begin{scope}[xshift=4cm]
\draw (0,0) -- (2,0) -- (2,2) -- (0,2) --cycle; 
\draw (2,2) -- (0,0);
\node[scale=0.9] at (0.5,1.5) {$w$};
\node[scale=0.9] at (1.5,0.5) {$v$};
\dast{(0.1,0.25)};
\dast{(0.25,0.1)};
{\color{blue}
\node[scale=0.8] at (1.3,1.5){$2$};
\node[left,scale=0.8] at (0,1){$1$};
\node[right,scale=0.8] at (2,1){$1$};
\node[above,scale=0.8] at (1,2){$3$};
\node[below,scale=0.8] at (1,0){$3$};
}
\node at (2,-0.5) {$T_a^{-1}(\bD)$};
\draw[<-] (0.75,-0.5) --node[midway,right]{$P_{(vw)}$} node[blue,midway,left]{$P_{(12)}$} ++(0,-1);
\end{scope}

\begin{scope}[yshift=-4cm]
\draw (0,0) -- (2,0) -- (2,2) -- (0,2) --cycle; 
\draw (2,2) -- (0,0);
\node[scale=0.9] at (0.5,1.5) {$v$};
\node[scale=0.9] at (1.5,0.5) {$w$};
\dast{(0.1,0.25)};
\dast{(1.85,0.1)};
{\color{blue}
\node[scale=0.8] at (1.3,1.5){$1$};
\node[left,scale=0.8] at (0,1){$2$};
\node[right,scale=0.8] at (2,1){$2$};
\node[above,scale=0.8] at (1,2){$3$};
\node[below,scale=0.8] at (1,0){$3$};
}
\draw[->] (2.5,1) --node[midway,below]{$A_w^{-1}$} node[blue,midway,above]{$\mathrm{id}$} ++(1,0);
\end{scope}

\begin{scope}[xshift=4cm,yshift=-4cm]
\draw (0,0) -- (2,0) -- (2,2) -- (0,2) --cycle; 
\draw (2,2) -- (0,0);
\node[scale=0.9] at (0.5,1.5) {$v$};
\node[scale=0.9] at (1.5,0.5) {$w$};
\dast{(0.1,0.25)};
\dast{(0.25,0.1)};
{\color{blue}
\node[scale=0.8] at (1.3,1.5){$1$};
\node[left,scale=0.8] at (0,1){$2$};
\node[right,scale=0.8] at (2,1){$2$};
\node[above,scale=0.8] at (1,2){$3$};
\node[below,scale=0.8] at (1,0){$3$};
}
\end{scope}

\begin{scope}[xshift=8cm]
\draw (0,0) -- (2,0) -- (2,2) -- (0,2) --cycle; 
\draw (0,2) -- (2,0); 
\node[scale=0.9] at (0.5,0.5) {$v$};
\node[scale=0.9] at (1.5,1.5) {$w$};
\dast{(0.15,0.15)};
\dast{(1.9,0.25)};
{\color{blue}
\node[scale=0.8] at (0.7,1.5){$1$};
\node[left,scale=0.8] at (0,1){$2$};
\node[right,scale=0.8] at (2,1){$2$};
\node[above,scale=0.8] at (1,2){$3$};
\node[below,scale=0.8] at (1,0){$3$};
}
\draw[red,thick] (1,0) -- (1,2) node[below right]{$b$};
\node at (2,-0.5) {$\bD$};
\draw[<-,thick] (2.5,1) --node[midway,above]{$T_b^{-1}$} ++(1,0);
\draw[->] (1,-0.5) --node[midway,below]{$T_{vw}$} node[blue,midway,above]{$\mu_1$} ++(2.5,-2.5);
\end{scope}

\begin{scope}[xshift=12cm]
\draw (0,0) -- (2,0) -- (2,2) -- (0,2) --cycle; 
\draw (2,2) -- (0,0);
\node[scale=0.9] at (0.5,1.5) {$v$};
\node[scale=0.9] at (1.5,0.5) {$w$};
\dast{(0.1,0.25)};
\dast{(1.85,0.1)};
{\color{blue}
\node[scale=0.8] at (1.3,1.5){$3$};
\node[left,scale=0.8] at (0,1){$2$};
\node[right,scale=0.8] at (2,1){$2$};
\node[above,scale=0.8] at (1,2){$1$};
\node[below,scale=0.8] at (1,0){$1$};
}
\node at (2,-0.5) {$T_b(\bD)$};
\end{scope}

\begin{scope}[xshift=12cm,yshift=-4cm]
\draw (0,0) -- (2,0) -- (2,2) -- (0,2) --cycle; 
\draw (2,2) -- (0,0);
\node[scale=0.9] at (0.5,1.5) {$v$};
\node[scale=0.9] at (1.5,0.5) {$w$};
\dast{(0.1,0.25)};
\dast{(1.85,0.1)};
{\color{blue}
\node[scale=0.8] at (1.3,1.5){$1$};
\node[left,scale=0.8] at (0,1){$2$};
\node[right,scale=0.8] at (2,1){$2$};
\node[above,scale=0.8] at (1,2){$3$};
\node[below,scale=0.8] at (1,0){$3$};
}
\node at (2,-0.5) {$T_b(\bD)$};
\draw[->] (1,2.5) --node[midway,right]{$\mathrm{id}$} node[blue,midway,left]{$P_{(13)}$} ++(0,1);
\end{scope}

\end{tikzpicture}
\caption{The representation paths for the Dehn twists $T_a$ and $T_b^{-1}$. Here the pairs of opposite sides are identified. The underlying flip sequence and the identification of edge labels are shown in \textcolor{blue}{blue}.}
\label{fig:Dehn_twist}
\end{figure}

\begin{prop}[Proof in \cref{subsubsec:proof_Dehn}]\label{prop:Dehn_twist}
\begin{enumerate}
    \item For the Dehn twist $T_a$, we have
\begin{align*}
    \bra{\mbb{m},\mb{n}} \opV_{T_a}^{\bD}(\bbp_a) \ket{\mbb{k},\mb{\ell}} = N i^{(1-N)/3} \Psi_{\bp_1}(q^{2(n-m)}) \gamma(m-\ell)^{-1}q^{2(n-\ell)k} 
\end{align*}
and 
\begin{align*}
    Z_N(T_a;\bbp_a) = \left|\sum_{u \in \bZ_N} \Psi_{\bp_1}(q^{2u}) \gamma(u)^{-1}\right|.
\end{align*}
\item For the Dehn twist $T_b^{-1}$, we have
\begin{align*}
    \bra{\mbb{m},\mb{n}} \opV_{T_b^{-1}}^{\bD}(\bbp_b) \ket{\mbb{k},\mb{\ell}} = N \Psi_{\bp_1}(q^{2(n-m)})\gamma(k-m)q^{2(n-m-\ell)(m-k)}
\end{align*}
and
\begin{align*}
    Z_N(T_b^{-1};\bbp_b) = \left| \sum_{u \in \bZ_N}\Psi_{\bp_1}(q^{2u})\right|.
\end{align*}
\end{enumerate}

\end{prop}

\paragraph{\textbf{General examples}}
The modular group $MC(\Sigma_1^1)\cong SL_2(\bZ)$ is generated by the two elements
\begin{align*}
    L :=\begin{pmatrix}
        1 & 1 \\ 0 & 1 
    \end{pmatrix}, \quad 
    R:=\begin{pmatrix}
        1 & 0 \\ 1 & 1
    \end{pmatrix}.
\end{align*}
The actions of $T_a$ and $T_b$ on the homology group $H_1(\Sigma_1^1)=\langle a,b \rangle$ are presented by $L^{-1}$ and $R$, respectively. Any mapping class $\phi \in MC(\Sigma_1^1)$ can be written as a word in $\{L^{\pm 1},R^{\pm 1}\}$, and \cref{rem:concatenation} tells us a way to obtain a sequence of elementary moves for $\phi$. 

For example, consider the pseudo-Anosov mapping class $\phi:=T_b^{-1}T_a$, which corresponds to the matrix $R^{-1}L^{-1}$. It is conjugate to $LR=\begin{pmatrix}
    2 & 1 \\ 1 & 1
\end{pmatrix}$, whose mapping torus is the figure-eight knot complement.

\begin{figure}[ht]
    \centering
\begin{tikzpicture}
\draw[mygreen,thick] (0,2) -- (2,0); 
\draw[red,thick] (0,0) -- (2,0);
\draw[red,thick] (0,2) -- (2,2);
\draw[myblue,thick] (0,0) -- (0,2);
\draw[myblue,thick] (2,0) -- (2,2);
\node[scale=0.9] at (0.5,0.5) {$v$};
\node[scale=0.9] at (1.5,1.5) {$w$};
\dast{(0.15,0.15)};
\dast{(1.9,0.25)};
{\color{blue}
\node[scale=0.8] at (0.7,1.5){$Y_1$};
\node[left,scale=0.8] at (0,1){$Y_2$};
\node[right,scale=0.8] at (2,1){$Y_2$};
\node[above,scale=0.8] at (1,2){$Y_3$};
\node[below,scale=0.8] at (1,0){$Y_3$};
}
\node[scale=0.9] at (2,-0.5) {$\bD$};
\draw[->,thick] (2.6,1) --node[midway,above,scale=0.8]{$T_{vw}A_w^{-1}P_{(vw)}$} node[midway,below,blue,scale=0.9]{$\mu_1$} ++(1.8,0);

\begin{scope}[xshift=5cm]
\draw (0,0) -- (2,0) -- (2,2) -- (0,2) --cycle; 
\draw (2,2) -- (0,0);
\node[scale=0.9] at (0.5,1.5) {$w$};
\node[scale=0.9] at (1.5,0.5) {$v$};
\dast{(0.1,0.25)};
\dast{(0.25,0.1)};
{\color{blue}
\node[scale=0.8] at (1.1,1.4){$Y'_1$};
\node[left,scale=0.8] at (0,1){$Y'_2$};
\node[right,scale=0.8] at (2,1){$Y'_2$};
\node[above,scale=0.8] at (1,2){$Y'_3$};
\node[below,scale=0.8] at (1,0){$Y'_3$};
}
\node[scale=0.9] at (2,-0.5) {$T_a^{-1}(\bD)$};
\node[scale=1.2,rotate=90] at (1,-1) {$=$};
\end{scope}

\begin{scope}[xshift=5cm,yshift=-3.5cm]
\draw (0,0) -- (4,0) -- (4,2) -- (0,2) --cycle; 
\draw (2,0) -- (2,2);
\draw (2,2) -- (0,0);
\draw (4,2) -- (2,0);
\node[scale=0.9] at (0.5,1.5) {$w$};
\node[scale=0.9] at (1.5,0.5) {$v$};
\node[scale=0.9] at (2.5,1.5) {$w$};
\node[scale=0.9] at (3.5,0.5) {$v$};
\dast{(0.1,0.25)};
\dast{(0.25,0.1)};
\dast{(2.1,0.25)};
\dast{(2.25,0.1)};
{\color{blue}
\node[scale=0.8] at (1.1,1.4){$Y'_1$};
\node[scale=0.8] at (3.1,0.7){$Y'_1$};
\node[right,scale=0.8] at (2,1){$Y'_2$};
\node[above,scale=0.8] at (3,2){$Y'_3$};
\node[below,scale=0.8] at (1,0){$Y'_3$};
}
\draw[->,thick] (4.6,1) --node[midway,above,scale=0.8]{$T_{vw}$}node[midway,below,blue,scale=0.9]{$\mu_2$} ++(1.8,0);
\end{scope}

\begin{scope}[xshift=12cm,yshift=-3.5cm]
\draw[gray!40] (0,0) -- (4,0) -- (4,2) -- (0,2) --cycle; 
\draw[mygreen,thick] (0,0) -- (2,0);
\draw[mygreen,thick] (2,2) -- (4,2);
\draw[red,thick] (0,0) -- (4,2);
\draw[myblue,thick] (2,2) -- (0,0);
\draw[myblue,thick] (4,2) -- (2,0);
\node[scale=0.9] at (1.8,0.5) {$w$};
\node[scale=0.9] at (2.2,1.5) {$v$};
\dast{(0.5,0.35)};
\dast{(1.9,0.15)};
{\color{blue}
\node[scale=0.8] at (1.1,1.4){$Y''_1$};
\node[scale=0.8] at (3.1,0.7){$Y''_1$};
\node[right,scale=0.8] at (2,0.8){$Y''_2$};
\node[above,scale=0.8] at (3,2){$Y''_3$};
\node[below,scale=0.8] at (1,0){$Y''_3$};
}
\node[scale=0.9] at (2,3) {$T_a^{-1}T_b(\bD)=\phi^{-1}(\bD)$};
\end{scope}


\end{tikzpicture}
    \caption{The representation sequence for $\phi=T_b^{-1}T_a$. In the first and last triangulations, the edges with the same color and identified by $\phi$.}
    \label{fig:flip_LR}
\end{figure}

\begin{prop}\label{prop:LR_local}
We have
\begin{align}\label{eq:LR_operator}
    \opV_\phi^{\bD}(\bbp_\phi) = \opT_{vw} \opA_w^{-1} \opP_{(vw)} \opT_{vw}
\end{align}
for the same initial dotted triangulation $\bD$ as in \cref{fig:Dehn_twist}, 
where an example of $\phi$-invariant coefficient $\bbp_\phi$ is shown in \eqref{eq:LR_coeff}. The local quantum trace is given by
\begin{align*}
    Z_N(\phi;\bbp_\phi) = N\left|\bigg(\sum_{k \in \bZ_N} \Psi_{\bp_1}(q^{2k})q^{-Mk^2} \bigg)\bigg(\sum_{\ell \in \bZ_N} \Psi_{\bp'_2}(q^{2\ell})q^{-M\ell^2} \bigg)  \right|
\end{align*}
\end{prop}

\begin{proof}
We get the operator expression \eqref{eq:LR_operator} from \cref{rem:concatenation}. See also \cref{fig:flip_LR}. The variables $Y_\alpha^\tri=\pi_N(\bp_\alpha^\tri)$ are also shown. 

\paragraph{\textbf{Coefficient parameters.}}
By the cluster transformation \eqref{eq:cluster_transf_Y}, we obtain
\begin{align*}
    Y'_1 &= Y_1^{-1}, & Y''_1 &= Y_1^{-1}(1+Y_2(1+Y_1)^2)^2, \\
    Y'_2 &= Y_2(1+Y_1)^2, & Y''_2 &= Y_2^{-1}(1+Y_1)^{-2}, \\
    Y'_3 &= Y_3(1+Y_1^{-1})^{-2}, & Y''_3 &= Y_3(1+Y_1^{-1})^{-2}(1+Y_2^{-1}(1+Y_1)^{-2})^{-2}, 
\end{align*}
and the $\phi$-fixed point equation $Y_1=Y''_3$, $Y_2=Y''_1$, $Y_3=Y''_2$. A solution is given by
\begin{align*}
    Y_1 &= e^{4\pi i/3} , & Y'_1 &= e^{-4\pi i/3}, & Y''_1 &= e^{4\pi i/3}, \\
    Y_2 &= e^{4\pi i/3} , & Y'_2 &= e^{-4\pi i/3}, & Y''_2 &= e^{4\pi i/3}, \\
    Y_3 &= e^{4\pi i/3} , & Y'_3 &= e^{-4\pi i/3}, & Y''_3 &= e^{4\pi i/3}.
\end{align*}
All the possible $\phi$-invariant coefficients over this solution are:
\begin{align}
\begin{aligned}\label{eq:LR_coeff}
    \bp_1&= (\bbe^{4+3(a+f)}, \bbe^{2+3a}), & \bp'_1&=(\bbe^{2+3a}, \bbe^{4+3(a+f)}), & \bp''_1&=(\bbe^{1+3(4a+b+2c+f)}, \bbe^{-1+3b}), \\
    \bp_2&=(\bbe^{1+3(4a+b+2c+f)}, \bbe^{-1+3b}), & \bp'_2&=(\bbe^{-4+3c}, \bbe^{-2+3(-2a-c-f)}), & \bp''_2&= (\bbe^{-2+3(-2a-c-f)}, \bbe^{-4+3c}), \\
    \bp_3&=(\bbe^{-2+3(-2a-c-f)}, \bbe^{-4+3c}) & \bp'_3&=(\bbe^{5+3(-2c+d+f)},\bbe^{-5+3d}), & \bp''_3&=(\bbe^{4+3(a+f)}, \bbe^{2+3a}).
\end{aligned}
\end{align}
Here, $a,b,c,d,f \in \bZ$ and we write $\mathbbm{e}:=e^{2\pi i /3N}$. 
We use $\bp_1$ and $\bp'_2$ in \eqref{eq:LR_operator}.  

\paragraph{\textbf{Matrix coefficients.}}
By using \cref{prop:Dehn_twist},
\begin{align*}
    &\bra{\mbb{m},\mb{n}} \opV_{\phi}^{\bD}(\bbp_\phi) \ket{\mbb{k},\mb{\ell}} \\
    =& \frac{1}{N^2} \sum_{p,r\in \bZ_N}\bra{\mbb{m},\mb{n}} \opT_{vw} \opA_w^{-1} \opP_{(vw)} \ket{\mbb{p},\mb{r}}\bra{\mbb{p},\mb{r}} \opT_{vw} \ket{\mbb{k},\mb{\ell}} \\
    =& i^{(1-N)/3} \sum_{p,r\in \bZ_N}\Psi_{\bp_1}(q^{2(n-m)}) \gamma(m-r)^{-1}q^{2(n-r)p} \Psi_{\bp'_2}(q^{2(r-p)})\gamma(k-p)q^{2(r-p-\ell)(p-k)}.
\end{align*}
Taking the trace,
\begin{align*}
    &Z_N(\phi;\bbp_\phi) \\
    =& \left| \sum_{m,n,p,r\in \bZ_N}\Psi_{\bp_1}(q^{2(n-m)}) \gamma(m-r)^{-1}q^{2(n-r)p} \Psi_{\bp'_2}(q^{2(r-p)})\gamma(m-p)q^{2(r-p-n)(p-m)}\right| \\
    =& \left| \sum_{m,n,p,r\in \bZ_N}\Psi_{\bp_1}(q^{2(n-m)})\Psi_{\bp'_2}(q^{2(r-p)})\frac{\gamma(p)}{\gamma(r)}q^{2m(r-p)} q^{2(n-r)p} q^{2(r-p-n)(p-m)}\right| \\
    =& \left| \sum_{m,k,\textcolor{red}{p},\ell\in \bZ_N}\Psi_{\bp_1}(q^{2k})\Psi_{\bp'_2}(q^{2\ell})\gamma(\ell)^{-1} q^{-2p\ell} q^{2m\ell} q^{2(m+k-p+\ell)p} q^{2(\ell-m-k)(p-m)}\right| \\
    =& \left| \zeta_-^{(2)}\sum_{\textcolor{red}{m},k,\ell\in \bZ_N}\Psi_{\bp_1}(q^{2k})\Psi_{\bp'_2}(q^{2\ell})\gamma(\ell)^{-1}  q^{2m\ell}  q^{-2m(\ell-m-k)} \gamma(\ell)^M\right|  \\
    =& \left| \zeta_-^{(2)}\zeta_+^{(2)}\sum_{k,\ell\in \bZ_N}\Psi_{\bp_1}(q^{2k})\Psi_{\bp'_2}(q^{2\ell})q^{-M\ell^2}  q^{-Mk^2}\right|  \\
    =& N\left|\bigg(\sum_{k \in \bZ_N} \Psi_{\bp_1}(q^{2k})q^{-Mk^2} \bigg)\bigg(\sum_{\ell \in \bZ_N} \Psi_{\bp'_2}(q^{2\ell})q^{-M\ell^2} \bigg)  \right|.
\end{align*}
Here we used \cref{lem:Gauss_2} twice.
\end{proof}

\section{Relation to the quantum hyperbolic field theory}\label{sec:invariant}

In this section, we discuss the relation with the \emph{reduced quantum hyperbolic operator} (reduced QH operator for short) introduced and investigated in \cite{BB-QT,BB-GD}. 
The reduced QH operator is a ``reduced'' version of quantum hyperbolic operator \cite{BB-QT}, where the latter is a part of the quantum hyperbolic field theory (QHFT) \cite{BB-GT,BB-AGT}. 

In \cite{BB-GD}, the transpose of reduced QH operator associated with the mapping cylinder $C_\phi$ of a mapping class $\phi \in MC(\Sigma)$ gives an intertwiner of the local representation $\iota'_{\bD}$ \cite[Theorem 1.1]{BB-GD}. Hence it selects a canonical intertwiner among those arising from Bai--Bonahon--Liu \cite{BBL,Mazzoli}. 
We will show that this coincides with a conjugate of our quantum intertwiner associated with $\phi$. 

The author thanks St\'ephane Baseilhac for illuminating discussion on this topic.  

\subsection{Conjugation of the projective representation}

Fix $\bbp_\Sigma \in \cF_\Sigma$. We define a projective functor $V'_\ast(\bbp_\Sigma): \mathrm{Pt}^{\tdot}_\Sigma \to \mathrm{Vect}_\bC^\op$ by setting
\begin{align}
    &\opA'_v\ket{k_v}:=\zeta_A \gamma(k_v)^{-1}\ket{\mb{k_v}}, \nonumber\\ &\opT'_{vw}:=\Psi_{\bp_\alpha^{\tri}}([\opP_w\opP_v^{-1}\opU_w]) \circ \opS'_{vw} \label{eq:flip_op_BB}
\end{align}
with $\opS'_{vw}:=N^{-1} \sum_{i,j=0}^{N-1} q^{2ij} \opU_w^i \dbra{\opU_v\opP_v}^j$ and recall $\zeta_A=(N\zeta_-)^{-1/3}=N^{-1/2}i^{(1-N)/6}$. The rotation operator $\opA'_v$ is designed so that $(\opA'_v)^3=\mathsf{1}$ and its conjugation action satisfies $\opU_v \mapsto \opP_v \mapsto [\opU_v\opP_v]^{-1} \mapsto \opU_v$.

It can be verified that this assignment indeed yields a projective functor $V'_\ast(\bbp_\Sigma): \mathrm{Pt}^{\tdot}_\Sigma \to \mathrm{Vect}_\bC^\op$, with the same central charge $\zeta=e^{\frac{\pi i}{6}(1-\frac 1N)}$ as before.

We have an embedding $\iota'_{\bD}: \X_{\by^\tri}^\tri \to \cW_{\bD}$ given by a modified version of Weyl weight assignment\footnote{Here, the correspondence takes the inverse of that in \cite{Bai}, because of the difference in conventions for the quantum cluster transformation. Compare \eqref{eq:cluster_transf_X} with \cite[Proposition 5]{Liu}, which \cite{Bai} follows.}: 
\begin{align}\label{eq:variant_Weyl}
\begin{tikzpicture}
\foreach \i in {30,150,270} \draw(0,0) -- (\i:1);
\node[above] at (30:1) {$\dbra{U_vP_v}$};
\node[above] at (150:1) {$U_v^{-1}$};
\node[below] at (270:1) {$P_v^{-1}$};
\node[below right] at (0,0) {$v$};
\dast{(0,0.2)};
\end{tikzpicture}
\end{align}
and the flip operator $\opT'_{vw}$ restricts to the quantum cluster transformation, just similarly to \cref{thm:K_FG}. 
As we will discuss in detail in the proof of \cref{thm:QHFT}, the flip operator \eqref{eq:flip_op_BB} is the same as the tetrahedron operator (the \emph{basic matrix dilogarithm}) in the Baseilhac--Benedetti's QHFT (see \cite[(11)]{BB-GD}).
Let us call $\iota'_{\bD}$ the \emph{Baseilhac--Benedetti realization}. This realization is directly connected to the representation theory of the quantum Borel algebra of $\mathfrak{sl}_2$ and the Kashaev's $6j$-symbol as well, as explained by Baseilhac \cite{Baseilhac} and Bai \cite{Bai}. For an explanation in the present notation, see Section 6 of the first arXiv version of this paper. 
The two realizations are explicitly related as follows:

\begin{prop}\label{prop:two_realizations}
Let $\opG_{\bD}: V_{\bD} \to V_{\bD}$ be the operator defined by $\opG_{\bD}\ket{k_v}:=\gamma(k_v)\ket{k_v}$ for each $v \in t(\tri)$. Then it acts on $\cW_\bD$ by
\begin{align*}
    \Ad_{\opG_\bD}(\opU_v) = \opU_v, \quad \Ad_{\opG_\bD}(\opP_v) = [\opU_v\opP_v]
\end{align*}
for $v \in t(\tri)$. Moreover, it intertwines the two realizations as follows:
\begin{align*}
    \opA'_v = \opG_{\bD_1} \opA_v (\opG_{\bD_2})^{-1}, \quad \opT'_{vw} = \opG_{\bD_1} \opT_{vw} (\opG_{\bD_2})^{-1}
\end{align*}
for each elementary move $\bD_1 \to \bD_2$, and $\iota'_{\bD} = \Ad_{\opG_{\bD}} (\iota_{\bD})$ for each $\bD \in \mathrm{Pt}^{\tdot}_\Sigma$. 
\end{prop}

\begin{proof}
Straightforward verification. 
\end{proof}
Therefore, the two projective representations $V_\ast(\bbp_\Sigma)$ and $V'_\ast(\bbp_\Sigma)$ are equivalent.

\begin{rem}
Yet another finite-dimensional representation of $\mathrm{Pt}^{\tdot}_\Sigma$ has been constructed by Kim \cite{Kim19} based on the Kashaev's $6j$-symbol. Kim's representation has a trivial central charge $\zeta=1$ but with another phase factor in the $\opA\opT\opA=\opA\opT\opA$ relation \cite[Propositions 5.4 and 5.8]{Kim19}.
\end{rem}

\subsection{Comparison with the reduced QH operator of mapping cylinder}

Let us fix a mapping class $\phi \in MC(\Sigma)$ and its representation sequence $\omega$ as in \eqref{eq:rep_seq}. Here we assume that $\omega$ is \emph{full}, in the sense that every interior edge in the initial triangulation $\bD$ is flipped at least once during the sequence $\omega$. This condition can be always satisfied if we deform a given representation sequence by inserting a ``round-trip'' corresponding to the relation $A_v T_{vw} A_w = A_w T_{wv} A_v$ for every edge. 

The underlying sequence $\underline{\omega}$ of flips in $\mathrm{Pt}_\Sigma$ defines a 3D triangulation $\Delta^{(3)}_{C_\phi}$ of the relative mapping cylinder\footnote{Obtained from $\Sigma \times [0,1]$ by collapsing $\{x\} \times [0,1]$ to a point for each $x \in \partial \Sigma$.} $C_\phi$,
where each flip corresponds to a tetrahedron:
\begin{align*}
\begin{tikzpicture}
\draw(0,0) -- (2,0) -- (2,2) -- (0,2) --cycle;
\draw(0,0) -- (2,2);
\draw[thick,->] (2.5,1) --++(3,0);
\begin{scope}[xshift=6cm]
\draw(0,0) -- (2,0) -- (2,2) -- (0,2) --cycle;
\draw(2,0) -- (0,2);
\end{scope}
\begin{scope}[xshift=3cm,yshift=-2cm]
\draw(0,0) -- (2,0) -- (2,2) -- (0,2) --cycle;
\draw(0,0) -- (2,2);
\filldraw[white] (1,1) circle(0.5cm);
\draw(2,0) -- (0,2);
\end{scope}
\end{tikzpicture}
\end{align*}
Let $\{K_\tau\}_{\tau=1,\dots,\ell(\underline{\omega})}$ denote the set of tetrahedra of $C_\phi$, where $\ell(\underline{\omega})$ is the length of $\underline{\omega}$. The mapping cylinder $C_\phi$ is obtained from these tetrahedra by 2-face pairings:
\begin{align*}
    C_\phi = \bigsqcup_{\tau=1}^{\ell(\underline{\omega})} K_\tau / \sim.
\end{align*}
Let $v(K_\tau), e(K_\tau), t(K_\tau)$ denote the set of vertices, edges and 2-faces (triangles) of $K_\tau$, respectively. We use the similar notation $v(C_\phi), e(C_\phi), t(C_\phi)$ for the triangulated relative mapping cylinder $C_\phi$. Let $\pi: \bigsqcup_\tau e(K_\tau) \to e(C_\phi)$ denote the projection.

For a detailed comparison with the reduced QH operator of Baseilhac--Benedetti, we need to explain the relation to their defining data, the \emph{weak branching} and the \emph{quantum shape parameters}. 

\subsubsection{Branchings}
We refer to \cite[Section 2]{BB-AGT} for several notions of branchings. 

\smallskip

\paragraph{\textbf{Pre-branching}}
Recall that a \emph{pre-branching} of a tetrahedron is a choice of co-orientation of 2-faces such that exactly two of them are incoming, and the others are outgoing. In our case, the tangent vector field $d/dt$ in the $[0,1]$-direction of $C_\phi$ naturally determines a pre-branching of each tetrahedron $K_\tau$. Obviously they are compatible with the 2-face pairings, and hence define a pre-branching of the 3D triangulation $\Delta^{(3)}_{C_\phi}$. 

\smallskip

\paragraph{\textbf{Weak branching}}
Recall that an ordering of vertices of a tetrahedron by $\{0,1,2,3\}$ is equivalent to fixing a \emph{(local) branching}, which is an orientation on the edges such that the $j$-th vertex $v_j$ has exactly $j$ incoming edges for $j=0,1,2,3$. It is a finer structure than a pre-branching, and we have a 4:1 correspondence $\{\text{branching}\} \to \{\text{pre-branching}\}$ (which is 2:1 if we fix an orientation of the tetrahedron): see \cite{BB-AGT}. 
In our case, the data of dots specify a branching of $K_\tau$ by the following rule:
\begin{align}
\begin{tikzpicture}
\draw[->-={0.125}{},-<-={0.375}{},->-={0.625}{},->-={0.875}{}](0,0) -- (2,0) -- (2,2) -- (0,2) --cycle;
\draw[-<-] (0,0) -- (2,2);
\dast{(0.15,1.85)};
\dast{(0.25,0.1)};
\draw[thick,->] (2.5,1) --++(3,0);
\begin{scope}[xshift=6cm]
\draw[->-={0.125}{},-<-={0.375}{},->-={0.625}{},->-={0.875}{}](0,0) -- (2,0) -- (2,2) -- (0,2) --cycle;
\draw[-<-] (2,0) -- (0,2);
\dast{(0.15,0.15)};
\dast{(0.25,1.9)};
\end{scope}
\begin{scope}[xshift=3cm,yshift=-2cm]
\draw[->-={0.125}{},-<-={0.375}{},->-={0.625}{},->-={0.875}{}](0,0) -- (2,0) -- (2,2) -- (0,2) --cycle;
\draw[-<-={0.2}{}] (0,0) -- (2,2);
\filldraw[white] (1,1) circle(0.3cm);
\draw[-<-] (2,0) -- (0,2);
\node[scale=0.8] at (2.2,2.2) {$0$};
\node[scale=0.8] at (-0.2,2.2) {$1$};
\node[scale=0.8] at (-0.2,-0.2) {$2$};
\node[scale=0.8] at (2.2,-0.2) {$3$};
\node[scale=0.9] at (1,2.5) {$E_1^\tau$};
\node[scale=0.9] at (-0.5,1) {$E_0^\tau$};
\draw[dashed] (1.7,1.7) --++(1,-0.5) node[right,scale=0.9]{$E_2^\tau$};
\end{scope}
\end{tikzpicture}
\end{align}
Here, the dot corresponds to the non-extremal vertex in each triangle with respect to the branching. These local branchings combine to give a \emph{weak branching} $\widetilde{b}$ of $\Delta^{(3)}_{C_\phi}$. 

\begin{rem}
In the definition of weak branching, the local branchings are not required to be matched under the 2-face pairing, while the induced pre-branchings are. This reflects the fact that the dot rotation $A_v$ breaks the branching but preserves the co-orientation of the 2-face.
\end{rem}

For each branched tetrahedron $K_\tau$, let us label the edges as $e(K_\tau)=\{E_0^\tau,E_1^\tau,E_2^\tau\}$, where 
\begin{align}\label{eq:tetra_edge_label}
    E_0^\tau := [v_1,v_2], \quad E_1^\tau:= [v_0,v_1], \quad E_2^\tau := [v_0,v_2].
\end{align}

\begin{rem}
Here, we choose the labeling \eqref{eq:tetra_edge_label} different from the one in \cite[Section 2]{BB-GD}, so that the quantum shape assignment \eqref{eq:shape_assignment} agrees with our mutation rule of coefficients. This could be also fixed by taking the opposite of our exchange matrix, which in turn corresponds to reversing the orientation of the surface. So we have several equivalent ways to match the conventions. 
\end{rem}

\subsubsection{Quantum shape parameters}
We refer to the explanation in \cite[Section 3.2]{BB-GD}. 
A \emph{quantum shape assignment} is a function
\begin{align*}
    \mathbf{w}: \prod_{\tau} e(K_\tau) \to \bC^\ast
\end{align*}
subject to the following relations:
\begin{enumerate}
    \item (Opposite relation): $\mathbf{w}(E)=\mathbf{w}(E')$ for each opposite pair $(E,E')$ of edges of $K_\tau$.
    \item (Tetrahedron relation): $\mathbf{w}(E_0^\tau)\mathbf{w}(E_1^\tau)\mathbf{w}(E_2^\tau) = -q$ for each $\tau$.
    \item (Edge relation): $\prod_{E \in \pi^{-1}(\alpha)} \mathbf{w}(E) = q^2$ for each edge $\alpha \in e(C_\phi)$.
\end{enumerate}
We call $\mathbf{w}(E)$ the \emph{quantum shape parameter} associated with an edge $E \in e(K_\tau)$. The following is essentially \cite[Proposition 4.1]{NTY}, which gives a ``reverse correspondence'' of \cite[Proposition 4.1]{BB-GD}.

\begin{prop}\label{prop:coeff_shape}
Any $\phi$-invariant coefficient $\bbp_{\underline{\omega}}$ over $\underline{\omega}$ defines a quantum shape assignment $\mathbf{w}=\mathbf{w}(\bbp_{\underline{\omega}})$. 
\end{prop}

\begin{proof}
Write the underlying sequence of flips as $\underline{\omega}=(\mu_{\alpha_1},\dots,\mu_{\alpha_{\ell}})$ with $\ell=\ell(\underline{\omega})$, where $\mu_{\alpha_j}: \tri_j \to \tri_{j+1}$. Let $K_{\tau_j}$ be the tetrahedron arising from $\mu_{\alpha_j}$. 
For each $j=1,\dots,\ell$, let $p_j^\pm:= (p_{\alpha_j}^{\tri_j})^{\pm}$ be the coefficient parameter assigned to the edge $\alpha_j$, and $y_j=p_j^+/p_j^-$ its exchange ratio. Then we define
\begin{align}\label{eq:shape_assignment}
    \mathbf{w}(E_0^{\tau_j}):=p_j^-, \quad \mathbf{w}(E_1^{\tau_j}):=(p_j^+)^{-1}, \quad \mathbf{w}(E_2^{\tau_j}):=-qy_j,
\end{align}
and extend by the opposite relation. Then it satisfies the tetrahedron relation. It remains to prove that it satisfies the edge relation for each edge $\beta \in e(C_\phi)$.

Suppose that $\beta$ appears after the $j_1$-th flip and disappears by the $j_2$-th flip during the sequence $\underline{\omega}$, where $1 \leq j_1 \leq j_2 \leq \ell$. The situation is illustrated as
\begin{align*}
\begin{tikzpicture}
\draw[->-={0.125}{},-<-={0.375}{},->-={0.625}{},->-={0.875}{}](0,0) -- (2,0) -- (2,2) -- (0,2) --cycle;
\draw[-<-={0.2}{}] (0,0) -- (2,2);
\filldraw[white] (1,1) circle(0.3cm);
\draw[-<-,red,thick] (2,0) -- (0,2);
\node[red,scale=0.9] at (0.5,1.8) {$\beta$};
\node[scale=0.9] at (0.4,0.8) {$\alpha_{j_1}$};
\node at (2.5,-0.5) {$K_{\tau_{j_1}}$};
\node at (1,-1) {$j=j_1$};
\draw[dotted,thick] (3,1) -- (5,1);
\begin{scope}[xshift=6cm]
\draw[->-={0.125}{},-<-={0.375}{},->-={0.625}{},->-={0.875}{}](0,0) -- (2,0) -- (2,2) -- (0,2) --cycle;
\draw[-<-={0.2}{},red,thick] (0,0) -- (2,2);
\node[red,scale=0.9,anchor=east] at (0.4,0.6) {$\alpha_{j_2}=\beta$};
\filldraw[white] (1,1) circle(0.3cm);
\draw[-<-] (2,0) -- (0,2);
\node at (1,-1) {$j=j_2$};
\node at (2.5,-0.5) {$K_{\tau_{j_2}}$};
\end{scope}
\end{tikzpicture}
\end{align*}

The assertion is $\prod_{j=j_1}^{j_2} \mathbf{W}_j = q^2$, where 
\begin{align*}
    \mathbf{W}_j:= \prod_{E \in e(K_{\tau_j}) \cap \pi^{-1}(\beta)} \mathbf{w}(E).
\end{align*}
By the assumption on $\beta$, we have $\mathbf{W}_{j_1}=\mathbf{w}(E_2^{\tau_{j_1}})= -q y_{j_1}$ and $\mathbf{W}_{j_2}=\mathbf{w}(E_2^{\tau_{j_2}})= -q y_{j_2}$. 

Let us also introduce the notation $y_\beta[j]:=y_\beta^{\tri_j}$ for $j_1 < j < j_2$. For example, $y_\beta[j_1+1] = y_{j_1}^{-1}$ and $y_\beta[j_2]=y_{j_2}$.  
Then we have the equality 
\begin{align}\label{eq:W_ratio}
    \mathbf{W}_j= \frac{y_\beta[j]}{y_\beta[j+1]}
\end{align}
for $j_1 < j < j_2$, which can be verified by the same case-by-case argument on the value of $\ve^{\tri_j}_{\alpha_j,\beta} \in \{0,\pm 1,\pm 2\}$ as in the proof of \cite[Proposition 4.1]{NTY}. Then the edge relation along $\beta$ is proved as 
\begin{align*}
    \prod_{j=j_1}^{j_2} \mathbf{W}_j = (-q y_{j_1}) \cdot \prod_{j_1 < j < j_2} \frac{y_\beta[j]}{y_\beta[j+1]} \cdot (-q y_{j_2}) = q^2.
\end{align*}
For instance, consider the case $\ve^{\tri_j}_{\alpha_j,\beta}=+1$. Then the mutation formula \eqref{eq:coeff_relation} reads $y_\beta[j+1] = y_\beta[j] \cdot p_j^+$. This case corresponds to the position $\alpha_j=E_2^{\tau_j}$ and $\beta=E_1^{\tau_j}$ (or its opposite edge), and hence $\mathbf{W}_j=(p_j^+)^{-1}$ by \eqref{eq:shape_assignment} in this case. Hence we have \eqref{eq:W_ratio}. When $\ve^{\tri_j}_{\alpha_j,\beta}=-1$, we have $y_\beta[j+1] = y_\beta[j] \cdot (p_j^-)^{-1}$ and $\beta=E_0^{\tau_j}$ (or its opposite edge), and hence $\mathbf{W}_j=p_j^-$. 
The other cases are similar. 
\end{proof}

\subsubsection{Comparison of operators}
Let us fix a mapping class $\phi \in MC(\Sigma)$ and its representation sequence $\omega$ which is full. Fix a coefficient $\bbp_{\underline{\omega}}$ over the underlying flip sequence $\underline{\omega}$. 

\begin{itemize}
    \item The sequence $\omega$ gives rise to the quantum intertwiner ${\opV'}^{\bD}_\phi(\bbp_{\underline{\omega}}) \in End(V'_{\bD})$ in the same way as in \cref{def:intertwiner}. 
    \item Let $\mathcal{H}_N^\mathrm{red}(\Delta^{(3)}_{C_\phi},\widetilde{b},\mathbf{w})^\top \in End((\bC^N)^{t(\tri)})$ denote the transpose of \emph{reduced QH operator} of Baseilhac--Benedetti associated with the 3D triangulation $\Delta^{(3)}_{C_\phi}$ of the relative mapping cylinder, the weak branching $\widetilde{b}$, and the quantum shape assignment $\mathbf{w}=\mathbf{w}(\bbp_{\underline{\omega}})$ defined by $\bbp_{\underline{\omega}}$.  
See \cite{BB-GD} for a detail. 
\end{itemize}
Here we naturally identified the domain $(\bC^N)^{t(\tri)}$ of reduced QH operator with $V'_{\bD}$, as explained in \cite[below (2)]{BB-GD}. 
Recall the mapping torus $M_\phi$ from \cref{def:mapping_torus}. 

\begin{thm}\label{thm:QHFT}
We have $\mathcal{H}_N^\mathrm{red}(\Delta^{(3)}_{C_\phi},\widetilde{b},\mathbf{w})^\top = {\opV'}^{\bD}_\phi(\bbp_{\underline{\omega}})$. In particular, the absolute value of the reduced QHI of the mapping torus $M_\phi$ coincides with the quantum trace $Z_N(\phi;\bbp_{\underline{\omega}})$. 
\end{thm}

\begin{proof}
The second statement follows from the first statement and \cref{prop:two_realizations}. 

The flip operator \eqref{eq:flip_op_BB} is the same as the tetrahedron operator (the \emph{basic matrix dilogarithm}) in the Baseilhac--Benedetti's QHFT \cite[(11)]{BB-GD},
where the correspondence is
\begin{align*}
    \opU_w = A_0 \otimes 1, \quad \opP_w = A_1 \otimes 1, \quad \opU_v = 1 \otimes A_0, \quad \opP_v = 1\otimes A_1, 
\end{align*}
and \eqref{eq:shape_assignment}. 
Indeed, we have
\begin{align*}
    \opS'_{vw} = N^{-1} \sum_{i,j \in \bZ_N} q^{2ij} \opU_w^i [\opU_v\opP_v]^j = N^{-1} \sum_{i,j \in \bZ_N} q^{2ij} A_0^i (qA_1A_0)^j = \Upsilon,
\end{align*}
where we replace $j \mapsto -j$ in the last equality. Moreover, the relation $\mathbf{w}_2^N = 1 - \mathbf{w}_0^N$ in \cite[Lemma 3.7]{BB-GD} is equivalent to $(p^+)^N + (p^-)^N=1$, and the difference equation \cite[(12)]{BB-GD} agrees with our first formula in \cref{lem:q-diff_eq} by identifying $U=-qX$. With the notice that $-A_0A_1 \otimes A_1^{-1}= -q[\opU_w\opP_w\opP_v^{-1}]$, the dilogarithm parts also agree. 

Observe that the way of identification, the map $\iota_\phi$ in our formulation, is the same as the way explained right after \cite[(2)]{BB-GD}. Therefore we get $\mathcal{H}_N^\mathrm{red}(\Delta^{(3)}_{C_\phi},\widetilde{b},\mathbf{w})^\top = {\opV'}^{\bD}_\phi(\bbp_{\underline{\omega}})$ as an endomorphism on $V_{\bD}\cong (\bC^N)^{|t(\tri)|}$. 
\end{proof}

%% file: 6_MCG_irreducible.tex
\section{Invariants of mapping classes II: Irreducible quantum trace}\label{sec:trace_irred}
In this section, we introduce our second invariant of mapping classes arising from the irreducible cyclic \Teich\ theory constructed in \cref{thm:functor_irrep}, which we call the \emph{irreducible quantum trace}.
We also describe a conjectural relation to the Bonahon--Liu intertwiner \cite{BL}. 

\subsection{Irreducible quantum trace}
Let $V_\ast(\bbp_\Sigma;L,\boldsymbol{\lambda}): \mathrm{Pt}^{\tdot}_\Sigma \to \mathrm{Vect}_\bC^\op$ denote the irreducible cyclic quantum \Teich\ theory  associated with a global coefficient $\bbp_\Sigma \in \cF_\Sigma$, a Lagrangian $L \subset H_1(\Sigma^\mathrm{cl};\bZ)$ and a character $\boldsymbol{\lambda}:\widehat{L} \to \bZ_N$ (\cref{thm:functor_irrep}). 

The mapping class group $MC(\Sigma)$ naturally acts on $H_\Sigma$. 
Recall that for a mapping class $\phi \in MC(\Sigma)$, we have an associated local intertwiner $\opV_\phi^{\bD}(\bbp_{\underline{\omega}}) : V_\bD \to V_\bD$ if $\bbp_{\underline{\omega}}$ is a $\phi$-invariant coefficient over a flip sequence $\underline{\omega}$ representing $\phi$. 

\begin{thm}\label{thm:polarized_intertwiner}
Let $\phi \in MC(\Sigma)$ be a mapping class, and $\bbp_{\underline{\omega}}$ a $\phi$-invariant coefficient over a flip sequence representing $\phi$. 
Given a Lagrangian sub-lattice $L \subset H_1(\Sigma^{\mathrm{cl}};\bZ)$ and an additive character $\boldsymbol{\lambda}:\widehat{L} \to \bZ_N$, the local intertwiner $\opV_\phi^{\bD}(\bbp_{\underline{\omega}})$ induces an intertwiner of $\X_{\by^\tri}^{\tri}$-modules
\begin{align*}
    \opV_{\phi}^{\bD}(\bbp_{\underline{\omega}};L,\boldsymbol{\lambda}): V_{\bD}(L,\boldsymbol{\lambda}) \xrightarrow{\sim} V_{\bD}(\phi(L),\phi_\ast\boldsymbol{\lambda}),
\end{align*}
where $\phi_\ast\boldsymbol{\lambda}: \phi(\widehat{L}) \to \bZ_N$ is given by $\phi_\ast\boldsymbol{\lambda}(\phi([c]))=\boldsymbol{\lambda}([c])$ for $[c] \in \widehat{L}$.
\end{thm}

\begin{proof}
By combining \cref{lem:homology_naturality,lem:homology_MCG}, we get the commutative diagram
\begin{equation*}
\begin{tikzcd}
    \cW_{\bD} \ar[r,"\Ad(\opV_\phi^{\bD}(\bbp_{\underline{\omega}}))"] & \cW_\bD \\
    H_\Sigma \ar[u,"\oph^{\bD}"] \ar[r,"\phi"'] & H_\Sigma \ar[u,"\oph^{\bD}"'].
\end{tikzcd}
\end{equation*}
In particular, we have $\opV_\phi^{\bD}(\bbp_{\underline{\omega}}) \oph^{\bD}_{[c]} = \oph^{\bD}_{\phi([c])} \opV_\phi^{\bD}(\bbp_{\underline{\omega}})$ for any $[c] \in \widehat{L}$. Therefore, the intertwiner $\opV_\phi^{\bD}(\bbp_{\underline{\omega}})$ maps each simultaneous eigenspace for the operators in $\oph^{\bD}(\widehat{L})$ to that for the operators in $\oph^{\bD}(\phi(\widehat{L}))$, preserving the weights. The assertion is proved. 
\end{proof}

\subsubsection*{Change of polarization}
In view of \cref{thm:polarized_intertwiner}, we need an operator that adjusts the polarization given by $\phi(L)$ to the original one.
As a slight variant of $\gamma$-operators (\cref{def:gamma_op}), we introduce the \emph{transvection operators} 
\begin{align}\label{eq:transvection}
    \gamma(\oph_a^{\bD})^{\pm 1/2} :=\frac 1 N \sum_{i,j \in \bZ_N}q^{\mp 2ij} (\oph_{a}^{\bD})^{\frac{i+j}{2}}
\end{align}
for $a \in H_\Sigma$. Here $(\oph_{a}^{\bD})^{1/2}$ is the square-root operator in the sense of \cref{def:square-roots}. 
\begin{lem}
We have
\begin{itemize}
    \item If $\ket{\lambda}_a$ is a $q^{2\lambda}$-eigenvector of $\oph_a^{\bD}$, then $\gamma(\oph_a^{\bD})^{\pm 1/2} \ket{\lambda}_a = \gamma(\lambda)^{\pm M} \ket{\lambda}_a$. 
    \item $\gamma(\oph_a^{\bD})^{\pm 1/2} \oph_b^{\bD} = \oph_{b\pm \bi(a,b)a}^{\bD}\gamma(\oph_a^{\bD})^{\pm 1/2}$ for any $b \in H_\Sigma$.
\end{itemize}
\end{lem}

\begin{proof}
We drop the superscript $\bD$.

(1): The proof is similar to that of \cref{lem:gamma_op} (1):
\begin{align*}
    \gamma(\oph_a)^{1/2} \ket{\lambda}_a &= \frac 1 N \sum_{i,j \in \bZ_N}q^{- 2ij} \oph_{a}^{\frac{i+j}{2}} \ket{\lambda}_a \\
    &= \bigg(\frac 1 N \sum_{i,j \in \bZ_N}q^{- 2ij+2M\lambda(i+j)} \bigg)\ket{\lambda}_a & \mbox{(recall $q^{1/2}=q^M$)}\\
    &= \bigg( \sum_{i \in \bZ_N}\delta_{i,M\lambda} q^{2M\lambda i} \bigg)\ket{\lambda}_a \\
    &= \gamma(\lambda)^M \ket{\lambda}_a.
\end{align*}
The case with the other sign is similar. 

(2): By \cref{lem:homology_intersection}, we obtain
\begin{align*}
    \gamma(\oph_a)^{1/2} \oph_{b} &= \oph_{b}\sum_{i,j \in \bZ_N}q^{- 2ij} q^{2(i+j)\bi(a,b)} \oph_{a}^{\frac{i+j}{2}} \\
    &= \oph_{b}\sum_{i,j \in \bZ_N}q^{- 2(i- \bi(a,b))(j- \bi(a,b))+ 2\bi(a,b)^2} \oph_{a}^{\frac{i+j}{2}} \\
    &= q^{2\bi(a,b)^2} \oph_{b}\oph_a^{\bi(a,b)}\gamma(\oph_a)^{1/2} & \mbox{(shift $i$, $j$ by $\bi(a,b)$)} \\
    &= \oph_{b+\bi(a,b)a} \gamma(\oph_a)^{1/2}.
\end{align*} 
The case with the other sign is similar. 
\end{proof}

The image of the homology action $MC(\Sigma) \to \mathrm{Aut}(H_1(\Sigma^{\mathrm{cl}};\bZ))$ is isomorphic to the symplectic group $Sp(2g,\bZ)$. It is well-known that $Sp(2g,\bZ)$ is generated by transvections $b \mapsto b+\bi(a,b)a$ for $a \in H_1(\Sigma^{\mathrm{cl}};\bZ)$. The equation above tells us that the conjugation by $\gamma(\oph_a^{\bD})^{1/2}$ implements the transvection. 
Hence, for any Lagrangian $L$ and a mapping class $\phi \in MC(\Sigma)$, we can find an operator $F_\phi^{\bD} \in \cW_{\bD}$ satisfying
\begin{align*}
    \opF_\phi^{\bD}(V_{\bD}(\phi(L),\phi_\ast\boldsymbol{\lambda})) = V_{\bD}(L,\boldsymbol{\lambda})
\end{align*}
as a composite of the transvection operators $\gamma(\oph_a^{\bD})^{1/2}$. 

\begin{lem}\label{lem:transvection_normalized}
$|\det \opF_\phi^{\bD}|=1$. 
\end{lem}

\begin{proof}
It suffices to prove that $|\det \gamma(\oph_a)^{1/2}|=1$ for any $a\in H_\Sigma$. This follows from
\begin{align*}
    \det \gamma(\oph_a)^{1/2} = \prod_{\lambda \in \bZ_N} \gamma(\lambda)^M = \zeta_\inv^{-NM}
\end{align*}
and the fact that $|\zeta_\inv|=1$, where we used \cref{lem:gamma_prod}.
\end{proof}

\begin{dfn}\label{def:reduced_intertwiner}
The composite operator 
\begin{align*}
    \overline{\opV}_\phi^{\bD}(\bbp_{\underline{\omega}};L,\lambda):=\opF_\phi^{\bD} \circ V_\phi^{\bD}(\bbp_{\underline{\omega}};L,\boldsymbol{\lambda}): V_{\bD}(L,\boldsymbol{\lambda}) \xrightarrow{\sim} V_{\bD}(L,\boldsymbol{\lambda})
\end{align*}
is referred to as the \emph{irreducible intertwiner} of the mapping class $\phi$. The absolute value of the trace 
\begin{align*}
\overline{Z}_N(\phi;\bbp_{\underline{\omega}};L,\boldsymbol{\lambda}) := |\mathrm{Tr} (\overline{V}_\phi^{\bD}(\bbp_{\underline{\omega}};L,\boldsymbol{\lambda}))|
\end{align*}
is called the \emph{irreducible quantum trace}. 
\end{dfn}
Observe that $\overline{\opV}_\phi^{\bD}(\bbp_{\underline{\omega}};L,\lambda)$ is still an intertwiner for $\X_{\by^{\tri}}^{\tri}$, since the homology operators (and hence $\opF_\phi^{\bD}$) commute with the cluster Poisson variables. 

\begin{rem}
Since the stabilizer subgroup of a Lagrangian $L$ is a non-trivial parabolic subgroup, the choice of $F_\phi^{\bD}$ is not unique. To be more precise, we need to fix a path in the \emph{symplectic Ptolemy groupoid} as in \cite{Kim24}. In practice, we can find a simplest one. We postpone the discussion of the dependence of the irreducible intertwiner on this choice for a future work.
\end{rem}

\subsection{Relation with the Bonahon--Liu intertwiner}\label{subsec:BL}
Here we discuss the relation with the Bonahon--Liu intertwiner \cite{BL}. The latter is the same as the Bonahon--Wong intertwiner \cite{BW15,BWY-1,BWY-2}), due to \cite[Theorem 16]{BWY-1}. The author thanks St\'ephane Baseilhac for discussion. 

For any irreducible representation $V$ of $\X_{\by^\tri}^\tri$, an $\X_{\by^\tri}^\tri$-intertwiner on $V$ is unique up to complex scalars by Schur's lemma. Bonahon--Liu defined the intertwiner associated with a mapping class by normalizing the absolute value of its determinant to be $1$. Since our irreducible intertwiner is an $\X_{\by^\tri}^\tri$-intertwiner, it suffices to prove that it satisfies the determinant normalization condition.

\begin{conj}\label{conj:BW_intertwiner}
Each irreducible intertwiner satisfies $|\det \overline{\opV}_\phi^{\bD}(\bbp_{\underline{\omega}};L,\boldsymbol{\lambda})|=1$. In particular, it coincides with the Bonahon--Liu intertwiner of $\phi$ associated with the irreducible representation $V_{\bD}(L,\boldsymbol{\lambda})$. 
\end{conj}

Below we give an observation that supports the conjecture in general, and proves it in the case of once-punctured closed surfaces. 

Let $L_0 \subset \widehat{L}$ be the sub-lattice generated by the classes $[c_p]$ for punctures $p$ and the classes $[c_b]$ for boundary components $b$, so that 
\begin{align*}
    0 \to L_0 \to \widehat{L} \to L \to 0
\end{align*}
is exact. 

\begin{lem}\label{lem:character_change}
For a homology class $[c] \in H_\Sigma$, the action of the square-root homology operator $g_{[c]}^{\bD}:=(\oph_{[c]}^{\bD})^{1/2}$ on $V_{\bD}$ maps $V_{\bD}(L;\boldsymbol{\lambda})$ to $V_{\bD}(L;\boldsymbol{\lambda'})$, where 
\begin{align*}
    \boldsymbol{\lambda'}([c'])=\boldsymbol{\lambda}([c']) + \bi([c],[c'])
\end{align*}
for $[c'] \in \widehat{L}$. In particular, $\boldsymbol{\lambda'}|_{L_0}=\boldsymbol{\lambda}|_{L_0}$. Moreover, any two characters $\boldsymbol{\lambda},\boldsymbol{\lambda'}$ on $\widehat{L}$ such that $\boldsymbol{\lambda'}|_{L_0}=\boldsymbol{\lambda}|_{L_0}$ are related in this way. 
\end{lem}

\begin{proof}
Omit the superscript $\bD$. 
From \cref{lem:homology_intersection}, we get $g_{[c]}\oph_{[c']}=q^{2\bi([c],[c'])}\oph_{[c']}g_{[c]}$ for any $[c'] \in \widehat{L}$. Hence, the action of $g_{[c]}$ shifts the eigenvalue $\lambda$ of $\oph_{[c']}$ to $\lambda q^{2\bi([c],[c'])}$. This proves the first assertion.

For the second assertion, let $a_1,\dots,a_g$ be a generator of $L$. Then we can find $b_1,\dots,b_g \in H_1(\Sigma^{\mathrm{cl}};\bZ)$ such that $\bi(a_i,b_j)=\delta_{ij}$. Therefore, any two characters $\boldsymbol{\lambda},\boldsymbol{\lambda'}$ on $\widehat{L}$ such that $\boldsymbol{\lambda'}|_{L_0}=\boldsymbol{\lambda}|_{L_0}$ are related by the action of $g_{[c]}$ with a suitable linear combination $[c]$ of $b_1,\dots,b_g$.
\end{proof}

\begin{prop}\label{prop:intertwiner_conjugate}
For any two characters $\boldsymbol{\lambda},\boldsymbol{\lambda'}$ on $\widehat{L}$ such that $\boldsymbol{\lambda'}|_{L_0}=\boldsymbol{\lambda}|_{L_0}$, the operators $\overline{\opV}_\phi^{\bD}(\bbp_{\underline{\omega}};L,\boldsymbol{\lambda})$ and $\overline{\opV}_\phi^{\bD}(\bbp_{\underline{\omega}};L,\boldsymbol{\lambda'})$ are conjugate to each other. In particular, they have the same determinant.
\end{prop}

\begin{proof}
Omit the superscript $\bD$. 
Let $[c] \in H_\Sigma$ be a homology class that transforms $\boldsymbol{\lambda}$ to $\boldsymbol{\lambda'}$ as in \cref{lem:character_change}. By the proof of \cref{thm:polarized_intertwiner}, $g_{[c]}$ commutes with $\opV_\phi(\bbp_{\underline{\omega}})$, and thus $\Ad_{g_{[c]}}(\opV_\phi(\bbp_{\underline{\omega}};L,\boldsymbol{\lambda}))=\opV_\phi(\bbp_{\underline{\omega}};L,\boldsymbol{\lambda'})$. 
Moreover, $g_{[c]}$ also commutes with $F_\phi^{\bD}$, since the both involve only the homology operators associated to $b_1,\dots,b_g$. 
Therefore, the two operators are conjugate to each other. 
\end{proof}

\begin{cor}\label{cor:conj_once_punc}
For a once-punctured closed surface $\Sigma$, \cref{conj:BW_intertwiner} holds true.
\end{cor}

\begin{proof}
Let $p$ be the unique puncture of $\Sigma$. Then we have $\oph_{[c_p]}^{\bD}=\iota_{\bD}(H)^2$, since the edge path representing $[c_p]$ traverses every edge twice. In particular, its character is fixed to be the square of the central load $h$ (see \cref{rem:central_load}). Hence, any additive character $\boldsymbol{\lambda}:\widehat{L} \to \bZ_N$ that contributes to a non-zero irreducible component is determined by its restriction to $L$. Therefore, \cref{prop:normalization_total,prop:intertwiner_conjugate} imply that $|\det\overline{\opV}_\phi^{\bD}(\bbp_{\underline{\omega}};L,\boldsymbol{\lambda})|=1$ for any character $\boldsymbol{\lambda}$.
\end{proof}

\begin{rem}
\begin{enumerate}
    \item The action of $g_{[c]}$ on $V_{\bD}$ is exactly the same as the $H_\Sigma$-action introduced by Mazzoli \cite{Mazzoli}, which is used to parametrize the canonical set of intertwiners of local representations.
    \item Fix an additive character $\boldsymbol{\nu}: L_0 \to \bZ_N$ (\emph{peripheral weights}). Let us consider the direct sum
    \begin{align*}
        V_{\bD}(L;\boldsymbol{\nu}):= \bigoplus_{\substack{\boldsymbol{\lambda} \\ \boldsymbol{\lambda}|_{L_0}=\boldsymbol{\nu}}} V_\bD(L;\boldsymbol{\lambda}),
    \end{align*}
    which corresponds to an \emph{isotypic component} in \cite{BB-GD}. Our \cref{prop:intertwiner_conjugate} shows that each irreducible block of the intertwiner contained in an isotypic component are conjugate to each other. 
\end{enumerate}
 
\end{rem}


\subsection{Examples}\label{subsub:example_torus}
Let $\Sigma=\Sigma_1^1$ be a once-punctured torus, and take the dotted triangulation $\bD$ as shown in \cref{fig:torus}. In this case, $\oph_{[c_p]}=\mathsf{1}$ on $V_{\bD}$ for the unique puncture. Let $a,b \in H_1(\Sigma^{\mathrm{cl}};\bZ)$ be the homology classes represented by the horizontal and vertical curves, respectively. 

\begin{figure}[ht]
    \centering
\begin{tikzpicture}[scale=1.3]
\draw (0,0) -- (2,0) -- (2,2) -- (0,2) --cycle; 
\draw (0,2) -- (2,0); 
\draw[red](0.5,0.5) --++(-90:0.8) node[below,scale=0.7]{$\opU_v^{-1}$};
\draw[red](0.5,0.5) node[right,scale=0.7]{$\dbra{\opU_v\opP_v^{-1}}$} --++(180:0.8) node[left,scale=0.7]{$\opP_v$};
\draw[red](1.5,1.5) --++(0:0.8) node[right,scale=0.7]{$\opU_w^{-1}$};
\draw[red](1.5,1.5)++(-0.2,-0.2) node[left,scale=0.7]{$\opP_w$};
\draw[red](1.5,1.5) --++(90:0.8) node[above,scale=0.7]{$\dbra{\opU_w\opP_w^{-1}}$};
\draw[red](0.5,0.5) -- (1.5,1.5);
\node[above left,scale=0.8] at (0.5,0.5) {$v$};
\node[below right,scale=0.8] at (1.5,1.5) {$w$};
\dast{(0.15,0.15)};
\dast{(1.9,0.25)};
\node[scale=0.8] at (0.7,1.5){$1$};
\node[left,scale=0.8] at (0,1){$2$};
\node[below,scale=0.8] at (1,0){$3$};
\node[right,scale=0.8] at (2,1){$2$};
\node[above,scale=0.8] at (1,2){$3$};
\end{tikzpicture}
    \caption{A dotted triangulation of $\Sigma_1^1$. Here the edges of the same label are identified.}
    \label{fig:torus}
\end{figure}
The homology operators associated with these classes are computed as
\begin{align*}
    \oph_a=[\opU_v\opU_w\opP_w^{-1}], \quad \oph_b = \opU_w\opP_v.
\end{align*}

\begin{prop}\label{prop:torus_basis}
Given $\lambda,\mu \in \bZ_N$, the solution spaces of the equations $\oph_a\ket{\psi} = q^{2\lambda}\ket{\psi}$, $\oph_b\ket{\psi} = q^{2\mu}\ket{\psi}$ are $N$-dimensional, and spanned by
\begin{align*}
    \ket{a_{\lambda,m}}:=& 
     \ket{m,\mbb{m-\lambda}}, && m \in \bZ_N\\
    \ket{b_{\mu,n}}:=& 
    \ket{\mb{\mu-n},n}, && n \in \bZ_N,
\end{align*}
respectively.  
Here the first (resp. second) component corresponds to $v$ (resp. $w$). These solutions are normalized so that 
\begin{align*}
    \braket{a_{\lambda,m}}{a_{\lambda,m'}}=N\delta_{mm'}\delta_{\lambda\lambda'},\quad \braket{b_{\mu,n}}{b_{\mu,n'}}=N\delta_{nn'}\delta_{\mu\mu'},
\end{align*}
and the base-change among them is given by
    \begin{align}\label{eq:torus_polarization_change}
    \braket{b_{\mu,n}}{a_{\lambda,m}} =\gamma(n)^{-1} q^{-4nm+2\lambda n + 2\mu m}.
\end{align}
\end{prop}
The proof is straightforward by using \cref{lem:slant_diagonal}. In particular, if we choose the Lagrangian to be $L_a:=\langle a\rangle \subset H_1(\Sigma^{\mathrm{cl}};\bZ)$, we get the polarized space
\begin{align*}
    V_{\bD}(L_a,\lambda) = \bigoplus_{m \in \bZ_N} \bC \ket{a_{\lambda,m}} \subset V_{\bD}.
\end{align*}
The cluster Poisson variables act on $V_{\bD}(L_a,\lambda)$ by
\begin{align*}
    X_1^\tri \ket{a_{\lambda,m}} &= y_1 q^{4m-2\lambda-2} \ket{a_{\lambda,m-1}}, \\
    X_2^\tri \ket{a_{\lambda,m}} &= y_2 \ket{a_{\lambda,m+1}}, \\
    X_3^\tri \ket{a_{\lambda,m}} &= y_3 q^{2\lambda-4m} \ket{a_{\lambda,m}}.
\end{align*}



\paragraph{\textbf{Irreducible intertwiner for $T_a$}}
Recall the Dehn twist operator $\opV_{T_a}^{\bD}(\bbp_a)=\opT_{vw}\opA_w^{-1}\opP_{(vw)}$ from \eqref{eq:Dehn_operator}. First note that it commutes with $\oph_{a}^{\bD}$:
\begin{align*}
    \opV_{T_a}^{\bD}(\bbp_a) \oph_{a}^{\bD} &= \opT_{vw} \opA_w^{-1} [\opU_w\opU_v \opP_v^{-1}] \opP_{(vw)} \\
    &= \opT_{vw} [\opP_w^{-1}\opU_v \opP_v^{-1}]  \opA_w^{-1} \opP_{(vw)} \\
    &= \Psi_{\bp}([\opP_v^{-1}\opU_v\opP_w]) [\opP_w^{-1}\opU_v \opU_w] \opS_{vw}  \opA_w^{-1} \opP_{(vw)} = \oph_{a}^{\bD} \opV_{T_a}^{\bD}(\bbp_a), 
\end{align*}
where we have used \cref{lem:A_comm} and \eqref{eq:S_comm}, and the fact that $[\opP_v^{-1}\opU_v\opP_w]$ commutes with $[\opP_w^{-1}\opU_v \opU_w]$. Hence $\opV_{T_a}^{\bD}(\bbp_a)$ induces a linear map $V_{T_a}^{\bD}(\bp_a;L_a,\lambda)$ on the eigenspace $V_{\bD}(L_a,\lambda)$.

\begin{prop}[Proof in \cref{subsec:reduced_Dehn_a}]\label{prop:reduced_Dehn_a}
The matrix coefficients of the irreducible intertwiner $\overline{V}_{T_a}^{\bD}(\bbp_a;L_a,\lambda)=V_{T_a}^{\bD}(\bbp_a;L_a,\lambda)$ are given by
\begin{align*}
    \bra{a_{\lambda,k}} \overline{V}_{T_a}^{\bD}(\bbp_a;L_a,\lambda) \ket{a_{\lambda,\ell}} =i^{(1-N)/3}\frac{\gamma(\ell)}{\gamma(k)}\frac{\gamma(\ell-\lambda)}{\gamma(k-\lambda)} q^{2\ell(-\ell+\lambda)} \sum_{u \in \bZ_N}\Psi_{\bp_1}(q^{2u}) \gamma(u)^{-1}q^{2u(k-\ell)},
\end{align*}
and the irreducible quantum trace is 
\begin{align}\label{eq:red_traca_a}
    \overline{Z}_N (T_a;\bbp_a;L_a,\lambda)= \frac{1}{N^{1/2}} \left| \sum_{k \in \bZ_N} \Psi_{\bp_1}(q^{2k})\gamma(k)^{-1} \right|.
\end{align}
\end{prop}




\paragraph{\textbf{Irreducible intertwiner for $T_b^{-1}$}}
Recall $\opV_{T_b^{-1}}^{\bD}(\bbp_b)=\opT_{vw}$ from \eqref{eq:Dehn_operator}. It is easy to verify the relation
\begin{align*}
    \opV_{T_b^{-1}}^{\bD}(\bbp_b) \oph_a^{\bD} = \oph_{a+b}^{\bD} \opV_{T_b^{-1}}^{\bD}(\bbp_b).
\end{align*}
Then we take $\opF_{T_b^{-1}}^{\bD}:= \gamma(\oph_b)^{1/2}$, and irreducible intertwiner $\overline{\opV}_{T_b^{-1}}^{\bD}(\bbp_b;L_a,\lambda)=\gamma(\oph_b)^{1/2}\opT_{vw}$ satisfies the relation
\begin{align*}
    \overline{\opV}_{T_b^{-1}}^{\bD}(\bbp_b;L_a,\lambda) \oph_a^{\bD} = \oph_a^{\bD} \overline{\opV}_{T_b^{-1}}^{\bD}(\bbp_b;L_a,\lambda).
\end{align*}

\begin{prop}[Proof in \cref{subsec:reduced_Dehn_b}]\label{prop:Dehn_b}
The matrix coefficients of the irreducible intertwiner $\overline{\opV}_{T_b^{-1}}^{\bD}(\bbp_b;L_a,\lambda)$ are given by
\begin{align*}
    \bra{a_{\lambda,k}} \overline{\opV}_{T_b^{-1}}^{\bD}(\bbp_b;L_a,\lambda) \ket{a_{\lambda,\ell}} = \zeta^{(2)}_+  \Psi_{\bp_b} (q^{2(2\ell-\lambda)}) \gamma(\ell-k)^{-2},
\end{align*}
and the irreducible quantum trace is
\begin{align}\label{eq:red_traca_b}
    \overline{Z}_N(T_b^{-1};\bbp_b;L_a,\lambda)
    &= \frac{1}{N^{1/2}} \left|\sum_{k \in \bZ_N} \Psi_{\bp_1} (q^{2(2k-\lambda)})\right|.
\end{align}
\end{prop}

\begin{rem}\label{rem:Dehn_BW}
For any function $f(n)$ on $\bZ_N$ with odd $N$, we have 
\begin{align*}
    \sum_{n \in \bZ_N} f(2n) = \sum_{n \in \bZ_N} f(n)
\end{align*}
since $2n$ runs over $\bZ_N$ exactly once as $n$ varies over $\bZ_N$. Therefore \eqref{eq:red_traca_a} and \eqref{eq:red_traca_b} (with $\lambda=0$) coincide with the traces of the operators $(L_{uv})^{-1}$ and $(R_{uv})$ in \cite[Section 4.4]{BWY-1}, respectively.  
\end{rem}

\begin{prop}
The irreducible quantum trace of $\phi=T_b^{-1}T_a$ is given by
\begin{align*}
    \overline{Z}_N(\phi;\bbp_\phi;L_a,\lambda) = \left|\bigg(\sum_{k \in \bZ_N} \Psi_{\bp_1}(q^{2k})q^{-Mk^2} \bigg)\bigg(\sum_{\ell \in \bZ_N} \Psi_{\bp'_2}(q^{2\ell})q^{-M\ell^2} \bigg)   \right|.
\end{align*}
Here we take a $\phi$-invariant coefficient $\bbp_\phi$ as in \eqref{eq:LR_coeff}. 
\end{prop}

\begin{proof}
By using the matrix coefficients in \cref{prop:reduced_Dehn_a,prop:Dehn_b}, 
\begin{align*}
    &\overline{Z}_N(\phi;\bbp_\phi;L_a,\lambda) \\
    =& \frac{1}{N} \left| \sum_{k, \ell \in \bZ_N} \frac{\gamma(\ell)}{\gamma(k)}\frac{\gamma(\ell-\lambda)}{\gamma(k-\lambda)} q^{2\ell(-\ell+\lambda)} \bigg(\sum_{u \in \bZ_N}\Psi_{\bp_1}(q^{2u}) \gamma(u)^{-1}q^{2u(k-\ell)} \bigg) \zeta^{(2)}_+  \Psi_{\bp'_2} (q^{2(2k-\lambda)}) \gamma(k-\ell)^{-2} \right| \\
    =&\frac{1}{N} \left| \zeta^{(2)}_+\zeta^{(2)}_-\sum_{k, u \in \bZ_N}\Psi_{\bp_1}(q^{2u}) \Psi_{\bp'_2} (q^{2(2k-\lambda)})\gamma(k)^{-1}\gamma(k-\lambda)^{-1}\gamma(u)^{-1} q^{2uk}q^{\lambda^2-2k^2} \gamma(-u+2k)^M \right| \\
    =& \left| \sum_{k, u \in \bZ_N}\Psi_{\bp_1}(q^{2u}) \Psi_{\bp'_2} (q^{2(2k-\lambda)})q^{-2k(k-\lambda)}q^{-Mu^2} \right| \\
    =& \left|\bigg(\sum_{u \in \bZ_N} \Psi_{\bp_1}(q^{2u})q^{-Mu^2} \bigg)\bigg(\sum_{\ell \in \bZ_N} \Psi_{\bp'_2}(q^{2\ell})q^{-M\ell^2} \bigg)  \right|.
\end{align*}
Here, in the last line, we changed the variable as $\ell:=2k-\lambda$. 
\end{proof}
Comparing with \cref{prop:LR_local}, we have
\begin{align*}
    Z_N(\phi;\bbp_\phi) = \sum_{\lambda \in \bZ_N} \overline{Z}_N(\phi;\bbp_\phi;L_a,\lambda)
\end{align*}
in this case.

%% file: 7_Proof.tex
\section{Proofs}\label{sec:proof}

\subsection{Proof of the pentagon relation}\label{sec:proof_pentagon}


\subsubsection{Proof of \cref{thm:pentagon}} 
We are going to prove the pentagon relation
\begin{align*}
    \Psi_\bp(\opU) \Psi_{\br'}(\opP) = \Psi_\br (\opP) \Psi_{\bp'}(q^{-1}\opU\opP) \Psi_{\bp''}(\opU),  
\end{align*}
where the parameters satisfy
\begin{align}
    \begin{aligned}\label{eq:pentagon_parameter_rel_proof}
    &p^- = {p'}^- {p''}^-, \quad p^+{r'}^- = {p'}^-{p''}^+, \quad p^+{r'}^+ = {p'}^+, \\
    &{r'}^- = r^-{p'}^-, \quad {r'}^+{p''}^- = r^+, \quad {r'}^+{p''}^+ = r^-{p'}^+.
    \end{aligned}
\end{align}

Let $F:=(\Psi_\bp(\opU) \Psi_{\br'}(\opP))^{-1}\Psi_\br (\opP) \Psi_{\bp'}(q^{-1}\opU\opP) \Psi_{\bp''}(\opU)$. We are going to prove $F \opP F^{-1} = \opP$ and $F \opU F^{-1} = \opU$, which implies that $F$ is a scalar. We refer the reader to the proof of \cite[Proposition 4.1]{Kim19} for a similar computation.\footnote{Here, the cyclic quantum dilogarithm in \cite{Kim19} corresponds to the inverse of ours. See \cite[Lemma 2.19]{Kim19} for the $q$-difference relation.}

By \cref{lem:q-diff_eq}, we compute
\begin{align*}
    &\Psi_{\br'}(\opP)^{-1}\Psi_\bp(\opU)^{-1}\opP^{-1} \\
    &= \Psi_{\br'}(\opP)^{-1}\opP^{-1}(p^- + q^{-1}p^+ \opU)\Psi_\bp(\opU)^{-1} \\
    &= \opP^{-1}(p^- + q^{-1}p^+ \opU({r'}^- + q^{-1}{r'}^+ \opP))\Psi_{\br'}(\opP)^{-1}\Psi_\bp(\opU)^{-1}
\end{align*}
and 
\begin{align*}
    & \Psi_{\bp''}(\opU)^{-1}\Psi_{\bp'}(q^{-1}\opU\opP)^{-1}\Psi_\br (\opP)^{-1}\opP^{-1} \\
    &= \Psi_{\bp''}(\opU)^{-1}\Psi_{\bp'}(q^{-1}\opU\opP)^{-1}\opP^{-1}\Psi_\br (\opP)^{-1} \\
    &= \Psi_{\bp''}(\opU)^{-1}\opP^{-1}({p'}^- + q^{-1}{p'}^+(q^{-1}\opU\opP))\Psi_{\bp'}(q^{-1}\opU\opP)^{-1}\Psi_\br (\opP)^{-1} \\
    &= \Psi_{\bp''}(\opU)^{-1}({p'}^- \opP^{-1} + {p'}^+ \opU)\Psi_{\bp'}(q^{-1}\opU\opP)^{-1}\Psi_\br (\opP)^{-1} \\
    &= ({p'}^- \opP^{-1}({p''}^- + q^{-1}{p''}^+ \opU) + {p'}^+\opU)\Psi_{\bp''}(\opU)^{-1}\Psi_{\bp'}(q^{-1}\opU\opP)^{-1}\Psi_\br (\opP)^{-1}.
\end{align*}
Then we claim
\begin{align*}
    \opP^{-1}(p^- + q^{-1}p^+ \opU({r'}^- + q^{-1}{r'}^+ \opP)) ={p'}^-\opP^{-1}({p''}^- + q^{-1}{p''}^+ \opU) + {p'}^+\opU. 
\end{align*}
By the term-wise comparison, it is equivalent to
\begin{align*}
    p^- = {p'}^- {p''}^-, \quad p^+{r'}^- = {p'}^-{p''}^+, \quad p^+{r'}^+ = {p'}^+.
\end{align*}
 Thus we obtain
\begin{align*}
    \Ad_{\Psi_{\br'}(\opP)^{-1}\Psi_\bp(\opU)^{-1}}(\opP^{-1}) = \Ad_{\Psi_{\bp''}(\opU)^{-1}\Psi_{\bp'}(q^{-1}\opU\opP)^{-1}\Psi_\br (\opP)^{-1}}(\opP^{-1}),
\end{align*}
which implies $\Ad_F(\opP)=\opP$. 

Similarly, from the computations
\begin{align*}
    &\Psi_{\bp''}(\opU)\Psi_{\br'}(\opP)^{-1}\Psi_\bp(\opU)^{-1}\opU \\
    &= \Psi_{\bp''}(\opU)\opU({r'}^- +q^{-1} {r'}^+ \opP)\Psi_{\br'}(\opP)^{-1}\Psi_\bp(\opU)^{-1} \\
    &= \opU({r'}^- + q^{-1} {r'}^+ \opP({p''}^- + q {p''}^+ \opU))\Psi_{\bp''}(\opU)\Psi_{\br'}(\opP)^{-1}\Psi_\bp(\opU)^{-1} \\
    &= ({r'}^-\opU + q^{-1} {r'}^+ {p''}^- \opU\opP +q^{-2} {r'}^+{p''}^+\opU^2P)\Psi_{\bp''}(\opU)\Psi_{\br'}(\opP)^{-1}\Psi_\bp(\opU)^{-1}
\end{align*}
and 
\begin{align*}
    & \Psi_{\bp'}(q^{-1}\opU\opP)^{-1}\Psi_\br (\opP)^{-1}\opU \\
    &= \Psi_{\bp'}(q^{-1}\opU\opP)^{-1}\opU(r^- + q r^+ \opP)\Psi_\br (\opP)^{-1} \\
    &= \Psi_{\bp'}(q^{-1}\opU\opP)^{-1}(r^-\opU + q r^+ \opU\opP)\Psi_\br (\opP)^{-1} \\
    &= (r^-\opU({p'}^- + q^{-1}{p'}^+(q^{-1}\opU\opP)) + q r^+ \opU\opP)\Psi_{\bp'}(q^{-1}\opU\opP)^{-1}\Psi_\br (\opP)^{-1} \\
    &= (r^-{p'}^-\opU + q^{-2}r^-{p'}^+\opU^2P + q r^+ \opU\opP)\Psi_{\bp'}(q^{-1}\opU\opP)^{-1}\Psi_\br (\opP)^{-1},
\end{align*}
we claim 
\begin{align*}
    {r'}^- = r^-{p'}^-, \quad {r'}^+{p''}^- = r^+, \quad {r'}^+{p''}^+ = r^-{p'}^+.
\end{align*}
 Thus we obtain $\Ad_F(\opU)=\opU$. Therefore $F=C_{\bp,\br} \in \bC$ must be a scalar. Thus we get
\begin{align}\label{eq:pentagon_scalar}
    \Psi_\bp(U) \Psi_{\br'}(\opP) = C_{\bp,\br}\Psi_\br (\opP) \Psi_{\bp'}(q^{-1}\opU\opP) \Psi_{\bp''}(U). 
\end{align}

Let us show $C_{\bp,\br}=1$.
By \cref{lem:psi_product}, we have $\det \Psi=1$ for each term. Then by taking the determinant of both sides of \eqref{eq:pentagon_scalar}, we get $C_{\bp,\br}^N=1$. Namely, $C_{\bp,\br}=q^{2m}$ for some $m \in \bZ_N$. Since the both sides depend analytically on $\bp,\br$, we see that $C=C_{\bp,\br}$ is independent of $\bp$ and $\br$. 
Observe that the parameters $(r^+,r^-)=(0,1)$, $({r'}^+,{r'}^-)=(0,1)$, $({p'}^+,{p'}^-)=(0,1)$, $({p''}^+,{p''}^-)=(p^+,p^-)$ satisfies the relation \eqref{eq:pentagon_parameter_rel_proof}. Then \eqref{eq:pentagon_scalar} becomes
\begin{align*}
    \Psi_\bp(\opU) = C \Psi_\bp(\opU),
\end{align*}
which shows $C=1$. 
Thus \cref{thm:pentagon} is proved.

\subsubsection{Proof of \cref{thm:parameter_coeff}}
Let us prove that the relations 
\begin{align*}
    &(1)\ p^- = {p'}^- {p''}^-, && (2)\ p^+{r'}^- = {p'}^-{p''}^+, && (3)\ p^+{r'}^+ = {p'}^+, \\
    &(4)\ {r'}^- = r^-{p'}^-, && (5)\ {r'}^+{p''}^- = r^+, && (6)\ {r'}^+{p''}^+ = r^-{p'}^+
\end{align*}
are equivalent to the relations
\newcommand{\num}[1]{\langle #1 \rangle}
\begin{align*}
    &\num{1}\ \frac{{r'}^+}{{r'}^-} = \frac{{r}^+}{{r}^-} (p^-)^{-1}, &
    \num{2}\ \frac{{p''}^-}{{p''}^+} &= \frac{{p}^-}{{p}^+} ({r'}^-)^{-1}, &  
    \num{3}\ \frac{{p'}^-}{{p'}^+} &= \frac{{r'}^-}{{r'}^+} ({p''}^+)^{-1}, \\
    &\num{4}\ \frac{{p''}^+}{{p''}^-} = \frac{{r}^-}{{r}^+} {p'}^+, &
    \num{5}\ \frac{{p'}^+}{{p'}^-} &= \frac{{p}^+}{{p}^-} {r}^+.
\end{align*}

Assume first the relations (1)--(6).

$\num{1}$: 
From (4) and (5), we get
\begin{align}\label{eq:pentagon_1}
    \frac{{r'}^+}{{r'}^-} = \frac{r^+/{p''}^-}{r^-{p'}^-} = \frac{{r}^+}{{r}^-} (p^-)^{-1}.
\end{align}
Here we have also used (1). 

$\num{2}$:  From (1) and (2), we get
\begin{align*}
    \frac{{p''}^-}{{p''}^+} =\frac{p^-/{p'}^-}{p^+{r'}^-/{p'}^-} = \frac{{p}^-}{{p}^+} ({r'}^-)^{-1}.
\end{align*}

$\num{3}$:  From (2) and (3), we get
\begin{align*}
    \frac{{p'}^-}{{p'}^+} = \frac{p^+{r'}^-/{p''}^+}{p^+{r'}^+} = \frac{{r'}^-}{{r'}^+} ({p''}^+)^{-1}.
\end{align*}

$\num{4}$:  From (5) and (6), we get
\begin{align*}
    \frac{{p''}^+}{{p''}^-} = \frac{r^-{p'}^+/{r'}^+}{r^+/{r'}^+} = \frac{{r}^-}{{r}^+} {p'}^+.
\end{align*}

$\num{5}$:  From (3) and (4), we get
\begin{align*}
    \frac{{p'}^+}{{p'}^-} = \frac{p^+{r'}^+}{{r'}^-/r^-} = \frac{{p}^+}{{p}^-} {r}^+.
\end{align*}
Here we have used \eqref{eq:pentagon_1} in the second equality. 

Let us prove the converse direction. Assume the relations $\num{1}$--$\num{5}$. 
Using $\num{3}$ and then $\num{1}$ and $\num{5}$, we obtain
\begin{align}\label{eq:p''+}
    {p''}^+ = \frac{{p'}^+}{{p'}^-}\frac{{r'}^-}{{r'}^+} = \frac{p^+r^+}{p^-} \frac{r^-p^-}{r^+} = p^+r^-.
\end{align}
Substituting this into $\num{2}$ and $\num{4}$, respectively, we get
\begin{align}\label{eq:p''-_2}
    {p''}^- = \frac{p^-{p''}^+}{p^+{r'}^-} = \frac{p^-r^-}{{r'}^-}
\end{align}
and
\begin{align}\label{eq:p''-_4}
    {p''}^- = \frac{{p''}^+r^+}{r^-{p'}^+} = \frac{p^+r^+}{{p'}^+}.
\end{align}
Then \eqref{eq:p''-_2} and \eqref{eq:p''-_4} together imply
\begin{align}\label{eq:p'-r'}
    \frac{{p'}^+}{{r'}^-} = \frac{p^+r^+}{p^-r^-}. 
\end{align}

(1): From \eqref{eq:p''-_4} and $\num{5}$, we get
\begin{align*}
    {p'}^-{p''}^- = p^+r^+ \frac{{p'}^-}{{p'}^+} =p^-.
\end{align*}

(2): From $\num{5}$, \eqref{eq:p''+} and then \eqref{eq:p'-r'}, we get
\begin{align*}
    {r'}^+{p''}^+ = \frac{p^-}{p^+r^+}{p'}^+ \cdot p^+r^- = p^+ {r'}^-.
\end{align*}

(3): From $\num{1}$ and \eqref{eq:p'-r'}, we get
\begin{align*}
    p^+{r'}^+ = p^+ \cdot \frac{r^+}{r^-p^-}{r'}^- = {p'}^+.
\end{align*}

(4): From $\num{5}$ and \eqref{eq:p'-r'}, we get
\begin{align*}
    r^-{p'}^- = r^- \cdot \frac{p^-}{p^+r^+}{p'}^+ = {r'}^-.
\end{align*}

(5): From \eqref{eq:p''-_2} and $\num{1}$, we get
\begin{align*}
    {r'}^+{p''}^- = p^-r^- \frac{{r'}^+}{{r'}^-} = r^+.
\end{align*}

(6): From $\num{1}$, \eqref{eq:p''+} and then \eqref{eq:p'-r'}, we get
\begin{align*}
    {r'}^+{p''}^+ = {r'}^-\frac{r^+}{r^-p^-}\cdot p^+r^- = r^- {p'}^+.
\end{align*}
Thus \cref{thm:parameter_coeff} is proved. 

\subsection{Proof of the relations among $\opA_v$ and $\opT_{vw}$}\label{subsec:AT_relations_proof}
In this section, we give a proof of \cref{thm:AT_relations}.

\subsubsection{Proof of $\opA_v^3=\mathsf{1}$ and the conjugation actions of $\opA_v$ and $\opT_{vw}$}

\begin{lem}\label{lem:A_comm}
The rotation operator $\opA_v$ has order $3$, and acts on the $v$-th component of $V_{\tri'_\ast}$ as
\begin{align*}
    &\ket{k} \mapsto \zeta_A \ket{\mbb{k}}, \\
    &\ket{\mbb{k}} \mapsto \zeta_-\zeta_A\gamma(k) \ket{\mb{-k}}, \\
    &\ket{\mb{k}} \mapsto N\zeta_A\gamma(k)^{-1} \ket{-k}.
\end{align*}
Recall $\zeta_A=(N\zeta_-)^{-1/3}=N^{-1/2}i^{(1-N)/6}$. 
The conjugation action of $\opA_v$ is given by
\begin{align*}
    \opA_v\opU_v \opA_v^{-1} = \dbra{\opP_v \opU_v^{-1}},\quad \opA_v\opP_v \opA_v^{-1} = \opU_v^{-1}.
\end{align*} 
\end{lem}

\begin{proof}
The first equation is nothing but the definition. For the second one,
\begin{align*}
    \opA_v \ket{\mbb{k}} &= \sum_{j \in \bZ_N} \gamma(j)^{-1} q^{-2kj}\opA_v \ket{j} = \zeta_A \sum_{j \in \bZ_N} \gamma(j)^{-1} q^{-2kj}\ket{\mbb{j}} \\
    &= \zeta_A \sum_{j,m \in \bZ_N} \gamma(j)^{-1} q^{-2kj}\gamma(m)^{-1}q^{-2jm}\ket{m} \\
    &= \zeta_A \sum_{j,m \in \bZ_N} \gamma(j+m)^{-1} q^{-2kj}\ket{m} \\
    &= \zeta_A \sum_{m \in \bZ_N} \bigg(\sum_{j \in \bZ_N} \gamma(j)^{-1} q^{-2k(j-m)}\bigg)\ket{m} \\
    &= \zeta_A \sum_{m \in \bZ_N} \zeta_-\gamma(k) q^{2km} \ket{m} \\
    &= \zeta_-\zeta_A\gamma(k) \ket{\mb{-k}}.
\end{align*}
Here we have used $\gamma(j+m)=\gamma(j)\gamma(m)q^{2jm}$ and the Fourier transformation formula of $\gamma(j)^{-1}$. For the third one, 
\begin{align*}
    \opA_v \ket{\mb{k}} &= \sum_{j \in \bZ_N} q^{-2jk}\opA_v\ket{k} 
    = \zeta_A \sum_{j \in \bZ_N} q^{-2jk}\ket{\mbb{k}} \\
    &= \zeta_A \sum_{j,m \in \bZ_N} q^{-2jk}\gamma(m)^{-1}q^{-2jm}\ket{m} \\ 
    &= \zeta_A \sum_{m \in \bZ_N}\gamma(m)^{-1} \bigg(\sum_{j \in \bZ_N} q^{-2j(k+m)}\bigg)\ket{m} \\
    &= \zeta_A \sum_{m \in \bZ_N}\gamma(m)^{-1} N\delta_{k+m,0}\ket{m} \\ 
    &= N\zeta_A \gamma(k)^{-1}\ket{-k}. 
\end{align*}
We get $\opA_v^3=\mathsf{1}$ by these three equations and $\zeta_A^3=(N\zeta_-)^{-1}$. 

From the action on the vectors, we get $\opA_v \opU_v \opA_v^{-1} \ket{\mbb{k}} = q^{2k} \ket{\mbb{k}}$, which shows the first conjugation equation by \cref{lem:slant_diagonal}. Similarly, $\opA_v \opP_v \opA_v^{-1} \ket{k} = q^{-2k} \ket{k}$ shows the second one.
\end{proof}

\begin{rem}
In terms of the matrix coefficients $A_{nm}:=\bra{n}\opA_v \ket{m}$, $\opA_v$ is characterized by the difference equation $A_{nm}= q^{-2n}A_{n,m-1}$ and the normalization $\opA_v^3=\mathsf{1}$. 
\end{rem}

\begin{prop}\label{prop:T_comm}
The conjugation action of $\opT_{vw}$ is given by
\begin{enumerate}
    \item $\opT_{vw} \opU_v = \opU_v \opU_w \opT_{vw}$.
    \item $\opT_{vw}\opP_v \opU_w = \opP_v \opU_w \opT_{vw}$.
    \item $\opT_{vw} \opP_v \opP_w = \opP_w \opT_{vw}$.
    \item $\opT_{vw}\opP_v = ((p_\alpha^\tri)^-\opP_v + (p_\alpha^\tri)^+\opU_v \opP_w) \opT_{vw}$.
\end{enumerate}
\end{prop}

\begin{proof}
Let us write $\bp:=\bp_\alpha^\tri$ for short.
We first note that the conjugation action of $\opS_{vw}$ is given by
\begin{align}
    \begin{aligned}\label{eq:S_comm}
    &\opS_{vw}\opU_v=\opU_v\opU_w \opS_{vw}, & &\opS_{vw}\opP_v=\opP_v \opS_{vw}, \\
    &\opS_{vw}\opU_w=\opU_w \opS_{vw}, & &\opS_{vw}\opP_w=\opP_w\opP_v^{-1} \opS_{vw}. 
    \end{aligned}
\end{align}
For instance, the first one is verified as
\begin{align*}
    \opS_{vw}\opU_v= \opU_v \sum_{i,j} q^{2(i-1)j}\opU_w^i \opP_v^j = \opU_v\opU_w \opS_{vw}.
\end{align*}
Since $\opU_v \opU_w$, $\opP_v \opU_w$ and $\opP_w$ commutes with $\dbra{\opP_v^{-1}\opU_v\opP_w}$, the first three statements are verified as
\begin{align*}
    \opT_{vw} \opU_v &= \Psi_\bp(\dbra{\opP_v^{-1}\opU_v\opP_w})\opU_v \opU_w \opS_{vw} =\opU_v \opU_w \opT_{vw} \\
    \opT_{vw}\opP_v \opU_w &= \Psi_\bp(\dbra{\opP_v^{-1}\opU_v\opP_w})\opP_v \opU_w \opS_{vw} = \opP_v \opU_w \opT_{vw} \\
    \opT_{vw} \opP_v \opP_w &= \Psi_\bp(\dbra{\opP_v^{-1}\opU_v\opP_w})\opP_w \opS_{vw} = \opP_w \opT_{vw}.
\end{align*}
The last one is computed as
\begin{align*}
    \opT_{vw}\opP_v &= \Psi_\bp(\dbra{\opP_v^{-1}\opU_v\opP_w})\opP_v \opS_{vw} =\opP_v \Psi_\bp(q^2\dbra{\opP_v^{-1}\opU_v\opP_w}) \opS_{vw}\\
    &= \opP_v(p^- + p^+\opP_v^{-1}\opU_v \opP_w) \opT_{vw}.
\end{align*}
Here we note $\dbra{\opP_v^{-1}\opU_v\opP_w}=q^{-1}\opP_v^{-1}\opU_v\opP_w$ and use the $q$-difference relation (\cref{lem:q-diff_eq}).
\end{proof}

\subsubsection{Proof of $\opT_{uv}\opT_{uw}\opT_{vw}=\opT_{vw}\opT_{uv}$}

\begin{lem}\label{lem:S_action}
The operator $\opS_{vw}$ acts on the $(v,w)$-th components of $V_{\bD}$ as $\ket{k_v,k_w} \mapsto \ket{k_v-k_w,k_w}$.
\end{lem}

\begin{proof}
\begin{align*}
    \opS_{vw}\ket{k_v,k_w} &= \frac{1}{N} \sum_{i,j \in \bZ_N} q^{2ij} \opU_w^i \opP_v^j \ket{k_v,k_w} = \frac{1}{N}\sum_{i,j \in \bZ_N} q^{2ij} q^{2k_w i} \ket{k_v+2j,k_w} \\
    &= \frac{1}{N}\sum_{j \in \bZ_N} N\delta_{j
    +k_w,0} \ket{k_v+2j,k_w} = \ket{k_v-k_w,k_w}. 
\end{align*}
\end{proof}

In order for the proof of the pentagon relation, let us label the coefficient tuples associated to the edges as in \cref{fig:pentagon_proof}. The five parameters $\bp,\bp',\bp'',\br,\br'$ are related to each other by \eqref{eq:coeff_relation}. 


It is easy to see 
\begin{align}\label{eq:S_pentagon}
    \opS_{uv}\opS_{uw} \opS_{vw} = \opS_{vw} \opS_{uv}.
\end{align}
Indeed, by \cref{lem:S_action}, both sides send $\ket{k_u,k_v,k_w}$ to $\ket{k_u-k_v,k_v-k_w,k_w}$ 
for any $k_u,k_v,k_w \in \bZ_N$. 

Then by using the commutation relations \eqref{eq:S_comm} to move all the $\opS$-operator to the right and by applying \eqref{eq:S_pentagon}, the statement can be rewritten as
\begin{align*}
    &\Psi_\br(\dbra{\opP_u^{-1}\opU_u\opP_v})  
    \Psi_{\bp'}(\dbra{\opP_u^{-1}\opU_u\opU_v\opP_w}) 
    \Psi_{\bp''}(\dbra{\opP_v^{-1}\opU_v\opP_w})=
    \Psi_\bp(\dbra{\opP_v^{-1}\opU_v\opP_w}) 
    \Psi_{\br'}(\dbra{\opP_u^{-1}\opU_u\opP_v}).
\end{align*}
Note that the operators $\opU:=\dbra{\opP_v^{-1}\opU_v\opP_w}$, $\opP:=\dbra{\opP_u^{-1}\opU_u\opP_v}$ satisfy $\opU\opP=q^2 \opP\opU$. Then the asserted equation above can be further rewritten as
\begin{align}\label{eq:pentagon_proof}
    \Psi_\br(\opP)  
    \Psi_{\bp'}(\dbra{\opU\opP}) 
    \Psi_{\bp''}(\opU)=
    \Psi_\bp(\opU) 
    \Psi_{\br'}(\opP).
\end{align}
Now observe that the mutation relation \eqref{eq:coeff_relation} gives exactly the parameter relation \eqref{eq:pentagon_coeff_rel}. Therefore \cref{thm:parameter_coeff} tells us that the relation \eqref{eq:pentagon_proof} is exactly the pentagon relation in \cref{thm:pentagon}. Thus $\opT_{uv}\opT_{uw}\opT_{vw}=\opT_{vw}\opT_{uv}$ holds. 

\subsubsection{Proof of $\opA_v \opT_{vw} \opA_w =\opA_w \opT_{wv} \opA_v$}

\begin{lem}\label{lem:S_symmetry}
Let $\widetilde{\opS}_{vw}:=\opA_v \opS_{vw}\opA_v^{-1}$. Then it acts as $\widetilde{\opS}_{vw}\ket{k_v,k_w} = q^{2k_v k_w} \ket{k_v,k_w}$. In particular, it has the symmetry $\widetilde{\opS}_{vw}=\widetilde{\opS}_{wv}$. 
\end{lem}

\begin{proof}
From \cref{lem:A_comm}, we get $\widetilde{\opS}_{vw}=\frac 1 N \sum_{i,j\in \bZ_N} q^{2ij} \opU_w^i \opU_v^{-i}$. The remaining statements are easily verified. 
\end{proof}

\begin{lem}
$\opT_{vw}= \opS_{vw} \Psi_{\bp_\alpha^\tri}(\dbra{\opU_v\opU_w^{-1}\opP_w})$. 
\end{lem}

\begin{proof}
Follows from \eqref{eq:S_comm}.
\end{proof}

Observe that $\opT_{vw}$ and $\opT_{wv}$ in the statement involves the same coefficient $\bp:=\bp_{\alpha}^\tri$, where $\alpha \in e_{\interior}(\tri)$ is the edge shared by $v,w$ in the initial triangulation. 
By using the previous lemma, the asserted relation is equivalent to
\begin{align*}
    \widetilde{\opS}_{vw}\opA_v \Psi_{\bp}(\dbra{\opU_v\opU_w^{-1}\opP_w}) \opA_w = \widetilde{\opS}_{wv}\opA_w \Psi_{\bp}(\dbra{\opU_w\opU_v^{-1}\opP_v}) \opA_v.
\end{align*}
Then we compute the action of the left-hand side as
\begin{align*}
    \widetilde{\opS}_{vw}\opA_v \Psi_{\bp}(\dbra{\opU_v\opU_w^{-1}\opP_w}) \opA_w \ket{k_v,k_w} &= \zeta_A\widetilde{\opS}_{vw}\opA_v \Psi_{\bp}(\dbra{\opU_v\opU_w^{-1}\opP_w}) \ket{k_v,\mbb{k_w}} \\
    &= \zeta_A\Psi_{\bp}(q^{2(k_v+k_w)})\widetilde{\opS}_{vw}\opA_v  \ket{k_v,\mbb{k_w}} \\
    &= \zeta_A\Psi_{\bp}(q^{2(k_v+k_w)})\widetilde{\opS}_{vw}\ket{\mbb{k_v},\mbb{k_w}}.
\end{align*}
Here we used computed the dilogarithm term by \cref{def:funct_analysis} and \cref{lem:slant_diagonal}. Similarly, the right-hand side is computed as 
\begin{align*}
    \widetilde{\opS}_{wv}\opA_w \Psi_{\bp}(\dbra{\opU_w\opU_v^{-1}\opP_v}) \opA_v \ket{k_v,k_w} &= \widetilde{\opS}_{wv}\opA_w \Psi_{\bp}(\dbra{\opU_w\opU_v^{-1}\opP_v}) \ket{\mbb{k_v},k_w} \\
    &= \zeta_A\Psi_{\bp}(q^{2(k_v+k_w)})\widetilde{\opS}_{wv}\opA_v  \ket{\mbb{k_v},k_w} \\
    &= \zeta_A\Psi_{\bp}(q^{2(k_v+k_w)})\widetilde{\opS}_{wv}\ket{\mbb{k_v},\mbb{k_w}}.
\end{align*}
Then by the symmetry $\widetilde{\opS}_{vw}=\widetilde{\opS}_{wv}$, we get the desired relation. 

\subsubsection{Proof of $\opT_{vw}\opA_v \opT_{wv} = \zeta \opA_v \opA_w \opP_{(12)}$}

Let $\bp:=\bp_{\alpha}^\tri$, where $\alpha \in e_{\interior}(\tri)$ is the edge shared by $v,w$ in the initial triangulation. 
We can rewrite the asserted relation as 
\begin{align*}
    \opS_{vw} \Psi_\bp(\dbra{\opU_v\opU_w^{-1}\opP_w}) \opA_v \Psi_{\bp^{-1}}(\dbra{\opP_w^{-1}\opU_w\opP_v})\opS_{wv} = \zeta \opA_v\opA_w \opP_{(vw)},
\end{align*}
or equivalently
\begin{align}\label{eq:charge_proof}
    \Psi_\bp(\dbra{\opU_v\opU_w^{-1}\opP_w}) \opA_v \Psi_{\bp^{-1}}(\dbra{\opP_w^{-1}\opU_w\opP_v}) = \zeta \opS_{vw} ^{-1}\opA_v\opA_w \opP_{(vw)}\opS_{wv}^{-1}.
\end{align}
We compute the matrix coefficients $\bra{m_v,\mbb{m_w}}\cdots \ket{\mb{n_v},\mbb{n_w}}$ of both sides of \eqref{eq:charge_proof}. For the left-hand side, 
\begin{align}
    &\bra{m_v,\mbb{m_w}} \Psi_\bp(\dbra{\opU_v\opU_w^{-1}\opP_w}) \opA_v \Psi_{\bp^{-1}}(\dbra{\opP_w^{-1}\opU_w\opP_v})\ket{\mb{n_v},\mbb{n_w}} \nonumber\\
    &=\Psi_\bp(q^{2(m_v+m_w)})\bra{m_v,\mbb{m_w}}\opA_v\ket{\mb{n_v},\mbb{n_w}} \Psi_{\bp^{-1}}(q^{2(n_v-n_w)}) \nonumber\\
    &= N\zeta_A \gamma(n_v)^{-1}\braket{m_v,\mbb{m_w}}{-n_v,\mbb{n_w}} \Psi_\bp(q^{2(m_v+m_w)})\Psi_{\bp^{-1}}(q^{2(n_v-n_w)})\nonumber\\
    &= \delta_{m_v+n_v,0}\delta_{m_w,n_w}N^2\zeta_A \gamma(n_v)^{-1}\Psi_\bp(q^{2(m_v+m_w)})\Psi_{\bp^{-1}}(q^{-2(m_v+m_w)})\nonumber\\
    &= \delta_{m_v+n_v,0}\delta_{m_w,n_w}N^2\zeta_A \gamma(n_v)^{-1}\zeta_\inv\gamma(m_v+m_w) \nonumber\\
    &= \delta_{m_v+n_v,0}\delta_{m_w,n_w}N^2\zeta_A\zeta_\inv \gamma(m_w)q^{2m_vm_w}. \label{eq:central_LHS}
\end{align}
Here we have used the inversion relation (\cref{thm:inv}(3)) in the fourth equality. 

For the right-hand side, we first compute
\begin{align*}
    &\opS_{wv}^{-1}\ket{\mb{n_v},\mbb{n_w}} = \sum_{j,k \in \bZ_N} q^{-2jn_v}q^{-2kn_w} \gamma(k)^{-1}\opS_{wv}^{-1}\ket{j,k} = \sum_{j,k \in \bZ_N} q^{-2jn_v}q^{-2kn_w} \gamma(k)^{-1}\ket{j,j+k}, \\
    &\bra{m_v,\mbb{m_w}}\opS_{vw}^{-1} = \sum_{i \in \bZ_N} \gamma(i) q^{2i m_w} \bra{m_v,i} \opS_{vw}^{-1} =\sum_{i \in \bZ_N} \gamma(i) q^{2i m_w} \bra{m_v-i,i}. 
\end{align*}
Then 
\begin{align}
    &\zeta \bra{m_v,\mbb{m_w}}\opS_{vw} ^{-1}\opA_v\opA_w \opP_{(vw)}\opS_{wv}^{-1}\ket{\mb{n_v},\mbb{n_w}} \nonumber\\
    &= \zeta\sum_{i,j,k \in \bZ_N} \gamma(i) q^{2i m_w} q^{-2jn_v}q^{-2kn_w} \gamma(k)^{-1}\bra{m_v-i,i}\opA_v\opA_w \opP_{(vw)}\ket{j,j+k}\nonumber \\
    &= \zeta\sum_{i,j,k \in \bZ_N} \gamma(i) q^{2i m_w} q^{-2jn_v}q^{-2kn_w} \gamma(k)^{-1}\zeta_A^2 \braket{m_v-i,i}{\mbb{j+k},\mbb{j}} \nonumber\\
    &= \zeta\zeta_A^2\sum_{i,j,k \in \bZ_N} \gamma(i) q^{2i m_w} q^{-2\textcolor{red}{j}n_v}q^{-2kn_w} \gamma(k)^{-1} \gamma(m_v-i)^{-1}q^{-2(m_v-i)(\textcolor{red}{j}+k)} \gamma(i)^{-1}q^{-2i\textcolor{red}{j}}\nonumber \\
    &= \zeta\zeta_A^2 N\delta_{m_v+n_v,0} \sum_{i,k \in \bZ_N} q^{2i m_w} q^{-2\textcolor{red}{k}n_w} \gamma(\textcolor{red}{k})^{-1}\gamma(m_v-i)^{-1} q^{-2(m_v-i)\textcolor{red}{k}} \nonumber\\
    &= \zeta\zeta_A^2 \zeta_- N\delta_{m_v+n_v,0} \sum_{i \in \bZ_N} q^{2i m_w} \gamma(m_v+n_w-i)\gamma(m_v-i)^{-1}\nonumber\\
    &= \zeta\zeta_A^2 \zeta_- N\delta_{m_v+n_v,0} \sum_{i \in \bZ_N} q^{2i m_w} \gamma(n_w)q^{2(m_v-i)n_w}\nonumber\\
    &= \zeta\zeta_A^2 \zeta_- N^2\delta_{m_v+n_v,0} \gamma(n_w) \delta_{m_w,n_w}q^{2m_vn_w}. \label{eq:central_RHS}
\end{align}
By comparing \eqref{eq:central_RHS} with \eqref{eq:central_LHS}, we see that $\opT_{vw}\opA_v \opT_{wv} = \zeta \opA_v \opA_w \opP_{(12)}$ holds if and only if 
$\zeta_A\zeta_{\inv} = \zeta \zeta_A^2 \zeta_-$, namely $\zeta = \zeta_\inv(\zeta_A \zeta_-)^{-1}$.

\subsection{Matrix coefficients of the Dehn twist operators}\label{subsec:proof_Dehn}

\subsubsection{Local intertwiners: Proof of \cref{prop:Dehn_twist}}\label{subsubsec:proof_Dehn}
(1): Since $\bra{\mbb{m},\mb{n}}$ is a $q^{2(n-m)}$-eigenvector for $\Psi_{\bp_1}([\opP_v^{-1}\opU_v\opP_w])$ and $\opP_{(vw)}$ permutes the components, we get
\begin{align*}
    \bra{\mbb{m},\mb{n}} \opV_{T_a}^{\bD}(\bbp_a) \ket{\mbb{k},\mb{\ell}} &= \Psi_{\bp_1}(q^{2(n-m)}) \bra{\mbb{m},\mb{n}}\opS_{vw}\opA_w^{-1} \ket{\mb{\ell},\mbb{k}} \\
    &= \zeta_A^{-1}\Psi_{\bp_1}(q^{2(n-m)}) \bra{\mbb{m},\mb{n}}\opS_{vw} \ket{\mb{\ell},k} \\
    &= \zeta_A^{-1}\zeta_+\Psi_{\bp_1}(q^{2(n-m)}) \gamma(m-\ell)^{-1}q^{2(n-\ell)k}.
\end{align*}
Here we have computed the matrix coefficient of $\opS_{vw}$ as
\begin{align*}
    \bra{\mbb{m},\mb{n}}\opS_{vw} \ket{\mb{\ell},k} &= \sum_{p \in \bZ_N} q^{-2\ell p}\bra{\mbb{m},\mb{n}}\opS_{vw} \ket{p,k} \\
    &= \sum_{p \in \bZ_N} q^{-2\ell p}\braket{\mbb{m},\mb{n}}{p-k,k} \\
    &= \sum_{p \in \bZ_N} q^{-2\ell p}\gamma(p-k)q^{2m(p-k)}q^{2nk} \\
    &= \sum_{p \in \bZ_N} q^{-2\ell (p+k)}\gamma(p)q^{2mp}q^{2nk} \\
    &= \zeta_+ \gamma(m-\ell)^{-1}q^{2(n-\ell)k}
\end{align*}
by using \cref{lem:gamma_Fourier}. With a notice that $\zeta_A^{-1}\zeta_+=N^{1/2}i^{(N-1)/6}\cdot N^{1/2}i^{(1-N)/2}=N i^{(1-N)/3}$, we obtain the asserted formula. 

The local quantum trace is computed as\footnote{Here one should notice that when we represent a linear map $f: V_{\bD} \to V_{\bD}$ in the momentum basis as $f\ket{\mb{\ell}}=\sum_{n \in \bZ_N} f_{n\ell} \ket{\mb{n}}$, we get $\bra{\mb{n}}f \ket{\mb{\ell}}=N f_{n\ell}$ by our normalization \eqref{eq:momentum_normalization}. We similarly get another $N$ by \eqref{eq:slant_normalization}.} 
\begin{align*}
    Z_N(T_a;\bbp_a) &= \frac{1}{N^2}\left|\sum_{k,\ell \in \bZ_N} \bra{\mbb{k},\mb{\ell}} \opV_{T_a}^{\bD}(\bbp_a) \ket{\mbb{k},\mb{\ell}} \right|\\
    &= \frac 1 N\left|\sum_{k,\ell \in \bZ_N}\Psi_{\bp_1}(q^{2(\ell-k)}) \gamma(k-\ell)^{-1} \right|\\
    &= \left|\sum_{u \in \bZ_N}\Psi_{\bp_1}(q^{2u}) \gamma(u)^{-1}\right|
\end{align*}
by changing the variable $\ell \mapsto u:=\ell-k$ and summing $k$. 

(2): Similarly, 
\begin{align*}
    \bra{\mbb{m},\mb{n}} \opV_{T_b^{-1}}^{\bD}(\bbp_b) \ket{\mbb{k},\mb{\ell}} &= \Psi_{\bp_1}(q^{2(n-m)})\bra{\mbb{m},\mb{n}} \opS_{vw} \ket{\mbb{k},\mb{\ell}} = N \Psi_{\bp_1}(q^{2(n-m)})\gamma(k-m)q^{2(n-m-\ell)(m-k)}.
\end{align*}
Here,
\begin{align*}
    \bra{\mbb{m},\mb{n}} \opS_{vw} \ket{\mbb{k},\mb{\ell}} &=\sum_{p,r \in \bZ_N} \gamma(p)^{-1}q^{-2kp}q^{-2\ell r}\bra{\mbb{m},\mb{n}} \opS_{vw} \ket{p,r} \\
    &= \sum_{p,r \in \bZ_N}\gamma(p)^{-1}q^{-2kp}q^{-2\ell r} \braket{\mbb{m},\mb{n}}{p-r,r} \\
    &= \sum_{p,r \in \bZ_N}\gamma(p)^{-1}q^{-2kp}q^{-2\ell r} \gamma(p-r)q^{2m(p-r)}q^{2nr} \\
    &= \sum_{p,r \in \bZ_N}q^{-2kp}q^{-2\ell r} \gamma(r)q^{-2pr}q^{2m(p-r)}q^{2nr} \\
    &= \sum_{r \in \bZ_N}N\delta_{-k-r+m,0} q^{-2\ell r} \gamma(r)q^{-2mr}q^{2nr} \\
    &=N \gamma(k-m)q^{2(n-m-\ell)(m-k)}.
\end{align*}
Then
\begin{align*}
    Z_N(T_b^{-1};\bbp_b) &= \frac 1 N\left|\sum_{k,\ell \in \bZ_N}\Psi_{\bp_1}(q^{2(\ell-k)}) \right|= \left| \sum_{u \in \bZ_N}\Psi_{\bp_1}(q^{2u})\right|.
\end{align*}

\subsubsection{The irreducible intertwiner for $T_a$}\label{subsec:reduced_Dehn_a}
We give a proof of \cref{prop:reduced_Dehn_a}. 
Recall that $\ket{a_{\lambda,m}}=\ket{m,\mbb{m-\lambda}}$ and $\opV_{T_a}^{\bD}(\bbp_a)=\opT_{vw}\opA_w^{-1}\opP_{(vw)}$. 
We compute
\begin{align*}
    \bra{a_{\lambda,k}} \overline{V}_{T_a}^{\bD}(\bbp_a;L_a,\lambda) \ket{a_{\lambda,\ell}} &=\frac{1}{N^2}\sum_{m,n \in \bZ_N} \braket{k,\mbb{k-\lambda}}{\mbb{m},\mb{n}} \bra{\mbb{m},\mb{n}} \Psi_{\bp_1}([\opP_v^{-1}\opU_v \opP_w]) \opS_{vw}\opA_w^{-1}\opP_{(vw)}\ket{\ell,\mbb{\ell-\lambda}}\\
    &=\frac{1}{N^2}\sum_{m,n \in \bZ_N}\Psi_{\bp_1}(q^{2(n-m)})  \braket{k,\mbb{k-\lambda}}{\mbb{m},\mb{n}} \bra{\mbb{m},\mb{n}} \opS_{vw}\opA_w^{-1}\ket{\mbb{\ell-\lambda},\ell} \\
    &=\frac{1}{N^3} \zeta_A^{-1}\gamma(\ell)\sum_{m,n \in \bZ_N}\Psi_{\bp_1}(q^{2(n-m)})  \braket{k,\mbb{k-\lambda}}{\mbb{m},\mb{n}} \bra{\mbb{m},\mb{n}} \opS_{vw}\ket{\mbb{\ell-\lambda},\mb{-\ell}}.
\end{align*}
We have $\braket{k,\mbb{k-\lambda}}{\mbb{m},\mb{n}} = \gamma(k)^{-1}q^{-2km} \zeta_+ \gamma(k-\lambda-n)^{-1}$ by using \cref{lem:mbb-mb}, and the matrix coefficients of $\opS_{vw}$ are computed as
\begin{align*}
    \bra{\mbb{m},\mb{n}} \opS_{vw}\ket{\mbb{\ell-\lambda},\mb{-\ell}} &= \sum_{r,p \in \bZ_N} \gamma(r)^{-1}q^{-2r(\ell-\lambda)}q^{2p\ell} \bra{\mbb{m},\mb{n}} \opS_{vw} \ket{r,p} \\
    &= \sum_{r,p \in \bZ_N} \gamma(r)^{-1}q^{-2r(\ell-\lambda)}q^{2p\ell} \braket{\mbb{m},\mb{n}}{r-p,p} \\
    &= \sum_{r,p \in \bZ_N} \textcolor{red}{\gamma(r)}^{-1}q^{-2r(\ell-\lambda)}q^{2p\ell} \textcolor{red}{\gamma(r-p)}q^{2m(r-p)}q^{2np} \\
    &= \sum_{r,p \in \bZ_N} q^{-2r(\ell-\lambda)}q^{2p\ell} \gamma(p)q^{-2rp}q^{2m(r-p)}q^{2np} \\
    &= \sum_{p \in \bZ_N}N\delta_{-\ell+\lambda-p+m,0}q^{2p\ell}\gamma(p)q^{-2mp}q^{2np} \\
    &= N\gamma(m-\ell+\lambda)q^{2(\ell-m+n)(m-\ell+\lambda)}
\end{align*}
by \cref{lem:S_action}. Hence
\begin{align*}
    \bra{a_{\lambda,k}} \overline{V}_{T_a}^{\bD}(\bbp_a;L_a,\lambda) \ket{a_{\lambda,\ell}} 
    &=\frac{1}{N^2}\zeta_A^{-1}\zeta_+\frac{\gamma(\ell)}{\gamma(k)}\sum_{m,n \in \bZ_N}\Psi_{\bp_1}(q^{2(n-m)}) q^{-2km} \frac{\gamma(\ell-\lambda-m)}{\gamma(k-\lambda-n)}q^{2(\ell-m+n)(m-\ell+\lambda)} \\
    &=\frac{1}{N}i^{(1-N)/3}\frac{\gamma(\ell)}{\gamma(k)}\sum_{m,u \in \bZ_N}\Psi_{\bp_1}(q^{2u}) q^{-2km} \frac{\gamma(\ell-\lambda-m)}{\gamma(k-\lambda-m-u)}q^{2(\ell+u)(m-\ell+\lambda)} \\
    &=\frac{1}{N}i^{(1-N)/3}\frac{\gamma(\ell)}{\gamma(k)}\sum_{\textcolor{red}{m},u \in \bZ_N}\Psi_{\bp_1}(q^{2u}) q^{-2k\textcolor{red}{m}} \frac{\gamma(\ell-\lambda)}{\gamma(k-\lambda-u)} q^{-2\textcolor{red}{m}(\ell-k+u)} q^{2(\ell+u)(\textcolor{red}{m}-\ell+\lambda)} \\
    &= i^{(1-N)/3}\frac{\gamma(\ell)}{\gamma(k)}\frac{\gamma(\ell-\lambda)}{\gamma(k-\lambda)} \sum_{u \in \bZ_N}\Psi_{\bp_1}(q^{2u}) \gamma(u)^{-1}q^{2u(k-\lambda)} q^{2(\ell+u)(-\ell+\lambda)} \\
    &= i^{(1-N)/3}\frac{\gamma(\ell)}{\gamma(k)}\frac{\gamma(\ell-\lambda)}{\gamma(k-\lambda)} q^{2\ell(-\ell+\lambda)} \sum_{u \in \bZ_N}\Psi_{\bp_1}(q^{2u}) \gamma(u)^{-1}q^{2u(k-\ell)}. 
\end{align*}
Therefore, the irreducible quantum trace is
\begin{align*}
   \overline{Z}_N (T_a;\bbp_a;L_a,\lambda)= \frac 1 N \left|\sum_{k\in \bZ_N} q^{2k(-k+\lambda)} \sum_{u \in \bZ_N} \Psi_{\bp_1}(q^{2u})\gamma(u)^{-1} \right| = \frac{1}{N^{1/2}} \left| \sum_{u \in \bZ_N} \Psi_{\bp_1}(q^{2u})\gamma(u)^{-1} \right|.
\end{align*}
Here we used \cref{lem:Gauss_2} for the summation on $k$.  

\subsubsection{The irreducible intertwiner for $T_b^{-1}$}\label{subsec:reduced_Dehn_b}
We give a proof of \cref{prop:Dehn_b}.
We use the following, which is proved easily:
\begin{lem}\label{lem:basis_a+b}
Given $\nu \in \bZ_N$, the solution space of the equation $\oph_{a+b}\ket{\psi}=q^{2\nu}\ket{\psi}$ is $N$-dimensional, and spanned by
\begin{align*}
    \ket{(a+b)_{\nu,r}} := \sum_{m \in \bZ_N}\gamma(m)^{-1}q^{2(2r-\lambda)m}\ket{m,r-m}.
\end{align*}
They are  normalized so that $\braket{(a+b)_{\nu,r}}{(a+b)_{r',\nu'}}=N\delta_{rr'}\delta_{\nu\nu'}$. We have the base-change matrix
\begin{align*}
    \braket{b_{\mu,p}}{(a+b)_{\nu,r}} = \gamma(p)q^{2(-2r+\nu-\mu)p} \gamma(r)^3 q^{2(\mu-\nu) r}.
\end{align*}
\end{lem}

Now we are going to compute the matrix coefficients of $\opV_{T_b^{-1}}^{\bD}(\bbp_b;L_a,\lambda)=\gamma(\oph_b)^{1/2}\opT_{vw}$. 

\paragraph{\textbf{(1) Matrix coefficients of $\opT_{vw}$}}
We first prove 
\begin{align}\label{eq:coeff_a+b_a}
    \bra{(a+b)_{\nu,r}} \opT_{vw} \ket{a_{\lambda,\ell}} = N \delta_{r,\ell} \Psi_{\bp_1}(q^{2(2\ell-\lambda)}) \gamma(\ell)^{-3}q^{2\lambda \ell}. 
\end{align}
We factorize the left-hand side as
\begin{align*}
    \bra{(a+b)_{\nu,r}} \opT_{vw} \ket{a_{\lambda,\ell}} &=\frac{1}{N^2} \sum_{m,n} \braket{(a+b)_{\nu,r}}{\mbb{m},\mb{n}}\bra{\mbb{m},\mb{n}}\Psi_{\bp_1} ([\opP_v^{-1}\opU_v\opP_w]) \opS_{vw} \ket{a_{\lambda,\ell}} \\
    &=\frac{1}{N^2} \sum_{m,n} \braket{(a+b)_{\nu,r}}{\mbb{m},\mb{n}}\Psi_{\bp_1} (q^{2(n-m)}) \bra{\mbb{m},\mb{n}} \opS_{vw} \ket{a_{\lambda,\ell}}.
\end{align*}
Each of the matrix coefficients are computed as
\begin{align*}
    \braket{(a+b)_{\nu,r}}{\mbb{m},\mb{n}} &= \sum_{s\in \bZ_N}\gamma(s)q^{-2(2r-\lambda)s}\braket{s,r-s}{\mbb{m},\mb{n}} \\
    &= \sum_{s\in \bZ_N}q^{-2(2r-\lambda)s} q^{-2sm}q^{-2(r-s)n} \\
    &= N \delta_{\lambda-2r-m+n,0}q^{-2rn}
\end{align*}
by \cref{lem:basis_a+b}, and 
\begin{align*}
    \bra{\mbb{m},\mb{n}} \opS_{vw} \ket{a_{\lambda,\ell}} &= 
    \bra{\mbb{m},\mb{n}}  \opS_{vw} \ket{\ell,\mbb{\ell-\lambda}} \\
    &=\sum_{r \in \bZ_N}  \bra{\mbb{m},\mb{n}}  \opS_{vw} \ket{\ell,r} \gamma(r)^{-1}q^{-2r(\ell-\lambda)} \\
    &=\sum_{r \in \bZ_N}  \braket{\mbb{m},\mb{n}}{\ell-r,r} \gamma(r)^{-1}q^{-2r(\ell-\lambda)} \\
    &=\sum_{r \in \bZ_N}  \textcolor{red}{\gamma(\ell-r)}q^{2m(\ell-r)}q^{2nr} \textcolor{red}{\gamma(r)^{-1}}q^{-2r(\ell-\lambda)} \\
    &=\gamma(\ell)\sum_{r \in \bZ_N}  q^{-2\ell r} q^{2m(\ell-r)}q^{2nr} q^{-2r(\ell-\lambda)} \\
    &=N \delta_{-m+n-2\ell+\lambda,0}\gamma(\ell) q^{2m \ell}.
\end{align*}
The two $\delta$-functions are compatible only when $r=\ell$, and impose the same condition in this case. Then
\begin{align*}
    \bra{(a+b)_{\nu,r}} \opT_{vw} \ket{a_{\lambda,\ell}} &= \delta_{r,\ell}\sum_{m,n \in \bZ_N} \delta_{\lambda-2r-m+n,0}q^{-2rn} \Psi_{\bp_1} (q^{2(n-m)})\gamma(\ell) q^{2m \ell} \\
    &= \delta_{r,\ell}\sum_{n \in \bZ_N} q^{-2\ell n} \Psi_{\bp_1} (q^{2(2\ell-\lambda)})\gamma(\ell) q^{2(n-2\ell+\lambda) \ell} \\
    &= \delta_{r,\ell}\sum_{n \in \bZ_N} \Psi_{\bp_1} (q^{2(2\ell-\lambda)}) q^{-3\ell^2+2\lambda\ell} \\
    &=N\delta_{r,\ell} \Psi_{\bp_1} (q^{2(2\ell-\lambda)}) \gamma(\ell)^{-3}q^{2\lambda\ell},
\end{align*}
as is claimed. 

\paragraph{\textbf{(2) Matrix coefficients of $\gamma(\oph_b)^{1/2}$}}
We prove
\begin{align}\label{eq:coeff_a_a+b}
    \bra{a_{\lambda,k}} \gamma(\oph_b)^{1/2}\ket{(a+b)_{\nu,r}} = \zeta^{(2)}_+ \delta_{\lambda,\nu}\gamma(r)^3 \gamma(2r-2k+\lambda-\nu)^{-M}q^{-2\nu r}.
\end{align}
Indeed, by inserting the diagonal basis for $\oph_b$, we compute
\begin{align*}
    \bra{a_{\lambda,k}} \gamma(\oph_b)^{1/2}\ket{(a+b)_{\nu,r}} &= 
    \frac{1}{N} \sum_{\mu,p \in \bZ_N} \braket{a_{\lambda,k}}{b_{\mu,p}}\bra{b_{\mu,p}}\gamma(\oph_b)^{1/2}\ket{(a+b)_{\nu,r}} \\
    &= 
    \frac{1}{N} \sum_{\mu,p \in \bZ_N} \braket{a_{\lambda,k}}{b_{\mu,p}}\gamma(\mu)^{M}\braket{b_{\mu,p}}{(a+b)_{\nu,r}} \\
    &= 
    \frac{1}{N} \sum_{\mu,p \in \bZ_N} \gamma(p)q^{4kp-2\lambda p-2\mu k}\cdot \gamma(\mu)^{M}\cdot\gamma(p)q^{2(-2r+\nu-\mu)p} \gamma(r)^3 q^{2(\mu-\nu) r} \\
    &= 
    \frac{1}{N} \gamma(r)^3\sum_{\mu \in \bZ_N} \gamma(\mu)^{M}q^{2(\mu-\nu) r-2\mu k} \sum_{p \in \bZ_N} \gamma(p)^2q^{2(2k-2r+\nu-\mu-\lambda)p}.
\end{align*}
Here we used \cref{prop:torus_basis,lem:basis_a+b}. Then by using \cref{lem:Gauss_2} for the summation on $p$, we get
\begin{align*}
    \bra{a_{\lambda,k}} \gamma(\oph_b)^{1/2}\ket{(a+b)_{\nu,r}} &=
    \zeta^{(2)}_+\frac{1}{N} \gamma(r)^3\sum_{\mu \in \bZ_N} \gamma(\mu)^{M}q^{2(\mu-\nu) r-2\mu k} \gamma(2k-2r+\nu-\mu-\lambda)^{-M} \\
    &=
    \zeta^{(2)}_+\frac{1}{N} \gamma(r)^3\sum_{\mu \in \bZ_N} q^{2(\mu-\nu) r-2\mu k}\gamma(2r-2k+\lambda-\nu)^{-M} q^{-(2r-2k+\lambda-\nu)\mu} \\
    &=
    \zeta^{(2)}_+\frac{1}{N} \gamma(r)^3\gamma(2r-2k+\lambda-\nu)^{-M}q^{-2\nu r} \sum_{\mu \in \bZ_N} q^{-(\lambda-\nu)\mu} \\
    &=
    \zeta^{(2)}_+\delta_{\lambda,\nu} \gamma(r)^3\gamma(2r-2k+\lambda-\nu)^{-M}q^{-2\nu r},
\end{align*}
as is claimed. 

Finally, by combining \eqref{eq:coeff_a+b_a} and \eqref{eq:coeff_a_a+b}, we obtain

\begin{align*}
    \bra{a_{\lambda,k}} \overline{V}_{T_b^{-1}}^{\bD}(\bbp_b;L_a,\lambda) \ket{a_{\lambda,\ell}} &= 
    \frac{1}{N} \sum_{\nu,r\in \bZ_N} \bra{a_{\lambda,k}} \gamma(\oph_b)^{1/2}\ket{(a+b)_{\nu,r}} \bra{(a+b)_{\nu,r}} \opT_{vw} \ket{a_{\lambda,\ell}} \\
    &= \zeta^{(2)}_+\sum_{\nu,r\in \bZ_N} \delta_{\lambda,\nu} \gamma(r)^3\gamma(2r-2k+\lambda-\nu)^{-M}q^{-2\nu r} \cdot \delta_{r,\ell} \Psi_{\bp_1} (q^{2(2\ell-\lambda)}) \gamma(\ell)^{-3}q^{2\lambda\ell} \\
    &= \zeta^{(2)}_+  \Psi_{\bp_1} (q^{2(2\ell-\lambda)}) \gamma(2\ell-2k)^{-M} \\
    &= \zeta^{(2)}_+  \Psi_{\bp_1} (q^{2(2\ell-\lambda)}) \gamma(\ell-k)^{-2}. 
\end{align*}
Taking its trace, we obtain
\begin{align*}
    \overline{Z}_N(T_b^{-1};\bbp_b;L_a,\lambda) &= \frac{1}{N} \left|\sum_{k \in \bZ_N} \bra{a_{\lambda,k}} \overline{V}_{T_b^{-1}}^{\bD}(\bbp_b;L_a,\lambda) \ket{a_{\lambda,k}}\right| \\
    &= \frac{1}{N^{1/2}} \left|\sum_{k \in \bZ_N} \Psi_{\bp_1} (q^{2(2k-\lambda)})\right|.
\end{align*}

%% file: 8_Root.tex
\appendix

\section{Mapping torus over a marked surface}\label{sec:mapping_torus}
Recall from \cref{subsec:surface} that a marked surface is a compact oriented surface $\Sigma$ equipped with a fixed non-empty finite set $\bM \subset \Sigma$ of marked points. For each boundary component $b$ of $\Sigma$, let $\bM_b:=\bM \cap b$. Let $\Sigma^\ast:=\Sigma \setminus \bM$. We also use the notation $\bM_\circ:=\bM \setminus \bM_\partial$, so that
\begin{align*}
    \bM = \bM_\circ \sqcup \bM_\partial = \bM_\circ \sqcup \bigsqcup_b \bM_b.
\end{align*}
Recall from \cref{subsec:trace} that a mapping class $\phi \in MC(\Sigma)$ is an isotopy class of an orientation-preserving homeomorphism of $\Sigma$ that preserves the subsets $\bM$ and $\partial\Sigma$. We consider the mapping torus of $\phi$ in the following sense:

\begin{dfn}[Mapping torus]\label{def:mapping_torus}
Given a mapping class $\phi \in MC(\Sigma)$, let 
\begin{align*}
    M_\phi^\circ:= ([0,1] \times (\Sigma\setminus \bM_\circ))/(0,x) \sim (1,\phi^{-1}(x)). 
\end{align*}
Each $p \in \bM_\circ$ gives rise to a torus boundary $T_p$ of $M_\phi^\circ$. 
\begin{itemize}
    \item If $\partial\Sigma=\emptyset$ (and hence $\bM_\partial=\emptyset$), let $M_\phi:=M_\phi^\circ$. 
    \item Otherwise, each boundary component $b$ of $\Sigma$ gives rise to a torus boundary $T_b$ of $M_\phi^\circ$. The orbit of $\bM_b$ forms a torus link $L_b \subset T_b$ of some slope $p_b/q_b$. Consider the $p_b/q_b$-Dehn filling: attach a solid torus $D^2 \times S^1$ to $T_b$ so that its meridian represents the slope $p_b/q_b$. The resulting space is denoted by $M_\phi$, and called the \emph{mapping torus} of $\phi$.
\end{itemize}

\end{dfn}
For example, if $\Sigma$ is a disk with $n$ marked points on the boundary and $\phi \in MC(\Sigma)$ is the $1/n$ rotation, then $M_\phi=L(n,1)$ is the lens space.

We are going to investigate the relationship between local systems on $\Sigma$ and $M_\phi$. We focus on the case $G=PGL_2$, while the generalization to semisimple adjoint groups is straightforward. 
Following \cite{GS19}, let $\mathcal{P}_{PGL_2,\Sigma}$ denote the moduli stack of framed $PGL_2$-local on $\Sigma$ with pinnings. It parametrizes the following data:
\begin{itemize}
    \item A $PGL_2$-local system $\mathcal{L}$ (a flat $PGL_2$-bundle) over $\Sigma^\ast$.
    \item For each $p \in \bM_\circ$, a flat section $\beta_p$ of $\mathcal{L}\times_{PGL_2} \mathbb{P}^1$ defined over a small punctured disk centered at $p$.
    \item For each $m \in \bM_\partial$, a flat section $\alpha_m$ of $\mathcal{L}\times_{PGL_2} \widehat{\mathbb{P}}^1$ defined over a small punctured half-disk centered at $m$. Here, $\widehat{\mathbb{P}}^1:=(\bC^2 \setminus \{0\})/\{\pm 1\}$. For each consecutive pair $(m,m')$, the pair $(\alpha_m,\alpha_{m'})$ is required to be generic, namely $\alpha_m \wedge \alpha_{m'} \neq 0$ (after a parallel-transport to a common fiber along the boundary). 
\end{itemize}
The mapping class group $MC(\Sigma)$ naturally acts on $\mathcal{P}_{PGL_2,\Sigma}$. 

\begin{thm}\label{thm:mapping_torus}
For a mapping class $\phi \in MC(\Sigma)$, the $\phi$-fixed points of $\mathcal{P}_{PGL_2,\Sigma}$ lifts to the following data on $M_\phi$:
\begin{itemize}
    \item A $PGL_2$-local system $\widehat{\mathcal{L}}$ on $M_\phi$.
    \item For each $p \in \bM_\circ$, a flat section $\widehat{\beta}_p$ of $\widehat{\mathcal{L}}\times_{PGL_2} \mathbb{P}^1$ along the torus boundary $T_p$.
    \item For each $\phi$-orbit $o$ of $\bM_b$, a flat section $\widehat{\alpha}_o$ of $\widehat{\mathcal{L}}\times_{PGL_2} \widehat{\mathbb{P}}^1$ along the corresponding component of the torus link $L_b$. Each pair of such decorations that are consecutive along the longitude of the filled solid torus is generic. 
\end{itemize}
If $\bM_\partial \neq \emptyset$ or the data is generic, then the correspoindence is one-to-one.  
\end{thm}

\begin{proof}
By the van Kampen theorem, the fundamental group of $M_\phi^\circ$ is generated by $\pi_1(\Sigma^\ast)$ and the fiber loop $t$, subject to the relation
\begin{align}\label{eq:rel_mapping_torus}
    t \ast \gamma \ast t^{-1} = \phi(\gamma), \qquad \gamma \in \pi_1(\Sigma^\ast).
\end{align}
For simplicity, suppose that the $\phi$-action on each boundary component of $\Sigma$ rotates the marked point by one step. Then the slope of $L_b$ is $N_b:=|\bM_b|$. By construction, the fundamental group of $M_\phi$ is obtained from $\pi_1(M_\phi^\circ)$ by adding the relation
\begin{align}\label{eq:rel_filling}
    t^{N_b} \gamma_b= 1, \qquad b \in \pi_0(\partial\Sigma),
\end{align}
where $\gamma_b$ is the loop surrounding the boundary component $b$ going clockwise, since the loop corresponding to the meridian of the solid torus is trivialized. 

Let $[\mathcal{L},\beta,\alpha] \in \P_{PGL_2,\Sigma}$ be a $\phi$-fixed point. Fix a basepoint $x_0 \in \Sigma^\ast$ and a point $s \in \mathcal{L}_{x_0}$. Then we get the monodromy homomorphism $\rho_s: \pi_1(\Sigma^\ast,x_0) \to PGL_2$ corresponding to $\mathcal{L}$. 
For $m \in \bM$, let $\Gamma_m$ denote the set of homotopy classes of paths from $x_0$ to $m$. Let $\Gamma_\Sigma^\circ:=\bigsqcup_{p \in \bM_\circ} \Gamma_p$ and $\Gamma_\Sigma^\partial:=\bigsqcup_{m \in \bM_\partial} \Gamma_m$. The decoration $\beta$ determines $\beta(\gamma) \in \mathbb{P}^1$ for $\gamma \in \bM_\circ$ by parallel-transport and the evaluation $(\mathcal{L}\times_{PGL_2} \mathbb{P}^1)_{x_0} \cong \mathbb{P}^1$ determined by $s$. Similarly, $\alpha$ determines $\alpha(\gamma) \in \widehat{\mathbb{P}}^1$ for $\gamma \in \bM_\partial$. 

Then there exists $g \in PGL_2$ such that
\begin{align}
    \rho(\phi(\gamma)) &= g \rho(\gamma) g^{-1}, \qquad \gamma \in \pi_1(\Sigma^\ast,x_0), \label{eq:fixed_monodromy}\\
    \beta(\phi(\gamma)) &= g. \beta(\gamma), \qquad \gamma \in \Gamma_\Sigma^\circ,\label{eq:fixed_beta} \\
    \alpha(\phi(\gamma)) &= g. \alpha(\gamma), \qquad \gamma \in \Gamma_\Sigma^\partial. \label{eq:fixed_alpha}
\end{align}
Such $g$ is unique if the stabilizer of the triple $(\rho,\beta,\alpha)$ is trivial. This condition is satisfied if $\bM_\partial \neq \emptyset$, by the genericity condition on $\alpha$. Extend $\rho_s$ to
\begin{align*}
    \rho^\circ_s: \pi_1(M_\phi^\circ) \to PGL_2
\end{align*}
by setting $\rho^\circ_s(t):=g$. Then the condition \eqref{eq:fixed_monodromy} implies the relation \eqref{eq:rel_mapping_torus}, hence $\rho^\circ_s$ is well-defined. For a small corner arc at $m \in \bM_\partial$ (cf. \cref{fig:boundary_sum}), a normalized parallel-transport $\rho_m \in G$ is determined. See \cite[Lemma-Definition 3.11]{IOS}.\footnote{$\rho_m$ is called the twisted Wilson line \emph{loc. cit.} We do not need to care about framings, since we are working on the adjoint group. See also \cite[Remark 3.10]{IOS}.} If $m_1,\dots,m_{N_b}$ are the marked points on $b$ in this counter-clockwise order, then these elements satisfy $\rho_{m_1} \rho_{m_2}\dots \rho_{m_{N_b}} = \rho_s(\gamma_b)^{-1}$. By \eqref{eq:fixed_beta}, \eqref{eq:fixed_alpha} and the definition of $\rho_m$, it follows that $g=\rho_{m_1}=\dots \rho_{m_{N_b}}$.\footnote{Recall from \cite[Section 3.2]{IOS}, the pinnings (local trivializations of $\mathcal{L}$) are determined by the pairs $p_E:=(\alpha_m,\pi(\alpha_{m'})) \in \widehat{\mathbb{P}}^1 \times  \mathbb{P}^1$, where $\pi: \widehat{\mathbb{P}}^1 \to \mathbb{P}^1$ denotes the canonical projection; $m'$ is counter-clockwise to $m$, and $E$ is the boundary interval bounded by $m$ and $m'$. Let $E_i$ denote the boundary interval on $b$ between $m_i$ and $m_{i+1}$ ($i$ mod $N_b$). Writing $p_{E_i}=g_i.p_\mathrm{std}$ for a unique $g_i \in PGL_2$, we have $\rho_{m_i}=g_{i-1}^{-1} g_i$. The $\phi$-fixed condition implies $\rho_{m_i}=\rho_{m_{i+1}}$ for $i$ mod $N_b$.} 
Therefore, the relation \eqref{eq:rel_filling} is satisfied and $\rho_s^\circ$ descends to a group homomorphism $\widehat{\rho}_s: \pi_1(M_\phi) \to PGL_2$. It determines a $PGL_2$-local system $\widehat{\mathcal{L}}$ on $M_\phi$. The correspondence of the remaining data should be clear. 

Conversely, given such data on $M_\phi$, a point of $\P_{PGL_2,\Sigma}^\phi$ is obtained via the restriction to $\Sigma^\ast=\{0\} \times \Sigma^\ast \hookrightarrow M_\phi$.
\end{proof}

\begin{rem}\label{rem:mapping_torus_hyperbolic}
By the Hyperbolic Dehn Surgery Theorem \cite{Thurston}, the mapping torus $M_\phi$ has a complete hyperbolic structure if $M_\phi^\circ$ has a complete hyperbolic structure and the slopes for the surgery avoid finitely many exceptions. Note that $M_\phi^\circ$ is hyperbolic if and only if $\phi$ is pseudo-Anosov. 
\end{rem}

\section{The case of \texorpdfstring{$q^N=-1$}{qN=-1}}\label{sec:root}
In this section, we briefly discuss how the theory changes if we employ the choice $q=e^{\pi i/N}$ for an arbitrary positive integer $N$, so that $q^N=-1$, $q^{2N}=1$. Most part of the theory does not change drastically, but gets extra signs in some formulas. 
As pointed out by Bonahon--Liu \cite{BL}, the representation theory of quantum tori best works for the cases $q^N=(-1)^{N+1}$ (see also \cite{Bai,BWY-1,BWY-2}).

For distinction, let us write $\zeta:=e^{\pi i/N}$. When $N$ is odd, it is related to the previous choice by $\zeta=-q$. Basically, the modification goes with this substitution. 

\smallskip 

\paragraph{\textbf{Cyclic quantum dilogarithm}}
The expression of cyclic quantum dilogarithm is changed into
\begin{align*}
    \Psi_\bp(X):= \prod_{j=0}^{N-1} ( p^- - \zeta^{2j+1} p^+X )^{j/N},
\end{align*}
where $\bp=(p^+,p^-) \in \overline{\cF}_N$. 
The basic formula in \cref{lem:cyclic_prod} changes as follows: 
\begin{lem}
For $\zeta=e^{\pi i/N}$, we have 
\begin{align*}
    \prod_{k=0}^{N-1} (y-q^{2k+1}x)=y^N+x^N, \quad
    \prod_{k=0}^{N-1} (y-q^{2k}x)=y^N-x^N.
\end{align*}
\end{lem}
Accordingly, the $q$-difference relation (\cref{lem:q-diff_eq}) is modified as follows:
\begin{align}\label{eq:zeta_diff_eq}
    \Psi_\bp(\zeta^{2}X) &= (p^- - \zeta p^+X ) \Psi_\bp(X), \\
    \Psi_\bp(\zeta^{-2}X) &= (p^- -\zeta^{-1}p^+ X )^{-1} \Psi_\bp(X)
\end{align}
for $X^N=1$. The inversion relation (\cref{thm:inv}) becomes
\begin{align*}
    \Psi_\bp(\zeta^{2k})\cdot \Psi_{\bp^\op}(\zeta^{-2k})=\gamma(k)\zeta_\inv,
\end{align*}
where $\zeta_\inv$ does not change but 
\begin{align*}
    \gamma(k):=(-1)^k \zeta^{k^2}=\zeta^{k(k-N)}
\end{align*}
The $\gamma$-operator (\cref{def:gamma_op}) is still defined, which has the spectrum $\{\gamma(k)\}_{k \in \bZ_N}$ with this modification. 

\smallskip

\paragraph{\textbf{Pentagon relation and the modified Weyl ordering}}
The pentagon relation (\cref{thm:pentagon}) still holds with the simple replacement $q=-\zeta$. 
In particular, $-\zeta^{-1}\opU\opP$ appears in the right-hand side, which satisfies $(-\zeta^{-1}\opU\opP)^N=\mathsf{1}$. In general, the following modification of the Weyl ordering is useful:

\begin{dfn}[Modified Weyl ordering, cf.~{\cite[Section 8.2]{AK-CS}}]\label{def:modified_Weyl}
For two operators $\opX,\opY$ acting on $V_N$ such that $\opX \opY = \zeta^{2a}\opY \opX$, we define 
\begin{align*}
    \ddbra{\opX\opY} :=(-\zeta)^{-a}\opX\opY = (-\zeta)^a\opY\opX.
\end{align*}
\end{dfn}
It differs from the usual Weyl ordering $[\opX\opY]=\zeta^{-a}\opX\opY$ by a sign. It is designed so that the following holds:

\begin{lem}
If $\opX^N=c_1$ and $\opY^N=c_2$ for some $c_1,c_2 \in \bC^\ast$, then $\ddbra{\opX\opY}^N=c_1c_2$. 
\end{lem}
The diagonalization of $\ddbra{\opP\opU^{-1}}$ (\cref{lem:slant_diagonal}) goes with the same formula with the redefined $\gamma(m)$ above. The formulae in \cref{subsec:Gaussian} remain unchanged. 

The constructions in \cref{sec:Kashaev} are mostly unchanged under the redefinitions above: with the replacements $q=-\zeta$ and $[\cdots] \mapsto \ddbra{\cdots}$. The formulae in \cref{thm:AT_relations} remain unchanged. 

\smallskip 

\paragraph{\textbf{Determinant of the flip operator}}

The determinant of the flip operator is changed into $\det \opT_{vw} = (-1)^{\frac{(N-1)^2N}{2}}$, while the important conclusion $|\det \opT_{vw}|=1$ is the same. 

\smallskip

\paragraph{\textbf{Chekhov--Fock algebras and the quantum Frobenius map}}
According to the sign corrections above, the setting for the Chekhov--Fock theory slightly changes. The basic idea is as follows. In view of \eqref{eq:zeta_diff_eq}, we should consider $\Psi_\bp(-\overline{X}_k)$ as the automorphism part in order to obtain the correct mutation formula. Also observe that this formula is valid only when $(-\overline{X}_k)^N=1$. 

Based on these facts, we consider
\begin{align*}
    \X_{\boldsymbol{y}} := \X/ \langle X_i^N - (-y_i)^N \mid i \in I \rangle
\end{align*}
for a given $\boldsymbol{y}=(y_i)_{i \in I} \in \bC^I$. 
By using $\Psi_\bp(-y_k^{-1}X_k)=\Psi_\bp(-\overline{X}_k)$ as the automorphism part, we again obtain the same mutation formula. 


\paragraph{\textbf{Square-root operators}}
It seems that we have no analog of the square-root operators when $N$ is even. The square-root operators are used in the transvection operator \eqref{eq:transvection} (and hence to introduce the irreducible intertwiner). We also need to take a square-root of $\gamma(\ell)$ in defining the basis $\ket{\phi_\ell}$ in \cref{thm:cyclic_quantum_group}.